\documentclass[a4paper,11pt]{article}
\usepackage{amsfonts}
\usepackage{amsmath}
\usepackage{graphics}
\usepackage{graphicx,psfrag}
\usepackage{mathrsfs}
\usepackage{amssymb}
\usepackage{enumerate}
\usepackage{vmargin}
\usepackage{amsfonts}
\title{}
\author{}

\newcommand{\I}{{\mathchoice {\rm 1\mskip-4mu l} {\rm 1\mskip-4mu l}
{\rm 1\mskip-4.5mu l} {\rm 1\mskip-5mu l}}}

\newcommand{\Mon}{5.1}

\newenvironment{proof}{\smallskip\noindent\emph{Proof}\hspace{1pt}}%
{\hspace{-5pt}{\nobreak\quad\nobreak\hfill\nobreak$\square$\vspace{8pt}%
\par}\smallskip\goodbreak}
\newcommand{\ki}{\ensuremath{\chi_{\alpha, \beta}^{x,y}}}
\newcommand{\oGamma}{\ensuremath{\overline{\Gamma}}}
\newcommand{\oPi}{\ensuremath{\overline{\Pi}}}
\newcommand{\oH}{\ensuremath{\overline{H}}}

\newcommand{\oN}{\ensuremath{\bar{n}}}

\newcommand{\lL}{\ensuremath{\widehat{N}^+}}
\newcommand{\lU}{\ensuremath{\widehat{N}^-}}

\newcommand{\esse}{\ensuremath{\overline{\mathcal{S}}}}
\newcommand{\acca}{\ensuremath{\overline{\mathcal{H}}}}

\newtheorem{teorema}{Theorem}[section]
\newtheorem{definition}[teorema]{Definition}
\newtheorem{proposizione}[teorema]{Proposition}
\newtheorem{corollario}[teorema]{Corollary}
\newtheorem{lemma}[teorema]{Lemma}
\newtheorem{remark}[teorema]{Remark}
\newtheorem{hyp}[teorema]{Hypothesis}

\numberwithin{equation}{section}
\begin{document}
\begin{center}
\textbf{\Large Stochastic order and attractiveness\\
 for particle systems with multiple births,\\
deaths and jumps}\\
\vspace{0.5 cm}
{\sc Davide Borrello}\footnote[1]{
\begin{itemize}\item[] Dipartimento di Matematica e Applicazioni, Universit\`a di Milano Bicocca, 20125 Milano, Italy.\\
Laboratoire de Math\'ematiques Rapha\"el Salem, CNRS UMR 6085, Universit\'e de Rouen, Saint-\'Etienne-du-Rouvray, France.\\
Supported by Fondation Sciences Math\'ematiques de Paris.\\
Supported by ANR grant BLAN07-218426.
\end{itemize}
}\\
d.borrello@campus.unimib.it
\end{center}
\begin{abstract}
An approach to analyse the properties of a particle system is to compare it with different processes to understand when one of them is larger than other ones. The main technique for that is coupling, which may not be easy to construct.\\
We give a characterization of stochastic order between different interacting particle systems in a large class of processes with births, deaths and jumps of many particles per time depending on the configuration in a general way: it consists in checking inequalities involving the transition rates. We construct explicitly the coupling that characterizes the stochastic order. As a corollary we get necessary and sufficient conditions for attractiveness. As an application, we first give the conditions on examples including reaction-diffusion processes, multitype contact process and conservative dynamics and then we improve an ergodicity result for an epidemic model.\\



\noindent Submitted to EJP on Mars 22, 2010, final version accepted on December 18, 2010.

\end{abstract}
\newpage
\section{Introduction}
\label{intr}

The use of interacting particle systems to study biological models is becoming more and more fruitful. In many biological applications a particle represents an individual from a species which interacts with others in many different ways; the empty configuration is often an absorbing state and corresponds to the extinction of that species. An important problem is to find conditions which give either the survival of the species, or the almost sure extinction. When the population in a system is always larger (or smaller) than the number of individuals of another one there is a \textit{stochastic order} between the two processes and one can get information on the larger population starting from the smaller one and vice-versa.

\textit{Attractiveness} is a property concerning the distribution at time $t$ of two processes \textit{with the same generator}: if a process is attractive the stochastic order between two processes starting from different configurations is preserved in the time evolution (see Section $\ref{background}$).

The main technique to check if there is stochastic order between two systems is coupling: if the transitions are intricate an increasing coupling may be hard to construct. The main result of the paper (Theorem \ref{cns}, Section \ref{background}) gives a characterization of the stochastic order (resp. attractiveness) in a large class of interacting particle systems: in order to verify if two particle systems are stochastically ordered (resp. one particle system is attractive), we are reduced to check inequalities involving the transition rates.\\

A first motivation is a general understanding of the ordering conditions between two processes. The analysis of interacting particle systems began with Spin Systems, that are processes with state space $\{0,1\}^{\mathbb{Z}^{d}}$. We refer to \cite{cf:Liggett} and \cite{cf:Liggett2} for construction and main results. The most famous examples are Ising model, contact process and voter model. These processes have been largely investigated, in particular their attractiveness (see \cite[Chapter III, Section 2]{cf:Liggett}). Many other models taking place on $X^{\mathbb{Z}^{d}}$, where $X=\{0,1,\ldots M\}\subseteq \mathbb{N}$, that is with more than one particle per site, have been studied. Reaction-diffusion processes, for example, are processes with state space $\mathbb{N}^{\mathbb{Z}^{d}}$ (hence non compact), used to model chemical reactions. We refer to \cite[Chapter $13$]{cf:Chen} for a general introduction and construction. In such particle systems a birth, death or jump of at most one particle per time is allowed. But sometimes the model requires births or deaths of more than one particle per time. This is the case of biological systems with mass extinction (\cite{cf:schi_clusters}, \cite{cf:schi_allee}), or multitype contact process (\cite{cf:durrett_ten}, \cite{cf:durrett_mutual}, \cite{cf:neuhauser_multi}, \cite{cf:schi_drug}). A partial understanding of attractiveness properties can be found in \cite{cf:Stover} for the multitype contact process. A system with jumps of many particles per time has been investigated in \cite[Theorem $2.21$]{cf:GobronSaada}, where the authors found necessary and sufficient conditions for attractiveness for a conservative particle system with multiple jumps on $\mathbb{N}^{\mathbb{Z}^d}$ with misanthrope type rates.

Those examples and the need for more realistic models for metapopulation dynamics systems (\cite{cf:Hanski}) has led us to consider systems ruled by births, deaths and migrations of more than one individual per time with general transition rates to get an exhaustive analysis of the stochastic order behaviour and attractiveness. Our method relies on \cite{cf:GobronSaada}, that it generalizes.\\

The main applications allow to investigate the ergodic properties of a process. A process is ergodic if there exists a unique invariant measure to which it converges starting from any initial configuration: if the process is attractive, it is enough to check the convergence of the largest and the smallest initial configurations. This is a first application of Theorem \ref{cns}. In Section \ref{IRM} we combine attractiveness and a technique called $u-$criterion (see \cite{cf:Chen}) to get ergodicity conditions on a model for spread of epidemics, either if there is a trivial invariant measure or not.\\

For many biological models the empty configuration $\underline{0}$ is an absorbing state and the main question is if the particle system may survive, that is if there is a positive probability that the process does not converge to the Dirac measure $\delta_{\underline{0}}$ concentrated on $\underline{0}$, which is a trivial invariant measure.
In order to prove that a metapopulation dynamics model (see \cite{cf:Borrello}) survives, we largely use comparison (therefore stochastic order) with auxiliary processes: this is a second application of the result. Instead of constructing a different coupling for each comparison, we just check that inequalities of Theorem \ref{cns} are satisfied on the transition rates. Moreover the main technique we use to get survival is a comparison with oriented percolation (see \cite{cf:durrett_ten}), and attractiveness is a key tool in many steps of the proofs.\\ 

The survival of a process does not imply the existence of a non trivial invariant measure: one can have the presence of particles in the system for all times but no invariant measures. If the process is attractive and the state space is compact, a standard approach allows to construct such a measure starting from the largest initial configuration: this is the third application. Once we get survival, we use this argument to construct non trivial invariant measures for metapopulation dynamics models. In Section \ref{meta} we introduce a metapopulation dynamics model with mass migration and Allee effect investigated in \cite{cf:Borrello}.\\

The transition rates of the particle systems we analyse in this paper depend on two sites $x$, $y$, on the number of particles at $x$ and $y$ and on the number of particles $k$ involved in a transition: they are of the form $b(k,\alpha,\beta)p(x,y)$, where $\alpha$ and $\beta$ are respectively the number of particles on $x$ and $y$ and $p(x,y)$ is a probability distribution on $\mathbb{Z}^{d}$ given by a bistochastic matrix (we require neither symmetry nor translation invariance). Moreover we allow birth and death rates on a site $x$ depending only on the configuration state on $x$.

In other words we work with three different types of transition rates: given a configuration $\eta$, on each site $y$ we can have a \textit{birth} (\textit{death}) of $k$ individuals depending on the configuration state $\eta(y)$ on the same site $y$ with rate $P^{k}_{\eta(y)}$ ($P^{-k}_{\eta(y)}$) and depending also on the number of particles on the other sites $x\neq y$ with rate $\sum_{x}R^{0,k}_{\eta(x),\eta(y)}p(x,y)$ ($\sum_{x}R^{-k,0}_{\eta(y),\eta(x)}p(y,x)$). %We use $R^{-k,0}_{\eta(y),\eta(x)}$ instead of $R^{0,-k}_{\eta(x),\eta(y)}$ since this is more consistent with the rest of the paper's notation of putting the loss of particles on the left; moreover 
We consider a death rate $R^{-k,0}_{\eta(y),\eta(x)}p(y,x)$ instead of the more natural $R^{-k,0}_{\eta(y),\eta(x)}p(x,y)$ to simplify the proofs and since we are interested in applications given by a symmetric probability distribution $p(\cdot,\cdot)$. This represents a possible different interaction rule between individuals of the same population and individuals from different populations. We can have a \text{jump} of $k$ particles from $x$ to $y$ with rate $\Gamma^{k}_{\eta(x),\eta(y)}p(x,y)$, which represents a migration of a flock of individuals (see Section $\ref{background}$). We require that the birth/death and jump rates differ only on the term $b(k,\alpha,\beta)$, that is the conservative and non-conservative rates depend on the same probability distribution $p(x,y)$.\\

In Section $\ref{background}$ we recall some classical definitions and propositions needed in the sequel, we introduce the particle system with more details and we state the main result, Theorem \ref{cns}. In Section $\ref{examples}$ we derive the conditions on several examples (multitype contact processes, conservative dynamics and reaction-diffusion processes); we also detail the conditions on models with transitions of at most one particle per time. In Section \ref{IRM} we apply the attractiveness conditions and the so-called $u$-criterion technique to improve an ergodicity result for a model of spread of epidemics. Others applications to the construction of non-trivial invariant measures in metapopulation dynamics models will be presented in a subsequent paper (see \cite{cf:Borrello}). In Section \ref{Proofs} we prove Theorem \ref{cns}: the coupling is constructed explicitly through a downwards recursive formula in Section \ref{coupling}, where a detailed analysis of the coupling mechanisms is presented. We have to mix births, jumps and deaths in a non trivial way by following a preferential direction. Section $\ref{proof_appl}$ is devoted to the proofs needed for the application to the epidemic model. Finally we propose some possible extensions to more general systems.

\section{Main result and applications}

\subsection{Stochastic order an attractiveness}
\label{background}

Denote by $S=\mathbb{Z}^{d}$ the set of sites and let $X \subseteq \mathbb{N}$ be the set of possible states on each site of an interacting particle system $(\eta_{t})_{t\geq 0}$ on the state space $\Omega=X^{S}$, with semi-group $T(t)$ and infinitesimal generator $\mathcal{L}$ given, for a local function $f$, by
\begin{align}
\mathcal{L}f(\eta)=&\sum_{x,y \in S}\sum_{\alpha,\beta \in X}\ki (\eta)p(x,y)\sum_{k > 0}\Big(\Gamma^{k}_{\alpha,\beta} (f(S^{-k,k}_{x,y}\eta)-f(\eta))+ \nonumber \\
&+(R^{0,k}_{\alpha,\beta}+P^{k}_{\beta})(f(S^{k}_{y}\eta)-f(\eta))+(R^{-k,0}_{\alpha,\beta}+P^{-k}_{\alpha})(f(S^{-k}_{x}\eta)-f(\eta))\Big)
\label{gen}
\end{align}
where $\ki$ is the indicator of configurations with values $(\alpha,\beta)$ on $(x,y)$, that is
\begin{equation*}
\ki(\eta)=\left\{
\begin{array}{ll}
1 & \text{if }\eta(x)=\alpha \text{ and } \eta(y)=\beta\\
0 & \text{ otherwise}
\end{array}
\right.
\end{equation*}
and $S^{-k,k}_{x,y}$, $S^{k}_{y}$, $S^{-k}_{y}$, where $k > 0$, are local operators performing the transformations whenever possible
\begin{align}
(S^{-k,k}_{x,y}\eta)(z)=&\left\{
\begin{array}{ll}
\eta(x)-k & \text{ if } z=x \text{ and }\eta(x)-k \in X, \eta(y)+k \in X\\
\eta(y)+k & \text{ if } z=y \text{ and }\eta(x)-k \in X, \eta(y)+k \in X\\
\eta(z) & \text{otherwise;}
\end{array}
\right.
\label{intr:Skk}\\
\nonumber\\
(S^{k}_{y}\eta)(z)=&\left\{
\begin{array}{ll}
\eta(y)+k & \text{ if } z=y \text{ and } \eta(y)+k \in X\\
\eta(z) & \text{otherwise;}
\end{array}
\right.
\label{intr:Sk}\\
\nonumber\\
(S^{-k}_{y}\eta)(z)=&\left\{
\begin{array}{ll}
\eta(y)-k & \text{ if } z=y \text{ and } \eta(y)-k \in X\\
\eta(z) & \text{otherwise.}
\end{array}
\right.
\label{intr:S-k}
\end{align}
The transition rates have the following meaning:
\begin{itemize}
\item[-]$p(x,y)$ is a bistochastic probability distribution on $\mathbb{Z}^{d}$;
\item[-]$\Gamma_{\alpha,\beta}^{k}p(x,y)$ is the jump rate of $k$ particles from $x$, where $\eta(x)=\alpha$, to $y$, where $\eta(y)=\beta$;
\item[-] $R_{\alpha,\beta}^{0,k}p(x,y)$ is the part of the birth rate of $k$ particles in $y$ such that $\eta(y)=\beta$ which depends on the value of $\eta$ in $x$ (that is $\alpha$);
\item[-]$R_{\alpha,\beta}^{-k,0}p(x,y)$ is the part of the death rate of $k$ particles in $x$ such that $\eta(x)=\alpha$ which depends on the value of $\eta$ in $y$ (that is $\beta$); 
\item[-]$P_{\beta}^{\pm k}$ is the birth/death rate of $k$ particles in $\eta(y)=\beta$ which depends only on the value of $\eta$ in $y$: we call it an \emph{independent} birth/death rate.
\end{itemize}
We call \emph{addition} on $y$ (\emph{subtraction} from $x$) of $k$ particles the birth on $y$ (death on $x$) or jump from $x$ to $y$ of $k$ particles. By convention we take births on the right subscript and deaths on the left one: formula ($\ref{gen}$) involves births upon $\beta$, deaths from $\alpha$ and a fixed direction, from $\alpha$ to $\beta$, for jumps of particles.
We define, for notational convenience
\begin{equation}
\Pi^{0,k}_{\alpha,\beta}:=R_{\alpha,\beta}^{0,k}+P_{\beta}^{k};\qquad \Pi^{-k,0}_{\alpha,\beta}:=R_{\alpha,\beta}^{-k,0}+P_{\alpha}^{-k}.
\end{equation}
We refer to $\cite{cf:Liggett}$ for the classical construction in a compact state space. Since we are interested also in non compact cases, we assume that $(\eta_{t})_{t\geq 0}$ is a well defined Markov process on a subset $\Omega_{0} \subset \Omega$, and for any bounded local function $f$ on $\Omega_{0}$,
\begin{equation}
\forall \eta \in \Omega_{0}, \quad \lim_{t \to 0}\frac{T(t)f(\eta)-f(\eta)}{t}=\mathcal{L}f(\eta)< \infty.
\label{defM}
\end{equation}
We will be more precise on the induced conditions on transition rates in the examples. We state here only a common necessary condition on the rates.
\begin{hyp}
We assume that for each $(\alpha,\beta) \in X^{2}$
\begin{align*}
&N(\alpha,\beta):=\sup\{n:\Gamma^{n}_{\alpha,\beta}+\Pi^{0,n}_{\alpha,\beta}+\Pi^{-n,0}_{\alpha,\beta}> 0\}<\infty,\\
&\sum_{k>0}\Big(\Gamma^{k}_{\alpha,\beta}+\Pi^{0,k}_{\alpha,\beta}+\Pi^{-k,0}_{\alpha,\beta}\Big)<\infty.
\end{align*}
\label{hip}
\end{hyp}
In other words, for each $\alpha,\beta$ there exists a maximal number of particles involved in birth, death and jump rates.
Notice that $N(\alpha,\beta)$ is not necessarily equal to $N(\beta,\alpha)$, which involves deaths from $\beta$, births upon $\alpha$ and jumps from $\beta$ to $\alpha$.\\
The particle system admits an \textit{invariant measure} $\mu$ if $\mu$ is such that $P_{\mu}(\eta_{t}\in A)=\mu(A)$ for each $t\geq 0$, $A \subseteq \Omega$, where $P_{\mu}$ is the law of the process starting from the initial distribution $\mu$. An invariant measure is \emph{trivial} if it is concentrated on an absorbing state when there exists any. The process is \textit{ergodic} if there is a unique invariant measure to which the process converges starting from any initial distribution (see \cite[Definition 1.9]{cf:Liggett}).\\
Given two processes $(\xi_{t})_{t \geq 0}$ and $(\zeta_{t})_{t \geq 0}$, a coupled process ($\xi_{t},\zeta_{t})_{t\geq0}$ is Markovian with state space $\Omega_{0} \times \Omega_{0}$, and such that each marginal is a copy of the  original process.\\
We define a partial order on the state space:
\begin{equation}
\forall \xi,\zeta \in \Omega,\hspace{1 cm}\xi \leq \zeta \Leftrightarrow (\forall x \in S, \xi(x) \leq \zeta(x)).
\end{equation}
A set $V \subset \Omega$ is \emph{increasing} (resp. \emph{decreasing}) if for all $\xi \in V$, $\eta \in \Omega$ such that $\xi \leq \eta$ (resp.  $\xi \geq \eta$ ) then $\eta \in V$. For instance if $\xi \in \Omega$, then $I_{\xi}=\{\eta \in \Omega: \xi\leq \eta\}$ is an increasing set, and $D_{\xi}=\{\eta \in \Omega: \xi \geq \eta\}$ a decreasing one.\\
A function $f:\Omega \to \mathbb{R}$ is \emph{monotone} if
$$
\forall \xi,\eta \in \Omega, \hspace{1 cm}\xi \leq \eta \Rightarrow f(\xi)\leq f(\eta).
$$
We denote by $\mathcal{M}$ the set of all bounded, monotone continuous functions on $\Omega$. 
The partial order on $\Omega$ induces a stochastic order on the set $\mathcal{P}$ of probability measures on $\Omega$ endowed with the weak topology:
\begin{equation}
\forall \nu,\nu' \in \mathcal{P}, \hspace{1 cm} \nu \leq \nu'  \Leftrightarrow (\forall f \in \mathcal{M}, \nu(f)\leq \nu'(f)).
\end{equation}
The following theorem is a key result to compare distributions of processes with \textit{different generators} starting with different initial distributions. 
\begin{teorema}
Let $(\xi_{t})_{t \geq 0}$ and $(\zeta_{t})_{t \geq 0}$ be two processes with generators $\widetilde{\mathcal{L}}$ and $\mathcal{L}$ and semi-groups $\widetilde{T}(t)$ and $T(t)$ respectively. The following two statements are equivalent:
\begin{itemize}
\item[(a)]$f \in \mathcal{M}$ and $\xi_{0}\leq \zeta_{0}$ implies $\widetilde{T}(t)f(\xi_{0})\leq T(t)f(\zeta_{0}) $ for all $t \geq 0$.
\item[(b)]For $\nu,\nu' \in \mathcal{P}, \nu \leq \nu'$ implies $\nu \widetilde{T}(t)\leq \nu' T(t)$ for all $t \geq 0$.
\end{itemize}
\label{Genlig1}
\end{teorema}
The proof is a slight modification of \cite[proof of Theorem II.2.2]{cf:Liggett}. 
\begin{definition}
A process $(\zeta_{t})_{t\geq 0}$ is \emph{stochastically larger} than a process $(\xi_{t})_{t\geq 0}$ if the equivalent conditions of Theorem $\ref{Genlig1}$ are satisfied. In this case the process $(\xi_{t})_{t\geq 0}$ is stochastically smaller than $(\zeta_{t})_{t\geq 0}$ and the pair $(\xi_{t},\zeta_{t})_{t\geq 0}$ is stochastically ordered.
\label{stoc_order}
\end{definition}
Attractiveness is a property concerning the distribution at time $t$ of two processes \textit{with the same generator} which start with different initial distributions. By taking $\widetilde{T}=T$, Theorem \ref{Genlig1} reduces to \cite[Theorem II.2.2]{cf:Liggett} and Definition \ref{stoc_order} is equivalent to the definition of an attractive process (see \cite[Definition II.2.3]{cf:Liggett}). If an attractive process in a compact state space starts from the larger initial configuration, it converges to an invariant measure.\\
 
The main result gives necessary and sufficient conditions on the transition rates that yield stochastic order or attractiveness of the class of interacting particle systems defined by (\ref{gen}). It extends \cite[Theorem 2.21]{cf:GobronSaada}.\\
We denote by $\mathcal{S}=\mathcal{S}(\Gamma,R,P,p)=\{\Gamma^{k}_{\alpha,\beta}, R^{0,k}_{\alpha, \beta},R^{-k,0}_{\alpha, \beta}, P^{\pm k}_{\beta},p(x,y)\}_{\{\alpha,\beta \in X, k>0,(x,y)\in S^{2}\}}$, the set of parameters that characterize the generator ($\ref{gen}$) of process $(\eta_{t})_{t\geq 0}$. We call the latter an $\mathcal{S}$ \emph{particle system} and we write $\eta_{t}\sim \mathcal{S}$. Given $(\alpha,\beta),(\gamma,\delta) \in X^2\times X^2$, we write $(\alpha,\beta)\leq (\gamma,\delta)$ if $\alpha\leq \gamma$ and $\beta\leq \delta$.\\
Let $\mathcal{S}=\mathcal{S}(\Gamma,R,P,p)$ and $\widetilde{\mathcal{S}}=\mathcal{S}(\widetilde{\Gamma},\widetilde{R},\widetilde{P},p)$ be two processes with different transition rates (but the same $p$). For all $(\alpha,\beta) \in X^2$, let $\widetilde{N}(\alpha,\beta)$ be defined as $N(\alpha,\beta)$ in Hypothesis \ref{hip} using the transition rates of $\widetilde{\mathcal{S}}$.
%\begin{definition}
%\label{sets}
%\end{definition} 
\begin{teorema}
Given $K\in \mathbb{N}$, $\mathbf{j}:=\{j_{i}\}_{i \leq K}$, $\mathbf{m}:=\{m_{i}\}_{i \leq K}$, $\mathbf{h}:=\{h_{i}\}_{i \leq K}$, three non-decreasing $K$-uples in $\mathbb{N}$, and $\alpha,\beta,\gamma,\delta$ in $X $ such that $\alpha\leq \gamma$, $\beta\leq \delta$, we define
\begin{eqnarray}
I_{a}:=&I^{K}_{a}(\mathbf{j},\mathbf{m})=&\bigcup^{K}_{i=1}\{k\in X:m_{i}\geq  k> \delta-\beta+j_{i}\}
\label{Ia}\\
I_{b}:=&I^{K}_{b}(\mathbf{j},\mathbf{m})=&\bigcup^{K}_{i=1}\{k\in X:\gamma -\alpha +m_{i} \geq k > j_{i}\}
\label{Ib}\\
I_{c}:=&I^{K}_{c}(\mathbf{h},\mathbf{m})=&\bigcup^{K}_{i=1}\{k\in X:m_{i}\geq k > \gamma-\alpha+h_{i}\}
\label{Ic}\\
I_{d}:=&I^{K}_{d}(\mathbf{h},\mathbf{m})=&\bigcup^{K}_{i=1}\{k\in X:\delta -\beta +m_{i} \geq k > h_{i}\}
\label{Id}
\end{eqnarray}
An $\mathcal{S}$ particle system $(\eta_{t})_{t \geq 0}$ is \emph{stochastically larger} than an $\widetilde{\mathcal{S}}$ particle system $(\xi_{t})_{t \geq 0}$ if and only if 
\begin{align}
\sum_{k\in X:k > \delta-\beta+j_{1}}\widetilde{\Pi}_{\alpha, \beta}^{0,k}+\sum_{k \in I_{a}}\widetilde{\Gamma}_{\alpha, \beta}^{k}\leq &
\sum_{l\in X:l>j_{1}}\Pi_{\gamma, \delta}^{0,l}+\sum_{l \in I_{b}}\Gamma_{\gamma, \delta}^{l}
\label{C+}\\
\sum_{k\in X:k>h_{1}}\widetilde{\Pi}_{\alpha, \beta}^{-k,0}+\sum_{k \in I_{d}}\widetilde{\Gamma}_{\alpha, \beta}^{k}\geq &
\sum_{l\in X: l > \gamma-\alpha+h_{1}}\Pi_{\gamma, \delta}^{-l,0}+\sum_{l \in I_{c}}\Gamma_{\gamma, \delta}^{l}
\label{C-}
\end{align}
for all choices of $K\leq \widetilde{N}(\alpha,\beta)\vee N(\gamma,\delta)$, $\mathbf{h}$, $\mathbf{j}$, $\mathbf{m}$, $\alpha\leq \gamma$, $\beta\leq \delta$.
\label{cns}
\end{teorema}
\begin{remark}
The restriction $K\leq \widetilde{N}(\alpha,\beta)\vee N(\gamma,\delta)$ avoids that an infinite number of $K$, $I_a$, $I_b$, $I_c$, $I_d$ result in the same rate inequality. Since $\widetilde{\Gamma}_{\alpha, \beta}^{k}=0$ for each $k>\widetilde{N}(\alpha,\beta)$, if $K>\widetilde{N}(\alpha,\beta)$ no terms are being added to the left hand side of (\ref{C+}), and adding more terms on the right hand side does not give any new restrictions.
A similar statement holds for (\ref{C-}), with the corresponding condition $K\leq N(\gamma,\delta)$.\\

\end{remark}
\begin{remark}
If $\mathcal{S}=\widetilde{\mathcal{S}}$, Theorem $\ref{cns}$ states necessary and sufficient conditions for attractiveness of $\mathcal{S}$.
\end{remark}
We follow the approach in \cite{cf:GobronSaada}: in order to characterize the stochastic ordering of two processes, first of all we find necessary conditions on the transition rates. Then we construct a Markovian increasing coupling, that is a coupled process $(\xi_{t},\zeta_{t})_{t \geq 0}$ which has the property that $\xi_{0}\leq\zeta_{0}$ implies
$$
P^{(\xi_{0},\zeta_{0})}\{\xi_{t}\leq \zeta_{t}\}=1,
$$
for all $t \geq 0$. Here $P^{(\xi_{0},\zeta_{0})}$ denotes the distribution of $(\xi_{t},\zeta_{t})_{t\geq 0}$ with initial state $(\xi_{0},\zeta_{0})$.

\subsection{Applications}
\label{examples}

We propose several applications to understand the meaning of Conditions (\ref{C+})--(\ref{C-}). We think of $(\alpha,\beta)\leq (\gamma,\delta)$ as the configuration state of two processes $\eta_t\leq \xi_t$ on a fixed pair of sites $x$ and $y$: namely $\eta_t(x)=\alpha$, $\eta_t(y)=\beta$,  $\xi_t(x)=\gamma$, $\xi_t(y)=\delta$.
\subsubsection{No multiple births, deaths or jumps}
\label{esempioM1}
Let $\displaystyle \overline{N}=\sup_{(\alpha,\beta) \in X^2}N(\alpha,\beta)$:
\begin{proposizione}
If $\overline{N}=1$ then a change of at most one particle per time is allowed and Conditions ($\ref{C+}$) and ($\ref{C-}$) become
\begin{align}
\widetilde{\Pi}_{\alpha, \beta}^{0,1}+\widetilde{\Gamma}_{\alpha, \beta}^{1}\leq & \Pi_{\gamma, \delta}^{0,1}+\Gamma_{\gamma, \delta}^{1}   & \text {if } \beta=\delta \text{ and } \gamma \geq \alpha,
\label{C+1}\\
\widetilde{\Pi}_{\alpha, \beta}^{0,1}\leq & \Pi_{\gamma, \delta}^{0,1}  & \text {if } \beta=\delta \text{ and } \gamma = \alpha,
\label{C+10}\\
\widetilde{\Pi}_{\alpha, \beta}^{-1,0}+\widetilde{\Gamma}_{\alpha, \beta}^{1}\geq &\Pi_{\gamma, \delta}^{-1,0}+\Gamma_{\gamma, \delta}^{1}   & \text {if } \gamma=\alpha \text{ and } \delta \geq \beta,
\label{C-1}\\
\widetilde{\Pi}_{\alpha, \beta}^{-1,0}\geq &\Pi_{\gamma, \delta}^{-1,0} & \text {if } \gamma=\alpha \text{ and } \delta = \beta.
\label{C-10}
\end{align}
\label{condM1}
\end{proposizione}
\begin{proof}.
If $\beta<\delta$, then $\delta-\beta+j_i\geq \delta-\beta+j_1\geq 1$ for all $K>0$, $1\leq i\leq K$ so that $1\notin I_a$ by definition (\ref{Ia}). Since $\overline{N}=1$ the left hand side of $(\ref{C+})$ is null; if $\beta=\delta$ the only case for which the left hand side of $(\ref{C+})$ is not null is $j_{1}=0$, which gives
$$
\sum_{k > 0}\widetilde{\Pi}_{\alpha, \beta}^{0,k}+\sum_{k \in I_{a}}\widetilde{\Gamma}_{\alpha, \beta}^{k}\leq
\sum_{l>0}\Pi_{\gamma, \beta}^{0,l}+\sum_{l \in I_{b}}\Gamma_{\gamma, \beta}^{l}.
$$
Since $\overline{N}=1$, the value $K=1$ covers all possible sets $I_a$ and $I_b$, namely $I_{a}=\{k:m_{1}\geq k >0\}$ and $I_{b}=\{\gamma-\alpha+m_1\geq l>0\}$. If $m_{1}>0$, we get (\ref{C+1}). If $\gamma=\alpha$ and $m_1=0$ we get (\ref{C+10}). One can prove $(\ref{C-1})$ in a similar way.
\end{proof}
If $\beta=\delta$ and $\gamma\geq \alpha$, Formula (\ref{C+1}) expresses that the sum of the addition rates of the smaller process on $y$ in state $\beta$ must be smaller than the corresponding addition rates on $y$ of the larger process on $y$ in the same state. If $\beta=\delta$ and $\gamma=\alpha$ we also need that the birth rate of the smaller process on $y$ is smaller than the one of the larger process, that is (\ref{C+10}). Conditions (\ref{C-1})--(\ref{C-10}) have a symmetric meaning with respect to subtraction of particles from $x$. 
\begin{remark}
If $\widetilde{\mathcal{S}}=\mathcal{S}$, when $\alpha=\gamma$ and $\beta=\delta$ conditions (\ref{C+10}), (\ref{C-10}) are trivially satisfied, and we only have to check (\ref{C+1}) when $\alpha<\gamma$ and (\ref{C-1}) when $\beta<\delta$.
\label{attr_M1} 
\end{remark}
Proposition \ref{condM1} will be used in a companion paper for metapopulation models, see \cite{cf:Borrello}.
If $R^{0,k}_{\alpha,\beta}=0$ for all $\alpha$, $\beta$, $k$, the model is the reaction diffusion process studied by Chen (see $\cite{cf:Chen}$) and the attractiveness Conditions ($\ref{C+1}$), ($\ref{C-1}$) (the only ones by Remark \ref{attr_M1}) reduce to
\begin{align*}
\Gamma_{\alpha, \beta}^{1}\leq \Gamma_{\gamma, \beta}^{1}\quad \text { if } \gamma > \alpha; &&\Gamma_{\alpha, \delta}^{1}\geq  \Gamma_{\alpha, \beta}^{1}\quad \text { if } \delta > \beta.
\end{align*}
In other words we need $\Gamma^{1}_{\alpha,\beta}$ to be non decreasing with respect to $\alpha$ for each fixed $\beta$, and non increasing with respect to $\beta$ for each fixed $\alpha$. In \cite{cf:Chen}, the author introduces several couplings in order to find ergodicity conditions of reaction diffusion processes. All these couplings are identical to the coupling $\mathcal{H}$ introduced in Section \ref{coupling} (and detailed in Appendix \ref{app} if $\overline{N}=1$), on configurations where an addition or a subtraction of particles may break the partial order, but differ from $\mathcal{H}$ on configurations where it cannot happen.

\subsubsection{Multitype contact processes}
\label{sec:stover}

If $\Gamma^{k}_{\alpha,\beta}=\widetilde{\Gamma}^{k}_{\alpha,\beta}=0$, for all $(\alpha,\beta) \in X^{2}, k \geq 0$, that is when no jumps of particles are present, all rates are contact-type interactions. Such a process is called multitype contact process. Conditions (\ref{C+})--(\ref{C-}) reduce to: for all $(\alpha,\beta) \in X^{2}$, $(\gamma,\delta) \in X^{2}$, $(\alpha,\beta)\leq(\gamma,\delta)$, $h_{1} \geq 0$, $j_{1}\geq 0$,  
\begin{align}
i) \sum_{k> \delta-\beta+j_{1}}\widetilde{\Pi}_{\alpha, \beta}^{0,k}\leq  \sum_{l>j_{1}}\Pi_{\gamma, \delta}^{0,l};
\qquad
ii)\sum_{k>h_{1}}\widetilde{\Pi}_{\alpha, \beta}^{-k,0}\geq \sum_{l > \gamma-\alpha+h_{1}}\widetilde{\Pi}_{\gamma, \delta}^{-l,0}.
\label{temp+reac}
\end{align}
Many different multitype contact processes have been used to study biological models. We propose some examples with the corresponding conditions. Since the state space $\Omega=\{0,1,\ldots,M\}^{\mathbb{Z}^d}$ (where $M<\infty$) is compact, we refer to the construction in \cite{cf:Liggett}.\\

\noindent \textbf{Spread of tubercolosis model (\cite{cf:schi_clusters})}. Here $M$ represents the number of individuals in a population at a site $x \in \mathbb{Z}^d$. The transitions are:
\begin{align*}
&P^{1}_{\beta}=\phi\beta\I_{\{0\leq \beta\leq M-1\}},& R^{0,1}_{\alpha,\beta}=2d \lambda \alpha\I_{\{\beta=0\}},\\
&P^{-\beta}_{\beta}=\I_{\{1\leq \beta\leq M\}}, &p(x,y)=\frac{1}{2d}\I_{\{x\sim y\}}.
\end{align*}
where $y \sim x$ is one of the $2d$ nearest neighbours of site $x$.\\ 
Given two systems with parameters $(\lambda, \phi, M)$ and $(\overline{\lambda}, \overline{\phi}, \overline{M})$, the proof of \cite[Proposition 1]{cf:schi_clusters} reduces to check Conditions (\ref{temp+reac}): 
\begin{align*}
\phi\beta\I_{\{0\leq \beta\leq M-1\}}+2d \lambda \alpha \I_{\{\beta=0\}}\leq & \overline{\phi}\delta\I_{\{0\leq \delta\leq \overline{M}-1\}}+2d \overline{\lambda} \gamma\I_{\{\delta=0\}},& \text{if } \beta=\delta,j_1=0
\\
\I_{\{1\leq \alpha\leq M, \alpha>h_1\}}\geq &\I_{\{1\leq \gamma\leq \overline{M}, \gamma>\gamma-\alpha+h_1\}}, &\text{if }\gamma\geq \alpha,h_1\geq 0
\end{align*}
which are satisfied if $\lambda\leq \overline{\lambda}$, $\phi\leq \overline{\phi}$ and $M\leq \overline{M}$.\\

In the following examples we suppose $\widetilde{\mathcal{S}}=\mathcal{S}$, that is we consider necessary and sufficient conditions for attractiveness.\\

\noindent \textbf{$2$-type contact process (\cite{cf:neuhauser_multi})}. In this model $M=2$. Since a value on a given site does not represent the number of particles on that site, we write the state space $\{A,B,C\}^{\mathbb{Z}^d}$. The value $B$ represents the presence of a type-$B$ species, $C$ the presence of a type-$C$ species and $A$ an empty site. If $A=0$, $B=1$, $C=2$ then the transitions are 
\begin{align*}
&R^{0,1}_{\alpha,\beta}=2d \lambda_1 \I_{\{\alpha=1,\beta=0\}},&R^{0,2}_{\alpha,\beta}=2d \lambda_2 \I_{\{\alpha=2,\beta=0\}},\\
&P^{-\beta}_{\beta}=\I_{\{1\leq \beta\leq 2\}},&p(x,y)=\frac{1}{2d}\I_{\{x\sim y\}}.
\end{align*}
By taking $h_1=0$, Condition (\ref{temp+reac}) is
\begin{align*}
\sum_{k>\delta-\beta}\big(\I_{\{k=1\}}2d \lambda_1 \I_{\{\alpha=1,\beta=0\}}+&\I_{\{k=2\}}2d \lambda_2 \I_{\{\alpha=2,\beta=0\}}\big)\\
&\leq 2d \lambda_1 \I_{\{\gamma=1,\delta=0\}}+2d \lambda_2 \I_{\{\gamma=2,\delta=0\}};
\end{align*}
By taking $\beta=0$, $\delta=1$, $\alpha=\gamma=2$ we get $2d \lambda_2 \leq 0$, which is not satisfied since $\lambda_2>0$. As already observed, see \cite[Section 5.1]{cf:Stover}, one can get an attractive process by changing the order between species: namely by taking $A=1$, $B=0$ and $C=2$ the process is attractive.\\

\subsubsection{Conservative dynamics}
\label{sec:gobronsaada}

If $\Pi^{0,k}_{\alpha,\beta}=0$, for all $(\alpha,\beta) \in X^{2}, k \in \mathbb{N}$, we get a particular case of the model introduced in $\cite{cf:GobronSaada}$, for which neither particles births nor deaths are allowed, and the particle system is conservative. Suppose that in this model the rate $\Gamma^{k}_{\alpha,\beta}(y-x)$ has the form $\Gamma^{k}_{\alpha,\beta}p(y-x)$ for each $k$, $\alpha$, $\beta$.
%Suppose that in this model the number $k$ of particles which make a jump together is bounded by a constant and their rate  $\Gamma^{k}_{\alpha,\beta}(y-x)$ has the form $\Gamma^{k}_{\alpha,\beta}p(y-x)$ that we introduced in $(\ref{gen})$. Denote by
%$$
%M:=\max \{k:\Gamma^{k}_{\alpha,\beta}(y-x)>0 \text{ for each } \alpha, \beta, x, y\}.
%$$
Necessary and sufficient conditions for attractiveness are given by \cite[Theorem $2.21$]{cf:GobronSaada}:
\begin{align}
\displaystyle \sum_{k > \delta-\beta+j}\Gamma_{\alpha, \beta}^{k}p(y-x)\leq &
\sum_{l>j}\Gamma_{\gamma, \delta}^{l}p(y-x) &\text{ for each } j\geq 0
\label{SaadaC+}\\
\displaystyle \sum_{k>h}\Gamma_{\alpha, \beta}^{k}p(y-x)\geq & \sum_{l > \gamma-\alpha+h}\Gamma_{\gamma, \delta}^{l}p(y-x) &\text{ for each } h \geq 0
\label{SaadaC-}
\end{align}
while $(\ref{C+})-(\ref{C-})$ become
\begin{align}
\sum_{k \in I_{a}}\Gamma_{\alpha, \beta}^{k}p(y-x)\leq &
\sum_{l \in I_{b}}\Gamma_{\gamma, \delta}^{l}p(y-x)
\label{C+conf1}\\
\sum_{k \in I_{d}}\Gamma_{\alpha, \beta}^{k}p(y-x)\geq &
\sum_{l\in I_{c}}\Gamma_{\gamma, \delta}^{l}p(y-x)
\label{C-conf1}
\end{align}
for all $I_a$, $I_b$, $I_c$ and $I_d$ given by Theorem \ref{cns}.
\begin{proposizione}
Conditions ($\ref{SaadaC+}$)--($\ref{SaadaC-}$) and Conditions ($\ref{C+conf1}$)--($\ref{C-conf1}$) are equivalent.
\label{equivprops}
\end{proposizione}
\begin{proof}.
By choosing $j_{i}=j$, $h_{i}=h$ and $m_{i}=m>N(\alpha,\beta)\vee N(\gamma,\delta)$ for each $i$ in Theorem \ref{cns}, Conditions (\ref{C+conf1})--(\ref{C-conf1}) imply (\ref{SaadaC+})--(\ref{SaadaC-}). For the opposite direction, by subtracting ($\ref{SaadaC-}$) with $h=m_i>\delta-\beta+j_i$ to ($\ref{SaadaC+}$) with $j=j_i$
\begin{align}
\sum_{m_i\geq k > \delta-\beta+j_i}\Gamma_{\alpha, \beta}^{k}p(y-x)\leq
\sum_{ \gamma-\alpha+m_i\geq l'>j_i}\Gamma_{\gamma, \delta}^{l'}p(y-x)
\label{modulo}
\end{align}
and we get  (\ref{C+conf1}) by summing $K$ times (\ref{modulo}) with different values of $j_i$, $m_i$, $1\leq i\leq K$. Condition (\ref{C-conf1}) follows in a similar way.
\end{proof}

\subsubsection{Metapopulation model with Allee effect and mass migration}
\label{meta}
The third model investigated in \cite{cf:Borrello} is a metapopulation dynamics model where migrations of many individuals per time are allowed to avoid the biological phenomenon of the Allee effect (see \cite{cf:allee1}, \cite{cf:allee2}). The state space is compact and on each site $x\in\mathbb{Z}^d$ there is a local population of at most $M$ individuals. Given $M\geq M_A>0$, $M>N>0$, $\phi$, $\phi_A$, $\lambda$ positive real numbers, the transitions are
\begin{align*}
&P^1_\beta=\beta\I_{\{\beta\leq M-1\}} \quad P^{-1}_\beta=\beta\big(\phi_{A}\I_{\{\beta \leq M_{A}\}}+\phi\I_{\{M_{A}<\beta\}}\big)\nonumber \\
&\Gamma^k_{\alpha,\beta}=\left\{\begin{array}{ll}
\lambda & \alpha-k\geq M-N \text{ and }\beta+k\leq M,\\
0 &\text{otherwise}.
\end{array}
\right.
\end{align*}
for each $\alpha,\beta \in X$, and $p(x,y)=\frac{1}{2d}\I_{\{y\sim x\}}$. In other words each individual reproduces with rate $1$, but dies with different rates: either $\phi_A$ if the local population size is smaller than $M_A$ (Allee effect) or $\phi$ if it is larger. When a local population has more than $M-N$ individuals a migration of more than one individual per time is allowed. Such a process is attractive by \cite[Proposition \Mon,  where $N$ and $M$ play opposite roles]{cf:Borrello}, which is an application of Theorem \ref{cns}.   

\subsubsection{Individual recovery epidemic model}
\label{IRM}

We apply Theorem \ref{cns} to get new ergodicity conditions for a model of spread of epidemics.\\
The most investigated interacting particle system that models the spread of epidemics is the contact process, introduced by Harris \cite{cf:Harris2}. It is a spin system $(\eta_{t})_{t \geq 0}$ on $\{0,1\}^{\mathbb{Z}^{d}}$ ruled by the transitions 
\begin{align*}
&\eta_{t}(x)=0\to  1 \text{ at rate }\sum_{y \sim x}\eta_{t}(y) \\
&\eta_{t}(x)=1\to  0 \text{ at rate }1 
\end{align*}
See \cite{cf:Liggett} and \cite{cf:Liggett2} for an exhaustive analysis of this model.\\
In order to understand the role of social clusters in the spread of epidemics, Schinazi \cite{cf:schi_clusters} introduced a generalization of the contact process. Then, Belhadji \cite{cf:Belhadji} investigated some generalizations of this model: on each site in $\mathbb{Z}^{d}$ there is a cluster of $M\leq \infty$ individuals, where each individual can be \emph{healthy} or \emph{infected}. A cluster is infected if there is at least one infected individual, otherwise it is healthy. The illness moves from an infected individual to a healthy one with rate $\phi$ if they are in the same cluster. The infection rate between different clusters is different: the epidemics moves from an infected individual in a cluster $y$ to an individual in a neighboring cluster $x$ with rate $\lambda$ if $x$ is healthy, and with rate $\beta$ if $x$ is infected.\\
We focus on one of those models, the \emph{individual recovery epidemic model} in a compact state space: each sick individual recovers after an exponential time and each cluster contains at most $M$ individuals. The non-null transition rates are
\begin{align}
R_{\eta(x),\eta(y)}^{0,k}=&\left\{
\begin{array}{ll}
2d\lambda\eta(x) & k=1 \text{ and }\eta(y)=0\\ 
2d\beta\eta(x) & k=1 \text{ and }1 \leq \eta(y) \leq M-1,
\end{array}
\right.\\
P^{k}_{\eta(y)}=&\left\{
\begin{array}{ll}
\gamma & k=1 \text{ and }\eta(y)=0\\
\phi\eta(y) & k=1 \text{ and }1 \leq \eta(y) \leq M-1\\
\eta(y) & k=-1 \text{ and }1 \leq \eta(y) \leq M,
\end{array}
\right.\\
p(x,y)=&\frac{1}{2d}\I_{\{y \sim x\}}\qquad \text{ for each }(x,y) \in S^{2}.
\end{align}
and $\Gamma_{\eta(x),\eta(y)}^{k}=0$ for all $k \in \mathbb{N}$. The rate $\gamma$ represents a positive ``pure birth'' of the illness: by setting $\gamma=0$, we get the epidemic model in \cite{cf:Belhadji}, where the author analyses the system with $M< \infty$ and $M=\infty$ and shows \cite[Theorem $14$]{cf:Belhadji} that different phase transitions occur with respect to $\lambda$ and $\phi$. Moreover (\cite[Theorem 15]{cf:Belhadji}), if 
\begin{equation}
\lambda\vee \beta<\frac{1-\phi}{2d}
\label{cond1}
\end{equation}
the disease dies out for each cluster size $M$.\\
By using the attractiveness of the model we improve this ergodicity condition. Notice that a dependence on the cluster size $M$ appears.
\begin{teorema}
Suppose 
\begin{equation}
\lambda \vee \beta < \frac{1-\phi}{2d(1-\phi^{M})},
\label{cond} 
\end{equation}
with $\phi<1$ and either $i)$ $\gamma=0$, or $ii)$ $\gamma>0$ and $\beta-\lambda\leq \gamma/(2d)$.\\
Then the system is ergodic.
\label{general_erg} 
\end{teorema}
Notice that if $\gamma>0$ and $\beta\leq \lambda$ hypothesis $ii)$ is trivially satisfied.\\
If $M=1$ and $\gamma=0$ the process reduces to the contact process and the result is a well known (and already improved) ergodic result (see for instance \cite[Corollary $4.4$, Chapter VI]{cf:Liggett}); as a corollary we get the ergodicity result in the non compact case as $M$ goes to infinity.\\
In order to prove Theorem \ref{general_erg}, we use a technique called \emph{$u$-criterion}. It gives sufficient conditions on transition rates which yield ergodicity of an attractive translation invariant process. It has been used by several authors (see \cite{cf:Chen}, \cite{cf:Chenbook}, \cite{cf:neuhauser_erg}) for reaction-diffusion processes.\\

First of all we observe that 
\begin{proposizione}
The process is attractive for all $\lambda$, $\beta$, $\gamma$, $\phi$, $M$.
\end{proposizione}
\begin{proof}.
Since $M=1$, then $\overline{N}=1$ and Conditions ($\ref{C+}$)--($\ref{C-}$) reduce to ($\ref{C+1}$), ($\ref{C-1}$). Namely, given two configurations $\eta \in \Omega$, $\xi \in \Omega$ with $\xi\leq \eta$, necessary and sufficient conditions for attractiveness are 
$$
\text {if }  \xi(x) < \eta(x) \text{ and }\xi(y)=\eta(y),\quad\left\{
\begin{array}{l}
R^{0,1}_{\xi(x),\xi(y)}+P^{1}_{\xi(y)}\leq R^{0,1}_{\eta(x),\eta(y)}+P^{1}_{\eta(y)};\\
P^{-1}_{\xi(y)}\geq P^{-1}_{\eta(y)}.
\end{array}
\right.
$$
If $\xi(y)=\eta(y)\geq 1$, $2d\beta \xi(x)\leq 2d\beta \eta(x)$; if $\xi(y)=\eta(y)=0$, $2d\lambda \xi(x)\leq 2d\lambda \eta(x)$. In all cases the condition holds since $\xi\leq \eta$. 
\end{proof}
The key point for attractiveness is that $R^{0,1}_{\eta(x),\eta(y)}$ is increasing in $\eta(x)$.\\

Given $\epsilon >0$ and $\{u_{l}(\epsilon)\}_{l\in X}$ such that $u_{l}(\epsilon)>0$ for all $l\in X$, let $F_{\epsilon}:X \times X \to \mathbb{R}^+$ be defined by 
\begin{equation}
F_{\epsilon}(x,y)=\I_{\{x\neq y\}}\sum_{j=0}^{|y-x|-1}u_{j}(\epsilon) 
\label{F}
\end{equation} 
for all $x,y \in X$. When not necessary we omit the dependence on $\epsilon$ and we simply write $\displaystyle F(x,y)=\I_{\{x\neq y\}}\sum_{j=0}^{|y-x|-1}u_{j}$. Since $u_{l}(\epsilon)>0$, this is a metric on $X$ and it induces in a natural way a metric on $\Omega$. Namely, for each $\eta$ and $\xi$  in $\Omega$ we define
\begin{equation}
\rho_{\alpha}(\eta,\xi):=\sum_{x\in \mathbb{Z}^{d}}F(\eta(x),\xi(x))\alpha(x)
\label{alfa_intro}
\end{equation}
where $\{\alpha(x)\}_{x\in \mathbb{Z}^{d}}$ is a sequence such that $\alpha(x)\in \mathbb{R}$, $\alpha(x)>0$ for each $x\in\mathbb{Z}^{d}$ and
\begin{equation}
\sum_{x\in \mathbb{Z}^{d}}\alpha(x)<\infty.
\label{ro_intro}
\end{equation}
Denote by $\Omega^m:=\{\eta \in \Omega: \eta(x)=m \text{ for each } x \in \mathbb{Z}^d\}$ and let $\eta_0^M\in \Omega^M$ and $\eta_0^0 \in \Omega^0$ (which is not an absorbing state if $\gamma>0$). 
The key idea consists in taking a ``good sequence" $\{u_l\}_{l \in X}$ and in looking for conditions on the rates under which the expected value $\widetilde{\mathbb{E}}(\cdot)$ (with respect to a coupled measure $\widetilde{\mathbb{P}}$) of the distance between $\eta^{M}_{t}$ and $\eta^{0}_{t}$ converges to zero as $t$ goes to infinity, uniformly with respect to $x \in S$. We will use the generator properties and Gronwall's Lemma to prove that if there exists $\epsilon>0$ and a sequence $\{u_l(\epsilon)\}_{l \in X}$, $u_l(\epsilon)>0$ for all $l\in X$ such that  the metric $F$ satisfies
\begin{equation}
\widetilde{\mathbb{E}}\Big(\mathcal{L}F(\eta^{0}_{t}(x),\eta^{M}_{t}(x))\Big)\leq -\epsilon \widetilde{\mathbb{E}}\Big(F(\eta^{0}_{t}(x),\eta^{M}_{t}(x))\Big)
\label{tmp_Chen}
\end{equation}
for each $x \in \mathbb{Z}^d$, then 
\begin{equation}
\widetilde{\mathbb{E}}\Big(\lim_{t\to \infty}F(\eta^{0}_{t}(x),\eta^{M}_{t}(x))\Big)=0
\label{intr:limF}
\end{equation}
uniformly with respect to $x \in S$ so that the distance $\rho_{\alpha}(\cdot,\cdot)$ between the larger and the smaller process converges to zero, and ergodicity follows.\\  

Condition (\ref{tmp_Chen}) leads to
\begin{proposizione}[u-criterion]
If there exists $\epsilon>0$ and a sequence $\{u_{l}(\epsilon)\}_{l\in X}$ such that for any $l \in X$
\begin{equation}
\left\{
\begin{array}{l}
\displaystyle \phi l u_{l}(\epsilon)-lu_{l-1}(\epsilon) \leq -\epsilon \sum_{j=0}^{l-1}u_{j}(\epsilon)-\bar{u}(\epsilon)(\lambda \vee \beta)2dl\\ 
u_{l}(\epsilon)>0 
\end{array}
\right.
\label{condu}
\end{equation}
where $\displaystyle \bar{u}(\epsilon):=\max_{l\in X}u_{l}(\epsilon)$, $u_{-1}=0$ by convention  and $u_{0}=U>0$, then the system is ergodic.
\label{ucriterio}
\end{proposizione} 
Hence we are left with checking the existence of $\epsilon>0$ and positive $\{u_{l}(\epsilon)\}_{l\in X}$ which satisfy $(\ref{condu})$. Such a choice is not unique.
\begin{remark}
Given $U=1>0$, if  $u_{l}(\epsilon)=1$ for each $l \in X$ then Condition ($\ref{condu}$) reduces to the existence of $\epsilon>0$ such that $\phi l-l \leq -\epsilon l -(\lambda \vee \beta)2dl$ for all $l\in X$, that is $\epsilon \leq 1-\phi-(\lambda \vee \beta)2d$, thus Condition $(\ref{cond1})$. Notice that in this case 
$$
F(x,y)=\sum^{|y-x|-1}_{j=0}1=|y-x|.
$$
\end{remark}
Another possible choice is  
\begin{definition}
Given $\epsilon>0$ and $U>0$, we set $u_{0}(\epsilon)=U$ and we define $(u_{l}(\epsilon))_{l \in X}$ recursively through
\begin{equation}
u_{l}(\epsilon)=\frac{1}{\phi l}\Big(-\epsilon\sum_{j=0}^{l-1}u_{j}(\epsilon)-U(\lambda \vee \beta)2dl+lu_{l-1}(\epsilon)\Big) \qquad \text{for each } l\in X, l\neq 0.
\label{eq:ul}
\end{equation}
\label{ul}
\end{definition}
Definition \ref{ul} gives a better choice of $\{u_{l}(\epsilon)\}_{l\in X}$, indeed the $u$-criterion is satisfied under the more general assumption (\ref{cond}). Proofs of Theorems \ref{general_erg} and \ref{ucriterio} are detailed in Section \ref{proof_appl}.

\section{Coupling construction and proof of Theorem \ref{cns}}
\label{Proofs}
In this section we prove the main result: we begin with the necessary condition, based on \cite[Proposition 2.24]{cf:GobronSaada}.
\subsection{Necessary condition}
\begin{proposizione}
If the particle system $(\eta_{t})_{t \geq 0}$ is stochastically larger than $(\xi_{t})_{t \geq 0}$, then for all $(\alpha,\beta),(\gamma,\delta) \in X^{2} \times X^{2}$ with $(\alpha, \beta)\leq (\gamma,\delta)$, for all $(x,y) \in S^{2}$, Conditions $(\ref{C+})$--$(\ref{C-})$ hold.
\label{cn}
\end{proposizione}
\begin{proof}.
Let $(\xi,\eta) \in \Omega \times \Omega$ be two configurations such that $\xi \leq \eta$. Let $V \subset \Omega$ be an increasing cylinder set. If $\xi \in V$ or $\eta \notin V$,
\begin{equation}
\I_{V}(\xi)=\I_{V}(\eta).
\label{ug}
\end{equation}
Since $\eta_{t}$ is stochastically larger than $\xi_{t}$, for all $t \geq 0$, by Theorem $\ref{Genlig1}$ (or \cite[Theorem II.2.2]{cf:Liggett} if we are interested in attractiveness)
$$
(\widetilde{T}(t)\I_{V})(\xi)\leq (T(t)\I_{V})(\eta)
$$
since $\nu \leq \nu'$ is equivalent to $\nu(V) \leq \nu'(V)$ for all increasing sets. Combining this with $(\ref{ug})$,
$$
t^{-1}[(\widetilde{T}(t)\I_{V})(\xi)-\I_{V}(\xi)]\leq t^{-1}[(T(t)\I_{V})(\eta)-\I_{V}(\eta)],
$$
which gives, by Assumption ($\ref{defM}$),
\begin{equation}
(\widetilde{\mathcal{L}}\I_{V})(\xi)\leq (\mathcal{L}\I_{V})(\eta).
\label{disgen}
\end{equation} 
We have, by using ($\ref{gen}$),
\begin{align}
(\widetilde{\mathcal{L}}\I_{V})(\xi)=&\sum_{x,y \in S}p(x,y)\sum_{\alpha,\beta \in X}\ki (\xi)\sum_{k > 0}\Big(\widetilde{\Gamma}^{k}_{\alpha,\beta} (\I_{V}(S^{-k,k}_{x,y}\xi)-\I_{V}(\xi))
\nonumber\\
&+\widetilde{\Pi}^{0,k}_{\alpha,\beta}(\I_{V}(S^{0,k}_{x,y}\xi)-\I_{V}(\xi))+\widetilde{\Pi}^{-k,0}_{\alpha,\beta}(\I_{V}(S^{-k,0}_{x,y}\xi)-\I_{V}(\xi))\Big)\nonumber \\
=&-\I_{V}(\xi)\sum_{x,y \in S}p(x,y)\sum_{\alpha,\beta \in X}\ki (\xi)\sum_{k \geq 0}\Big(\widetilde{\Gamma}^{k}_{\alpha,\beta} \I_{\Omega \backslash V}(S^{-k,k}_{x,y}\xi)\nonumber\\
&+\widetilde{\Pi}^{0,k}_{\alpha,\beta}\I_{\Omega \backslash V}(S^{0,k}_{x,y}\xi)+\widetilde{\Pi}^{-k,0}_{\alpha,\beta}\I_{\Omega \backslash V}(S^{-k,0}_{x,y}\xi)\Big)\nonumber\\
&+\I_{\Omega \backslash V}(\xi)\sum_{x,y \in S}p(x,y)\sum_{\alpha,\beta \in X}\ki (\xi)\sum_{k \geq 0}\Big(\widetilde{\Gamma}^{k}_{\alpha,\beta} \I_{V}(S^{-k,k}_{x,y}\xi)\nonumber\\
&+\widetilde{\Pi}^{0,k}_{\alpha,\beta}\I_{V}(S^{0,k}_{x,y}\xi)+\widetilde{\Pi}^{-k,0}_{\alpha,\beta}\I_{V}(S^{-k,0}_{x,y}\xi)\Big).
\label{genset}
\end{align}
We write $(\mathcal{L}\I_{V})(\eta)$ by using the corresponding rates of $\mathcal{S}$.\\
We fix $y \in S$, $(\alpha_{z},\gamma_{z},\beta,\delta) \in X^{4}$ with $(\alpha_{z},\beta) \leq (\gamma_{z},\delta)$ for all $z \in S$, $z \neq y$, and two configurations $(\xi,\eta) \in \Omega \times \Omega$ such that $\xi(z)=\alpha_{z}$, $\eta(z)=\gamma_{z}$, for all $z \in S$, $z \neq y$, $\xi(y)=\beta$, $\eta(y)=\delta$. Thus $\xi \leq \eta$. We define the set $C^{+}_{y}$ of sites which interact with $y$ with an increase of the configuration on $y$,
\begin{equation}
\begin{split}
C^{+}_{y}=&\Big\{z \in S: p(z,y)>0 \text{ and }\sum_{k>0}\big(\widetilde{\Gamma}^{k}_{\alpha_{z},\beta}+\Gamma^{k}_{\gamma_{z},\delta}+\widetilde{\Pi}^{0,k}_{\alpha_{z},\beta}+\Pi^{0,k}_{\gamma_{z},\delta}\big)>0\Big\}.
\end{split}
\end{equation}
Denote by $x=(x_{1},\ldots,x_{d})$ the coordinates of each $x \in S$.  
We define, for each $n \in \mathbb{N}$, $C^{+}_{y}(n)=C^{+}_{y}\cap\{z \in S: \sum_{i=1}^{d}|z_{i}-y_{i}|\leq n\}$. We may suppose $C^{+}_{y}\neq \emptyset$, since otherwise $(\ref{C+}$)--($\ref{C-})$ would be trivially satisfied.
Given $K\in \mathbb{N}$, we fix $\{p^{i}_{z}\}_{i\leq K, z \in C^{+}_{y}}$ such that for each $i$, $z$, $p^{i}_{z}\in X$ and $p^{i}_{z}\leq \xi(z)$. Moreover we fix $\{p^{i}_{y}\}_{i\leq K}$ such that $p^{i}_{y}>\delta$, for each $i$.
For $n \in \mathbb{N}$, let 
$$
I_{y}(n)=\bigcup^{K}_{i=1}\big\{\zeta \in \Omega: \zeta(y)\geq p^{i}_{y} \text{ and }\zeta(z)\geq p^{i}_{z},\text{ for all } z \in C^{+}_{y}(n)\big\}.
$$
The union of increasing cylinder sets $I_{y}(n)$ is an increasing set, to which neither $\xi$ nor $\eta$ belong. We compute, using ($\ref{genset}$),
\begin{equation*}
\begin{split}
(\widetilde{\mathcal{L}}\I_{I_{y}(n)})(\xi)=&\sum_{z \in C^{+}_{y}(n)}p(z,y)\sum_{k>0}\big(\I_{\bigcup^{K}_{i=1}\{\alpha_{z}-k\geq p^{i}_{z},\beta+k \geq p^{i}_{y}\}}\widetilde{\Gamma}^{k}_{\alpha_{z},\beta}+\I_{\bigcup^{K}_{i=1}\{\beta+k \geq p^{i}_{y}\}}\widetilde{\Pi}^{0,k}_{\alpha_{z},\beta}\big)\\
&+\sum_{z \notin C^{+}_{y}(n)}p(z,y)\sum_{k>0}\I_{\bigcup^{K}_{i=1}\{\beta+k \geq p^{i}_{y}\}}\big(\widetilde{\Gamma}^{k}_{\alpha_{z},\beta}+\widetilde{\Pi}^{0,k}_{\alpha,\beta}\big)
\end{split}
\end{equation*}
\begin{equation*}
\begin{split}
(\mathcal{L}\I_{I_{y}(n)})(\eta)=&\sum_{z \in C^{+}_{y}(n)}p(z,y)\sum_{l>0}\big(\I_{\bigcup^{K}_{i=1}\{\gamma_{z}-l\geq p^{i}_{z},\delta+l \geq p^{i}_{y}\}}\Gamma^{l}_{\gamma_{z},\delta}+\I_{\bigcup^{K}_{i=1}\{\delta+l \geq p^{i}_{y}\}}\Pi^{0,l}_{\gamma_{z},\delta}\big)\\
&+\sum_{z \notin C^{+}_{y}(n)}p(z,y)\sum_{l>0}\I_{\bigcup^{K}_{i=1}\{\delta+l \geq p^{i}_{y}\}}\big(\Gamma^{l}_{\gamma_{z},\delta}+\Pi^{0,l}_{\gamma,\delta}\big).
\end{split}
\end{equation*}
So, by setting 
$$
J(p_{z},a,b):=J(\{p^{i}_{z}\}_{i\leq K},a,b)=\bigcup^{K}_{i=1}\{l:a-l\geq p^{i}_{z},b+l \geq p^{i}_{y}\},
$$ 
and by ($\ref{disgen}$), if $p^{1}_{y}:=\min_{i}\{p^{i}_{y}\}$, we get
\begin{equation*}
\begin{split}
\sum_{z \in C^{+}_{y}(n)}p(z,y)\big(\sum_{k \in J(p_{z},\alpha_{z},\beta)}\widetilde{\Gamma}^{k}_{\alpha_{z},\beta}+\sum_{k \geq p^{1}_{y}-\beta}\widetilde{\Pi}^{0,k}_{\alpha_{z},\beta}\big)+\sum_{z \notin C^{+}_{y}(n)}p(z,y)\sum_{k \geq p^{1}_{y}-\beta}\big(\widetilde{\Pi}^{0,k}_{\alpha_{z},\beta}+\widetilde{\Gamma}^{k}_{\alpha_{z},\beta}\big)\\
\leq \sum_{z \in C^{+}_{y}(n)}p(z,y)\big(\sum_{l \in J(p_{z},\gamma_{z},\delta)}\Gamma^{l}_{\gamma_{z},\delta}+\sum_{l \geq p^{1}_{y}-\delta}\Pi^{0,l}_{\gamma_{z},\delta}\big)+\sum_{z \notin C^{+}_{y}(n)}p(z,y)\sum_{l \geq p^{1}_{y}-\delta}\big(\Pi^{0,l}_{\gamma_{z},\delta}+\Gamma^{l}_{\gamma_{z},\delta}\big)
\end{split}
\end{equation*}
Taking the monotone limit $n \to \infty$ gives
\begin{equation*}
\begin{split}
\sum_{z \in S}\big(\sum_{k \in J(p_{z},\alpha_{z},\beta)}\widetilde{\Gamma}^{k}_{\alpha_{z},\beta}+\sum_{k \geq p^{1}_{y}-\beta}\widetilde{\Pi}^{0,k}_{\alpha_{z},\beta}\big)p(z,y) \leq \sum_{z \in S}\big(\sum_{l \in J(p_{z},\gamma_{z},\beta)}\Gamma^{l}_{\gamma_{z},\delta}+\sum_{l \geq p^{1}_{y}-\delta}\Pi^{0,l}_{\gamma_{z},\delta}\big)p(z,y)
\end{split}
\end{equation*}
By choosing the values $\alpha_{z}\equiv\alpha$, $\gamma_{z}\equiv \gamma$, $p^{i}_{z}\equiv p^{i}_{\alpha}$, since $\sum_z p(z,y)=1$, 
\begin{equation}
\sum_{k \in J(p_{\alpha},\alpha,\beta)}\widetilde{\Gamma}^{k}_{\alpha,\beta}+\sum_{k \geq p^{1}_{y}-\beta}\widetilde{\Pi}^{0,k}_{\alpha,\beta} \leq
\sum_{l \in J(p_{\alpha},\gamma,\delta)}\Gamma^{l}_{\gamma,\delta}+\sum_{l \geq p^{1}_{y}-\delta}\Pi^{0,l}_{\gamma,\delta}.
\label{cn1}
\end{equation}
A similar argument with 
\begin{equation*}
\begin{split}
C^{-}_{x}=&\{z \in S: p(z,y)>0 \text{ and } \sum_{k>0}\big(\widetilde{\Gamma}^{k}_{\alpha,\gamma(z)}+\Gamma^{k}_{\gamma,\delta_{z}}+\widetilde{\Pi}^{-k,0}_{\alpha,\beta_{z}}+\Pi^{-k,0}_{\gamma,\delta_{z}}\big)>0\}
\end{split}
\end{equation*}
subsets $C^{-}_{x}(n), n \in \mathbb{N}$, $\{p^{i}_{z}\}_{z \in C^{-}_{x}(n),i\leq K} \in X$, $\{p^{i}_{x}\}_{i\leq K} \in X$ with $p^{i}_{x}< \alpha$, for each $i$, $p^{i}_{z} \geq \delta_{z}$, for all $z \in C^{-}_{y}(n)$, $i\leq K$, decreasing cylinder sets
$$
D_{x}(n)=\bigcup_{i=1}^{K}\big\{\zeta \in \Omega: \zeta(x)\leq p^{i}_{x} \text{ and }\zeta(z)\leq p^{i}_{z},\text{ for all } z \in C^{-}_{y}(n)\big\}
$$
to the complement of which $\xi$ and $\eta$ belong, and the application of inequality ($\ref{disgen}$) to $\xi,\eta,\Omega \backslash D_{x}(n)$ (which is an increasing set since it is the complement of an increasing one) leads to
\begin{equation}
\sum_{k \in J^{-}(p_{\gamma},\alpha,\beta)}\widetilde{\Gamma}^{k}_{\alpha,\beta}+\sum_{\alpha-k \leq p^{1}_{x}}\widetilde{\Pi}^{-k,0}_{\alpha,\beta} \geq
\sum_{l \in J^{-}(p_{\gamma},\gamma,\delta)}\Gamma^{l}_{\gamma,\delta}+\sum_{\gamma-l \leq p^{1}_{x}}\Pi^{-l,0}_{\gamma,\delta}
\label{cn2}
\end{equation}
where $p^{1}_{x}=\max_{i}\{p^{i}_{x}\}$,
$$
J^{-}(p_{\gamma},a,b):=J^{-}(\{p^{i}_{\gamma}\}_{i\leq K},a,b)=\bigcup_{i=1}^{K}\{l:a-l \leq p^{i}_{x},b+l \leq p^{i}_{\gamma}\}.
$$ 
Finally, taking $p^{i}_{y}=\delta+j_{i}+1$, $p^{i}_{\alpha}=\alpha-m_{i}$ in ($\ref{cn1}$),
$p^{i}_{x}=\alpha-h_{i}-1$, $p^{i}_{\gamma}=\delta+m_{i}$ in ($\ref{cn2}$) gives $(\ref{C+}$)--($\ref{C-})$. 
\end{proof}

\subsection{Coupling construction}
\label{coupling}

The (harder) sufficient condition of Theorem $\ref{cns}$ is obtained by showing (in this subsection) the existence of a Markovian coupling, which appears to be increasing under Conditions $(\ref{C+})$--$(\ref{C-})$ (see Subsection \ref{sufcond}). Our method is inspired by \cite[Propositions 2.25, 2.39, 2.44]{cf:GobronSaada}, but it is much more intricate since we are dealing with jumps, births and deaths.\\
Let $\xi_{t}\sim \widetilde{\mathcal{S}}$ and $\eta_{t}\sim \mathcal{S}$ such that $\xi_t\leq \eta_t$. The first step consists in proving that instead of taking all possible sites, it is enough to consider an ordered pair of sites $(x,y)$ and to construct an increasing coupling concerning some of the rates depending on $\eta_t(x),\eta_t(y)$ and $\xi_t(x),\xi_t(y)$ (remember that we choose to take births on $y$, deaths on $x$ and jumps from $x$ to $y$) and a small part of the independent rates (by this, we mean deaths from $x$ with a rate depending only on $\eta_t(x)$ and $\xi_t(x)$ and births upon $y$ with a rate depending only on $\eta_t(y)$ and $\xi_t(y)$). We do not have to combine any ``dependent reaction'' $\widetilde{R}^{\cdot,\cdot}_{\xi_t(x),\xi_t(y)}$ or jumps rate $\widetilde{\Gamma}^{\cdot}_{\xi_t(x),\xi_t(y)}$ on $y$ with any rate $R^{\cdot,\cdot}_{\eta_t(z),\eta_t(y)}$ or $\Gamma^{\cdot}_{\eta_t(z),\eta_t(y)}$ if $z$ is different from $x$.
\begin{definition}
For fixed $(x,y)\in S^2$, for all $\eta \in \Omega$ and $k \in \mathbb{N}$ let
$$
h(P^{k}_{\eta(z)})=\left\{
\begin{array}{ll}
P^{k}_{\eta(y)}p(x,y) & z=y,\\ 
0 & \text{otherwise;}
\end{array}
\right. \qquad
h(P^{-k}_{\eta(z)})=\left\{
\begin{array}{ll}
P^{-k}_{\eta(x)}p(x,y) & z=x,\\ 
0 & \text{otherwise.}
\end{array}
\right.
$$
and 
$$
q(z,w)=\left\{
\begin{array}{ll}
p(x,y) & \text{if } z=x \text { and }w=y,\\
0& \text{otherwise.}
\end{array}
\right.
$$
An ordered pair of sites $(x,y)$ is an \emph{attractive pair} for $(\widetilde{\mathcal{S}},\mathcal{S})$ if there exists an increasing coupling for $(\xi_{t},\eta_{t})_{t\geq 0}$ where $\xi_{t}\sim \widetilde{\mathcal{S}}(\widetilde{\Gamma},\widetilde{R},h(\widetilde{P}),q)$, $\eta_{t}\sim \mathcal{S}(\Gamma,R,h(P),q)$. For notational convenience we call these new systems $\mathcal{S}_{(x,y)}$ and $\widetilde{\mathcal{S}}_{(x,y)}$.
\label{Sxy}
\end{definition}
Notice that $\mathcal{S}_{(x,y)}\neq \mathcal{S}_{(y,x)}$, because we take into account births on the second site, deaths on the first one, and only particles' jumps from the first site to the second one. The same remark holds for $\widetilde{\mathcal{S}}_{(x,y)}$. In other words in order to see if a pair is attractive, we reduce ourselves to a system with only part of the rates depending on the pair, and a part of the independent rates depending on $p(x,y)$ ($P^{k}_{\eta_t(y)}p(x,y)$ and $P^{-k}_{\eta_t(x)}p(x,y)$).
\begin{proposizione}
The process $\eta_{t}\sim \mathcal{S}$ is stochastically larger than $\xi_{t}\sim \widetilde{\mathcal{S}}$ if all its pairs are attractive pairs for $(\widetilde{\mathcal{S}},\mathcal{S})$.
\label{attpairs}
\end{proposizione}
\begin{proof}.
If for each pair $(x,y)$ we are able to construct an increasing coupling for $(\widetilde{\mathcal{S}}_{(x,y)},\mathcal{S}_{(x,y)})$, we define an increasing coupling for $(\widetilde{\mathcal{S}},\mathcal{S})$ by superposition of all these couplings for pairs. Indeed: it is a coupling since by Definition \ref{Sxy} the sum of all marginals gives the original rates; it is increasing since each coupling for a pair is increasing. 
\end{proof}
From now on, we work on a fixed pair of sites $(x,y)\in S^{2}$ with $p(x,y)>0$ and on two ordered configurations $(\eta_{t},\xi_{t}) \in \Omega^{2}$ for a fixed $t \geq 0$, $\eta_{t}\sim \mathcal{S}$, $\xi_{t}\sim \widetilde{\mathcal{S}}$, $\xi_{t} \leq \eta_{t}$ and
\begin{equation}
\xi_{t}(x)=\alpha,\quad \xi_{t}(y)=\beta,\quad\eta_{t}(x)=\gamma,\quad \eta_{t}(y)=\delta.
\label{config}
\end{equation}
We denote by $N:=N(\alpha,\beta)\vee N(\gamma,\delta)$ and by $p:=p(x,y)$.
\begin{remark}
With a slight abuse of notation, since rates on $(\alpha,\beta)$ involve $\widetilde{\mathcal{S}}$ and rates on $(\gamma,\delta)$ involve $\mathcal{S}$, we omit the superscript $\sim$ on the lower system rates: we denote by 
$$
\widetilde{\Pi}^{\cdot,\cdot}_{\alpha,\beta}=\Pi^{\cdot,\cdot}_{\alpha,\beta}; \qquad \widetilde{\Gamma}^{\cdot}_{\alpha,\beta}=\Gamma^{\cdot}_{\alpha,\beta};  \qquad \mathcal{S}:=(\widetilde{\mathcal{S}},\mathcal{S}).
$$
\label{abuse}
\end{remark}
The rest of this section is devoted to the construction of an increasing coupling for $(x,y)$ and $(\alpha,\beta)\leq (\gamma,\delta)$. 
\begin{definition}
There is a \emph{lower attractiveness problem} on $\beta$ if there exists $k$ such that $\beta+k >\delta$ and $\Pi^{0,k}_{\alpha,\beta}+\Gamma^{k}_{\alpha,\beta}>0$; $\beta$ is $k$-\emph{bad} and $k$ is a \emph{bad value} (with respect to $\beta$). There is a \emph{higher attractiveness problem} on $\gamma$ if there exists $l$ such that $\gamma-l <\alpha$ and $\Pi^{-l,0}_{\gamma,\delta}+\Gamma^{l}_{\gamma,\delta}>0$; $\gamma$ is $l$-\emph{bad} and $l$ is a \emph{bad value} (with respect to $\gamma$). Otherwise $\beta$ is $k$-good (resp. $\gamma$ is $l$-good). There is an \emph{attractiveness problem} on $(\alpha,\beta),(\gamma,\delta)$ if there exists at least one bad value.
\label{kbad}
\end{definition}
In other words we distinguish bad situations, where an addition of particles allows lower states to go over upper ones (or upper ones to go under lower ones) from good ones, where it cannot happen. Notice that Definition $\ref{kbad}$ involves addition of particles upon $\beta$ and subtraction of particles from $\gamma$. If we are interested in attractiveness problems coming from addition upon $\alpha$ and subtraction from $\delta$ we refer to $(\beta,\alpha),(\delta,\gamma)$.\\

\textit{We choose to define a coupling rate that moves both processes only if we are dealing with an attractiveness problem, otherwise we let the two processes evolve independently through uncoupled rates. Conditions $(\ref{C+})-(\ref{C-})$ do not involve configurations without an attractiveness problem, so a different construction for them does not change the result. Since $N$ is finite, we can construct the coupling by a downwards recursion on the number of particles involved in a transition}.\\

Our purpose now is to describe a coupling for $(\widetilde{\mathcal{S}}_{(x,y)},\mathcal{S}_{(x,y)})$, that we denote by $\mathcal{H}(x,y)$ (or simply $\mathcal{H}$), which will be increasing under Conditions $(\ref{C+})-(\ref{C-})$.\\

\emph{First of all we detail the construction on terms involving the larger number $N$ of particles and we prove that under Conditions (\ref{C+})--(\ref{C-}) none of these coupling terms breaks the partial order: this is the claim of Proposition \ref{imp}.}

\begin{remark}
By Hypothesis \ref{hip}, at least one of the terms $\Pi^{0,N}_{\alpha,\beta}$, $\Pi^{0,N}_{\gamma,\delta}$,  $\Pi^{-N,0}_{\alpha,\beta}$, $\Pi^{-N,0}_{\gamma,\delta}$, $\Gamma^{N}_{\alpha,\beta}$ and $\Gamma^{N}_{\gamma,\delta}$ is not null. We assume all these terms (and the smaller ones) positive. Otherwise the construction works in a similar way with some null terms.
\label{remN}
\end{remark} 
\begin{definition}
Let
\begin{equation}
\lL:=N-\delta+\beta; \qquad \lU:=N-\gamma+\alpha.
\label{lastLU}
\end{equation}
\end{definition}
If there is a lower attractiveness problem on $\beta$, then $\beta+N>\delta$ ($\lL>0$) and an addition of $N$ particles upon $\beta$ breaks the partial order. Such a problem comes both from birth ($\Pi^{0,N}_{\alpha,\beta}p$) and from jump ($\Gamma^{N}_{\alpha,\beta}p$) rates. If $K=1$, $j_1=N-\delta+\beta-1=\lL-1$ and $m_1=N$ then Condition (\ref{C+}) writes
\begin{equation}
\Pi_{\alpha, \beta}^{0,N}+\Gamma_{\alpha, \beta}^{N}\leq 
\sum_{l\geq \lL}(\Pi_{\gamma, \delta}^{0,l}+\Gamma_{\gamma, \delta}^{l}).
\label{condN}
\end{equation}
Notice that if $l\geq \lL$, then $\beta+N\leq \delta+l$, and additions of $N$ particles upon $\beta$ and of $l$ particles upon $\delta$ do not break the partial order on $y$. The construction consists in coupling the terms on the left hand side (which involve $N$ particles and break the partial order) to the ones on the right hand side (in such a way that the final configuration preserves the partial order on $y$ and on $x$) by following a basic coupling idea. We couple jumps on the lower configuration with jumps on the upper one and births on the lower configuration with births on the upper one. Only if this is not enough to solve the attractiveness problem, we mix births with jumps.\\   
If there is a higher attractiveness problem on $\gamma$, then  $\gamma-N<\alpha$ ($\lU>0$) and a subtraction of $N$ particles from $\gamma$ breaks the partial order. In this case the problem comes from $\Pi^{-N,0}_{\gamma,\delta}p$ and $\Gamma^{N}_{\gamma,\delta}p$; we use a symmetric construction starting from Condition (\ref{C-}) with $K=1$, $h_1=N-\gamma+\alpha-1=\lU-1$ and $m_1=N$:
\begin{equation}
\Pi_{\gamma,\delta}^{-N,0}+\Gamma_{\gamma,\delta}^{N}\leq 
\sum_{k\geq \lU}(\Pi_{\alpha, \beta}^{-k,0}+\Gamma_{\alpha, \beta}^{k}).
\label{condN-}
\end{equation}
%We suppose both $\beta$ and $\gamma$ are $N-$bad, if this is not the case one of them is good and the construction works in an easier way.\\
We denote by $H^{k,k,\cdot,\cdot}_{\alpha,\beta,\gamma,\delta}$ (resp. $H^{\cdot,\cdot,l,l}_{\alpha,\beta,\gamma,\delta}$) the coupling terms which involve jumps of $k$ (resp. $l$) particles from $x$ to $y$ on the lower (resp. upper) configuration; $H^{0,k,\cdot,\cdot}_{\alpha,\beta,\gamma,\delta}$ (resp. $H^{\cdot,\cdot,0,l}_{\alpha,\beta,\gamma,\delta}$) are the coupling terms concerning births of $k$ (resp. $l$) particles on $y$ on the lower (resp. upper) configuration and $H^{-k,0,\cdot,\cdot}_{\alpha,\beta,\gamma,\delta}$ (resp. $H^{\cdot,\cdot,-l,0}_{\alpha,\beta,\gamma,\delta}$) are the symmetric ones for death rates.\\
For instance $H^{k,k,0,l}_{\alpha,\beta,\gamma,\delta}$ combines the jump of $k$ particles from $x$ to $y$ on the lower configuration and the birth of $l$ particles on $y$ on the upper one.\\

The coupling construction takes place in three main steps.\\

\noindent \textbf{Step 1)} Suppose both $\beta$ and $\gamma$ are $N-$bad; if this is not the case one of them is good and the construction works in an easier way.\\
We begin with jump rates. We couple the lower configuration $N-$jump rate $\Gamma^{N}_{\alpha,\beta}p$ with jumps on the upper configuration. We first couple it with $\Gamma^{N}_{\gamma,\delta}p$, because $\alpha \leq \gamma$, $\alpha-N\leq \gamma-N$ and the final pairs of values are $(\alpha-N,\beta+N)\leq(\gamma-N,\delta+N)$. Let
\begin{align}
H^{N,N,N,N}_{\alpha,\beta,\gamma,\delta}:=(\Gamma^{N}_{\alpha,\beta}\wedge \Gamma^{N}_{\gamma,\delta})p.
\label{firstdiffusion}
\end{align} 
Then if the lower attractiveness problem is not solved, that is $H^{N,N,N,N}_{\alpha,\beta,\gamma,\delta}=\Gamma^{N}_{\gamma,\delta}p$, we have a remainder of the lower configuration jump rate that we couple with the upper configuration jump rate with the largest change of particles left, $N-1$. We go on by coupling the new remainder of $\Gamma^{N}_{\alpha,\beta}p$, if positive, with $\Gamma^{l}_{\gamma,\delta}p$ at $l^{th}$ step. The final pairs of values we reach are $(\alpha-N,\beta+N),(\gamma-l,\delta+l)$, which always preserve the partial order on $x$ since $\alpha-N\leq \gamma-l$ when $l\leq N$. The partial order on $y$ is preserved only if $\beta+N \leq \delta+l$, that is if $l\geq \lL$. For this reason we stop the coupling between jumps at step $\lL$: this is the meaning of formula (\ref{condN}).\\
More precisely, when $\Gamma^{N}_{\alpha,\beta}>\Gamma^{N}_{\gamma,\delta}$ and $N-1\geq \lL$, we get the second coupling rate
$$
H^{N,N,N-1,N-1}_{\alpha,\beta,\gamma,\delta}:=(\Gamma^{N}_{\alpha,\beta}p-H^{N,N,N,N}_{\alpha,\beta,\gamma,\delta})\wedge \Gamma^{N-1}_{\gamma,\delta}p=(\Gamma^{N}_{\alpha,\beta}-\Gamma^{N}_{\gamma,\delta})p\wedge \Gamma^{N-1}_{\gamma,\delta}p.
$$
Then either $H^{N,N,N-1,N-1}_{\alpha,\beta,\gamma,\delta}=(\Gamma^{N}_{\alpha,\beta}-\Gamma^{N}_{\gamma,\delta})p$, and with the two steps $l=N$ and $l=N-1$ we have no remaining part of $\Gamma^{N}_{\alpha,\beta}p$, or $H^{N,N,N-1,N-1}_{\alpha,\beta,\gamma,\delta}=\Gamma^{N-1}_{\gamma,\delta}p$ and we are left with a remainder $(\Gamma^{N}_{\alpha,\beta}-\Gamma^{N}_{\gamma,\delta}-\Gamma^{N-1}_{\gamma,\delta})p$, to go on with $l=N-2, \ldots$
Therefore we can define recursively starting from (\ref{firstdiffusion})
\begin{align}
H^{N,N,l,l}_{\alpha,\beta,\gamma,\delta}:=&(\Gamma^{N}_{\alpha,\beta}p-\sum_{l'>l}H^{N,N,l',l'}_{\alpha,\beta,\gamma,\delta})\wedge \Gamma^{l}_{\gamma,\delta}p=:J^{N,l+1}_{\alpha,\beta}\wedge \Gamma^{l}_{\gamma,\delta}p
\label{gammaN0}
\end{align} 
where \emph{$J^{N,l}_{\alpha,\beta}$ is the remainder of the jump (hence  notation $J$) rate  $\Gamma^{N}_{\alpha,\beta}p$ left over after the $l^{th}$ step} of the coupling construction; (\ref{gammaN0}) means that we couple the remainder from the $(l+1)^{th}$ step, $J^{N,l+1}_{\alpha,\beta}$, with $\Gamma^{l}_{\gamma,\delta}p$ at the  $l^{th}$ step.\\
We proceed this way until either we have no remainder of $\Gamma^{N}_{\alpha,\beta}p$, or we have reached $\lL$ with the remainder
\begin{equation}
J^{N,\lL}_{\alpha,\beta}=\Gamma^{N}_{\alpha,\beta}p-\sum_{l\geq \lL}H^{N,N,l,l}_{\alpha,\beta,\gamma,\delta}=\Gamma_{\alpha,\beta}^{N}p-\sum_{l\geq \lL}\Gamma_{\gamma, \delta}^{l}p>0
\label{remainder}
\end{equation} 
Note that, since $\beta+N\leq \delta+l$ for $l\geq \lL$, none of the coupled transitions until this point have broken the partial ordering. Proceeding until $\lL-1$ would break the partial ordering. Therefore we stop at $\lL$ the construction in Step $1$ and we will couple the remainder $J^{N,\lL}_{\alpha,\beta}$ with upper birth rates at Step $3a$.\\

Step $1$ is detailed in Tables $\ref{tabNd1}$ and $\ref{tabNd2}$, where $N^{d+}$ corresponds to the first $l$ (going downwards from $N$) such that the minimum in (\ref{gammaN0}) is $J^{N,l+1}_{\alpha,\beta}$.
 
\begin{table}[htb]\caption{$N$ jump rate, $N^{d+}\geq \lL$ ($J^{N,\lL}_{\alpha,\beta}=0$)}\label{tabNd1}
\begin{center}
\begin{tabular}{|l|ccc|cc|}
\cline{1-6} 
&&&&&\\
$l$              &$J_{\alpha, \beta}^{N,l+1}$& $\wedge$ &$\Gamma_{\gamma, \delta}^{l}p$&=& $H^{N,N,l,l}_{\alpha,\beta,\gamma,\delta}$\\
\cline{1-6} 
&&&&&\\
$N$&             &                        &$\times$&& \\
\vdots           &                        &&\vdots&& $\Gamma^{l}_{\gamma,\delta}p$\\
$N^{d+}+1$       &                        &&$\times$&&\\
\cline{2-6}
&                &                        &&&\\
$N^{d+}$         &$\times$                &&& &$\displaystyle \Gamma^{N}_{\alpha,\beta}p-\sum_{l'>N^{d+}}\Gamma^{l'}_{\gamma,\delta}p\geq 0$ \\
\cline{2-6}
$N^{d+}-1$       &$\times$&&&&\\
\vdots           &\vdots&&&&$0$\\
$\lL$ &$\times$&&&&\\
\cline{1-6}
\end{tabular}
\end{center}
\end{table}
\begin{table}[htb]\caption{$N$ jump rate, $N^{d+}=\lL-1$ ($J^{N,\lL}_{\alpha,\beta}> 0$)}\label{tabNd2}
\begin{center}
\begin{tabular}{|l|ccc|cc|}
\cline{1-6}
&&&&&\\
$l$&$J_{\alpha, \beta}^{N,l+1}$& $\wedge$ &$\Gamma_{\gamma, \delta}^{l}p$&=& $H^{N,N,l,l}_{\alpha,\beta,\gamma,\delta}$\\
\cline{1-6} 
&&&&&\\
$N$ & &&$\times$&& \\
\vdots&&&\vdots &&$\Gamma^{l}_{\gamma,\delta}p$\\
$\lL$&&&$\times$&&\\
\cline{1-6}
\end{tabular}
\end{center}
\end{table}
We have to distinguish between two situations:\\
$\bullet$ if the minimum given by $H^{N,N,l,l}_{\alpha,\beta,\gamma,\delta}$ is always the second term, we do not reach $N^{d+}$ at $\lL$ yet. Since we have decided to stop the coupling construction at step $\lL$, we need to couple with birth rates (Step $3a$) in order to solve the attractiveness problem (Table $\ref{tabNd2}$ when $J^{N,\lL}_{\alpha,\beta}>0$) and we put $N^{d+}=\lL-1$.\\
$\bullet$ if $N^{d+}\geq \lL$ then $H^{N,N,N^{d+}-1,N^{d+}-1}_{\alpha,\beta,\gamma,\delta}=0$ by (\ref{gammaN0}). Therefore there is no need to continue a coupling involving $\Gamma^{N}_{\alpha,\beta}p$, since the attractiveness problem is solved. In this case $H^{N,N,l,l}_{\alpha,\beta,\gamma,\delta}=0$ for $\lL\leq l<N^{d+}$ by definition, we do not need Step $3a$ and we define
\begin{equation}
H^{N,N,0,l}_{\alpha,\beta,\gamma,\delta}:=0 \qquad \text{ for each } l>0.
\label{a27}
\end{equation}
In both cases, we define
\begin{equation}
 H^{N,N,l,l}_{\alpha,\beta,\gamma,\delta}:=0 \qquad \text{ for } 0<l<\lL.
\label{a69}
\end{equation}
since these terms could break the partial ordering. Moreover, since $\beta$ is $N-$bad, we define
\begin{equation}
H^{N,N,-l,0}_{\alpha,\beta,\gamma,\delta}:=0 \quad \text{ for each }l>0.
\label{HNN-l0}
\end{equation}
Notice that such terms are not a priori null. Indeed if $\gamma$ is $l-$bad and $\beta$ is $N-$good one can use this coupling term to solve the attractiveness problem of $\Pi^{-l,0}_{\gamma,\delta}p$.\\
 
If there is a higher attractiveness problem, we repeat the same construction for the coupling terms involving the jump rate $\Gamma^N_{\gamma,\delta}p$ starting from (\ref{firstdiffusion}) and we define a value $N^{d-}$ analogous to the previous $N^{d+}$. The recursive formula symmetric to (\ref{gammaN0}) if $\Gamma^{N}_{\gamma,\delta}>\Gamma^{N}_{\alpha,\beta}$ is 
\begin{align}
H^{k,k,N,N}_{\alpha,\beta,\gamma,\delta}:=&\Gamma^{k}_{\alpha,\beta}p\wedge (\Gamma^{N}_{\gamma,\delta}p-\sum_{k'>k}H^{k',k',N,N}_{\alpha,\beta,\gamma,\delta})=:\Gamma^{k}_{\alpha,\beta}p\wedge J^{N,k+1}_{\gamma,\delta}
\label{gammaNsim}
\end{align}
where \emph{$J^{k,N}_{\gamma,\delta}$ represents the remainder of the jump rate  $\Gamma^{N}_{\gamma,\delta}p$ left over after the $k^{th}$ step of the coupling construction}. We need to couple the remainder of $\Gamma^{N}_{\gamma,\delta}p$ with the lower death rates in Step $3a$ if
\begin{equation}
J^{\lU,N}_{\gamma,\delta}=\Gamma_{\gamma,\delta}^{N}p-\sum_{k\geq \lU}H^{k,k,N,N}_{\alpha, \beta,\gamma,\delta}=\Gamma_{\gamma,\delta}^{N}p-\sum_{k\geq \lU}\Gamma_{\alpha, \beta}^{k}p>0.
\label{remainder_sim}
\end{equation}
In this case we put $N^{d-}=\lU-1$.  
By symmetry we define $H^{k,k,N,N}_{\alpha,\beta,\gamma,\delta}=0$ for each $0<k< \lU$, $H^{-k,0,N,N}_{\alpha,\beta,\gamma,\delta}=0$ for each $k>0$ if $N^{d-}\geq \lU$ and $H^{0,k,N,N}_{\alpha,\beta,\gamma,\delta}=0$ for each $k>0$ if $\gamma$ is $N-$bad.
\begin{remark}
By (\ref{firstdiffusion}), either $N^{d+}=N$ or $N^{d-}=N$, that is either the lower or the higher attractiveness problem given by the jump of $N$ particles is solved by the first coupling term. 
\label{rem_Nd}
\end{remark}
If there is no lower (higher) attractiveness problem we put $N^{d+}=N+1$ ($N^{d-}=N+1$).\\
If either $\beta$ or $\gamma$ is $N-$good, we use only one of the previous constructions. Suppose for instance that $\gamma$ is $N-$good: then the construction involving $\Gamma^N_{\alpha,\beta}p$ works in the same way, but the symmetric one is not required and we use the coupling terms $H^{\cdot,\cdot,N,N}_{\alpha,\beta,\gamma,\delta}$ only to solve the lower attractiveness problems induced by either $\Gamma^{N}_{\alpha,\beta}p$ (at Step $1$) or by $\Pi^{0,N}_{\alpha,\beta}p$ (at Step $3b$). Therefore $H^{0,N,N,N}_{\alpha,\beta,\gamma,\delta}$ (defined at Step $3a$) might be non null, but we define
\begin{equation}
H^{-l,0,N,N}_{\alpha,\beta,\gamma,\delta}:=0 \text{ for each }l>0.
\label{good1}
\end{equation}
If $\beta$ is $N-$good a symmetric remark holds.\\

\noindent \textbf{Step 2)} Suppose $\beta$ is $N-$bad. The birth rate $\Pi^{0,N}_{\alpha,\beta}p$ could break the partial order on $\beta$. We work as in Step $1$ and we begin with the coupling term
\begin{align}
H^{0,N,0,N}_{\alpha,\beta,\gamma,\delta}:=&(\Pi^{0,N}_{\alpha,\beta}\wedge \Pi^{0,N}_{\gamma,\delta})p.
\label{firstbirths}
\end{align} 
If the attractiveness problem is not solved we couple the remainder of $\Pi^{0,N}_{\alpha,\beta}p$ with the birth rate of the upper configuration with the largest change of particles, $\Pi^{0,N-1}_{\gamma,\delta}p$ and going down we couple it with $\Pi^{0,l}_{\gamma,\delta}p$ at $l^{th}$-step, until $l=\lL$. We define recursively starting from (\ref{firstbirths}) the terms
\begin{align}
H^{0,N,0,l}_{\alpha,\beta,\gamma,\delta}:=&(\Pi^{0,N}_{\alpha,\beta}p-\sum_{l'>l}H^{0,N,0,l'}_{\alpha,\beta,\gamma,\delta})\wedge \Pi^{0,l}_{\gamma,\delta}p=: B^{N,l+1}_{\alpha,\beta}\wedge \Pi^{0,l}_{\gamma,\delta}p
\label{PiN}
\end{align}
where \emph{$B^{N,l}_{\alpha,\beta}$ is the remainder of the birth (hence notation $B$) rate  $\Pi^{0,N}_{\alpha,\beta}p$ left over after the $l^{th}$ step of the coupling construction}.\\
We proceed as in Step $1$ for transitions involving births: while the minimum in (\ref{PiN}) is the second term we go on downwards in $l$. As soon as the minimum is the first term, we have no remainder of $\Pi^{0,N}_{\alpha,\beta}p$, so the lower attractiveness problem is solved and we define $N^{B}$ to be this first such $l$. 
Otherwise we have reached $\lL$ with the remainder
\begin{equation}
B^{N,\lL}_{\alpha,\beta}=\Pi^{0,N}_{\alpha,\beta}p-\sum_{l\geq \lL}H^{0,N,0,l}_{\alpha,\beta,\gamma,\delta}=\Pi_{\alpha,\beta}^{0,N}p-\sum_{l\geq \lL}\Pi_{\gamma, \delta}^{0,l}p>0
\label{remainder2}
\end{equation} 
We have to distinguish between two situations:\\
$\bullet$ if the minimum given by $H^{0,N,0,l}_{\alpha,\beta,\gamma,\delta}$ is always the second term, we do not reach $N^{B}$ at $\lL$ yet. Since we stop the coupling construction at step $\lL$, we need to couple the remainder with jump rates (Step $3b$) in order to solve the attractiveness problem (Table \ref{tabNB2} when $B^{N,\lL}_{\alpha,\beta}>0$) and we put $N^{B}=\lL-1$;\\
$\bullet$ if $N^{B}\geq \lL$ then $B^{N,\lL}_{\alpha,\beta}=0$ (see Table \ref{tabNB1}). Therefore there is no need to continue a coupling involving $\Pi^{0,N}_{\alpha,\beta}p$, since the attractiveness problem is solved. In this case $H^{0,N,0,l}_{\alpha,\beta,\gamma,\delta}=0$ for $\lL\leq l<N^{B}$ by definition, we do not need Step $3b$ and we define
\begin{equation}
H^{0,N,l,l}_{\alpha,\beta,\gamma,\delta}:=0 \qquad \text{ for each } l>0,
\label{a68bis}
\end{equation}
In both cases, we define
\begin{equation}
 H^{0,N,0,l}_{\alpha,\beta,\gamma,\delta}:=0 \qquad \text{ for } 0<l<\lL.
\label{a69bis}
\end{equation}

\begin{table}[htb]\caption{$N$ birth rate, $N^{B}\geq \lL$ ($B^{N,\lL}_{\alpha,\beta}=0$)}\label{tabNB1}
\begin{center}
\begin{tabular}{|l|ccc|cc|}
\cline{1-6}
&&&&&\\
$l$&$B^{N,l+1}_{\alpha,\beta}$& $\wedge$ &$\Pi^{0,l}_{\gamma,\delta}p$&=&$H^{0,N,0,l}_{\alpha,\beta,\gamma,\delta}$\\
\cline{1-6} 
$N$& &&$\times$&& \\
\vdots&&&\vdots&& $\Pi^{0,l}_{\gamma,\delta}p$\\
$N^{B}+1$&&&$\times$&&\\
\cline{2-6}
&&&&&\\
$N^{B}$&$\times$&&& &  $ \displaystyle \Pi^{0,N}_{\alpha,\beta}p-\sum_{l'>N^{B}}\Pi^{0,l'}_{\gamma,\delta}p\geq 0$ \\
\cline{2-6}
$N^{B}-1$&$\times$&&&&\\
\vdots &\vdots&&&& $0$\\
$\lL$&$\times$&&&&\\
\cline{1-6}
\end{tabular}
\end{center}
\end{table}
\begin{table}[htb]\caption{$N$ birth rate, $N^{B}=\lL-1$ ($B^{N,\lL}_{\alpha,\beta}> 0$)}\label{tabNB2}
\begin{center}
\begin{tabular}{|l|ccc|cc|}
\cline{1-6}
&&&&&\\
$l$&$B_{\alpha, \beta}^{N,l+1}$& $\wedge$ &$\Pi_{\gamma, \delta}^{0,l}p$&=& $H^{0,N,0,l}_{\alpha,\beta,\gamma,\delta}$\\
\cline{1-6} 
&&&&&\\
$N$ & &&$\times$&& \\
\vdots&&&\vdots &&$\Pi^{0,l}_{\gamma,\delta}p$\\
$\lL$&&&$\times$&&\\
\cline{1-6}
\end{tabular}
\end{center}
\end{table}
The attractiveness problem coming from either birth or death rates is solved. Indeed
\begin{lemma}
Under Condition (\ref{C+}),
$$
N^B \vee N^{d+}\geq \lL.
$$
\label{newthird}
\end{lemma} 
\begin{proof}.
If this is not the case, by (\ref{gammaNsim}), Table \ref{tabNd1} and definition of $N^{d+}$
\begin{equation}
\Gamma^{N}_{\alpha,\beta}p-\sum_{l'>\lL}\Gamma^{l'}_{\gamma,\delta}p=J_{\alpha,\beta}^{N,\lL+1}>\Gamma_{\gamma, \delta}^{\lL}p.
\label{new_step1}
\end{equation}
By (\ref{PiN}), Table \ref{tabNB1} and definition of $N^{B}$
\begin{equation}
\Pi^{0,N}_{\alpha,\beta}p-\sum_{l'>\lL}\Pi^{0,l'}_{\gamma,\delta}p=B_{\alpha,\beta}^{N,\lL+1}>\Pi_{\gamma, \delta}^{0,\lL}p.
\label{new_step2}
\end{equation}
The sum of (\ref{new_step1}) and (\ref{new_step2}) contradicts (\ref{condN}).
\end{proof}
If $\beta$ is $N-$good, then $H^{0,N,\cdot,\cdot}_{\alpha,\beta,\gamma,\delta}=0$ and we define $N^{B}=N^{d+}=N+1$.
\begin{remark}
If $\gamma$ is $N-$bad, the construction involving death rates works in a similar way with symmetric definitions of $N^{D}$ and of a remainder $D^{k,N}_{\gamma,\delta}$. As in Lemma \ref{newthird}, under Condition (\ref{C-}),
\begin{equation}
N^{D}\vee N^{d-}\geq \lU.
\label{newthirdmeno}
\end{equation}
If $\gamma$ is $N-$good, then $H^{\cdot,\cdot,-N,0}_{\alpha,\beta,\gamma,\delta}=0$ and we define $N^{D}=N^{d-}=N+1$.
\end{remark}

\noindent \textbf{Step 3a)} Suppose $\beta$ is $N-$bad and $N^{d+}=\lL-1$. We come back to Step $1$, where even if the remaining part of $\Gamma^{N}_{\alpha,\beta}p$ was still positive at step $\lL$, we decided to stop. We refer to Table $\ref{tabN3a}$.\\
By Lemma $\ref{newthird}$, $N^B\geq \lL$.  We use the upper configuration birth rate remaining from Step $2$ in order to solve the attractiveness problem: we couple the remainder from Step $1$ of $\Gamma^{N}_{\alpha,\beta}p$ with the remainder from Step $2$ of $\Pi^{0,N}_{\gamma,\delta}p$, 
$$
H^{N,N,0,N}_{\alpha,\beta,\gamma,\delta}:=J^{N,\lL}_{\alpha,\beta}\wedge [\Pi^{0,N}_{\gamma,\delta}p-\Pi^{0,N}_{\gamma,\delta}p\wedge \Pi^{0,N}_{\alpha,\beta}p].
$$
Then, if the minimum is the second term we proceed downwards in $l$ with terms 
\begin{align}
H^{N,N,0,l}_{\alpha,\beta,\gamma,\delta}:=&[J^{N,\lL}_{\alpha,\beta}-\sum_{l'>l}H^{N,N,0,l'}_{\alpha,\beta,\gamma,\delta}]\wedge [\Pi^{0,l}_{\gamma,\delta}p-\Pi^{0,l}_{\gamma,\delta}p\wedge B^{N,l+1}_{\alpha,\beta}]\nonumber\\
=:&\mathcal{J}^{N,l+1}_{\alpha,\beta}\wedge [\Pi^{0,l}_{\gamma,\delta}p-\Pi^{0,l}_{\gamma,\delta}p\wedge B^{N,l+1}_{\alpha,\beta}].
\label{a540}
\end{align}
where \emph{$\mathcal{J}^{N,l}_{\alpha,\beta}$ is the remainder of the jump ($\mathcal{J}$) rate  $\Gamma^{N}_{\alpha,\beta}p$ left over after $l^{th}$ step}. Notice that $\mathcal{J}^{N,N+1}_{\alpha,\beta}=J^{N,\lL}_{\alpha,\beta}$.\\
We proceed with (\ref{a540}) until the minimum is the first term, in which case there is no remainder of $\Gamma^{N}_{\alpha,\beta}p$, so the attractiveness problem is solved, and we define $N^{dB}$ to be the first such $l$. In other words when $l>N^{dB}$, the coupling term is the remainder of the upper configuration $l$-birth rate, when $l=N^{dB}$ it is the remainder of the lower configuration $N$--jump rate, and when $l<N^{dB}$ the coupling terms are null. 
\begin{remark}
If $l>N^{dB}$ then the minimum of (\ref{a540}) is the second term, which depends on Step $1$:\\
-if $l>N^{B}$, by Table \ref{tabNB1}, $B^{N,l+1}_{\alpha,\beta}\wedge \Pi^{0,l}_{\gamma,\delta}p= \Pi^{0,l}_{\gamma,\delta}p$ and by (\ref{a540}), $H^{N,N,0,l}_{\alpha,\beta,\gamma,\delta}=0;$\\
-if $l=N^{B}$, by (\ref{a540}), $H^{N,N,0,l}_{\alpha,\beta,\gamma,\delta}=\Pi^{0,l}_{\gamma,\delta}p-B^{N,l+1}_{\alpha,\beta}=\sum_{l'\geq l}\Pi^{0,l'}_{\gamma,\delta}p-\Pi^{0,N}_{\alpha,\beta}p;$\\
-if $l<N^{B}$ then $B^{N,l+1}_{\alpha,\beta}\wedge \Pi^{0,l}_{\gamma,\delta}p=B^{N,l+1}_{\alpha,\beta}=0$ and 
$H^{N,N,0,l}_{\alpha,\beta,\gamma,\delta}=\Pi^{0,l}_{\gamma,\delta}p.$\\
In other words even if the minimum is the second term in $H^{N,N,0,l}_{\alpha,\beta,\gamma,\delta}$, it could be null, when solving the lower attractiveness problem left no remainder of $\Pi^{0,l}_{\gamma,\delta}p$. It means that positive coupling terms begin below $N^{B}$.
\label{remthird}
\end{remark}
If $J^{N,\lL}_{\alpha,\beta}=0$, that is if $N^{d+}\geq \lL$, we put $N^{dB}=N+1$. We give it the same value if $\beta$ is $N-$good: in these cases $H^{N,N,0,l}_{\alpha,\beta,\gamma,\delta}=0$ for each $l>0$.

\begin{table}[htb]\caption{Third step, $N^{d+}=\lL-1$ ($J^{N,\lL}_{\alpha,\beta}>0$)}\label{tabN3a}
\begin{center}
\begin{tabular}{|l|ccc|cc|}
\cline{1-6}
&&&&&\\
$l$& $\mathcal{J}^{N,l+1}_{\alpha,\beta}$ & $\wedge$ & $[\Pi^{0,l}_{\gamma,\delta}p-\Pi^{0,l}_{\gamma,\delta}p\wedge B^{N,l+1}_{\alpha,\beta}]$& = & $H^{N,N,0,l}_{\alpha,\beta,\gamma,\delta}$\\
\cline{1-6} 
$N$& &&$\times$& &\\
\vdots&&&\vdots& &$\Pi^{0,l}_{\gamma,\delta}p-\Pi^{0,l}_{\gamma,\delta}p\wedge B^{N,l+1}_{\alpha,\beta}$\\
$N^{dB}+1$& & &$\times$ &&\\
\cline{2-6}
&&&&&\\
$N^{dB}$&$\times$&&& &$\mathcal{J}^{N,l+1}_{\alpha,\beta}\geq 0$ \\
\cline{2-6}
$N^{dB}-1$&$\times$&&&&\\
\vdots &\vdots&&& &$\mathcal{J}^{N,l+1}_{\alpha,\beta}=0$\\
0&$\times$&&&&\\
\cline{1-6}
\end{tabular}
\end{center}
\end{table}

If $\gamma$ is $N-$bad and $N^{d+}=\lL-1$ then $N^{d-}=N$ by Remark \ref{rem_Nd} and we do not need Step $3a$ for the upper jump rate. If $N^{d-}=\lU-1$ (in this case $N^{d+}=N$), the construction works in a symmetric way, by coupling the remainder of the upper configuration jump rate $\mathcal{J}^{\lU,N}_{\gamma,\delta}$ defined in a similar way with the remainder of the death rate from Step $2$. If $\gamma$ is $N-$good we put $N^{dD}=N+1$: in these cases $H^{-k,0,N,N}_{\alpha,\beta,\gamma,\delta}=0$ for each $k>0$.\\

\noindent \textbf{Step 3b)} Suppose $\beta$ is $N-$bad and $N^{B}=\lL-1$. We come back to Step $2$, where even if the remainder of $\Pi^{0,N}_{\alpha,\beta}p$ was still positive at step $\lL$, we decided to stop. We refer to Table $\ref{tabN3b}$. By Lemma $\ref{newthird}$, $N^{d+}\geq \lL$. We cannot couple $B^{N,\lL}_{\alpha,\beta}$ with the upper configuration jump rates remaining from Step $2$ with $l>\gamma-\alpha$, because the final states we would reach are $(\alpha,\beta+N),(\gamma-l,\delta+l)$, which break the partial order. Therefore we put
\begin{equation}
H^{0,N,l,l}_{\alpha,\beta,\gamma,\delta}:=0 \text{ for each }l>\gamma-\alpha
\label{0Nll>}
\end{equation}
Then we couple the remainder $B^{N,\lL}_{\alpha,\beta}$ from Step $2$ of $\Pi^{0,N}_{\alpha,\beta}p$ with the upper configuration jump rates remaining from Step $1$ that do not break the partial order in $\alpha$ to solve the attractiveness problem, that is $\Gamma^{l}_{\gamma,\delta}p$ with $l\leq \gamma-\alpha$, 
\begin{equation}
H^{0,N,l,l}_{\alpha,\beta,\gamma,\delta}:=B^{N,\lL}_{\alpha,\beta}\wedge [\Gamma^{l}_{\gamma,\delta}p-J^{N,l+1}_{\alpha,\beta}\wedge \Gamma^{l}_{\gamma,\delta}p]\quad \text{ with }l=\gamma-\alpha,
\label{temp30}
\end{equation}
if the minimum is the second term we proceed downwards in $l$ with $\Gamma^{l}_{\gamma,\delta}p$ with $l<\gamma-\alpha$
\begin{align}
H^{0,N,l,l}_{\alpha,\beta,\gamma,\delta}:=&[B^{N,\lL}_{\alpha,\beta}-\sum_{\gamma-\alpha\geq l'>l}H^{0,N,l',l'}_{\alpha,\beta,\gamma,\delta}]\wedge [\Gamma^{l}_{\gamma,\delta}p-J^{N,l+1}_{\alpha,\beta}\wedge \Gamma^{l}_{\gamma,\delta}p]\nonumber \\
=:&\mathcal{B}^{N,l+1}_{\alpha,\beta}\wedge [\Gamma^{l}_{\gamma,\delta}p-J^{N,l+1}_{\alpha,\beta}\wedge \Gamma^{l}_{\gamma,\delta}p]
\label{temp3}
\end{align}
where \emph{$\mathcal{B}^{N,l}_{\alpha,\beta}$ is the remainder of the birth ($\mathcal{B}$) rate  $\Pi^{0,N}_{\alpha,\beta}p$ left over after step $l^{th}$ in Step $3b$}. Notice that $\mathcal{B}^{N,N+1}_{\alpha,\beta}=B^{N,\lL}_{\alpha,\beta}$.\\
We proceed with (\ref{temp3}) until the minimum is the first term, in which case there is no remainder of $\Pi^{0,N}_{\alpha,\beta}p$, so the attractiveness problem is solved, and we define $N^{Bd}$ to be the first such $l$. In other words when $l>N^{Bd}$, the coupling term is the remainder of the upper configuration $l$-jump rate, when $l=N^{Bd}$ it is the remainder of the lower configuration $N$-birth rate, and when $l<N^{Bd}$ the coupling terms are null.\\
\begin{remark}
If $l>N^{Bd}$ then the minimum in (\ref{temp3}) is the second term, which depends on 
Step $2$:\\
$\bullet$ if $l>N^{d+}$ then $J^{N,l+1}_{\alpha,\beta}\wedge \Gamma^{l}_{\gamma,\delta}p= \Gamma^{l}_{\gamma,\delta}p$ by Table \ref{tabNd1} and 
\begin{equation}
H^{0,N,l,l}_{\alpha,\beta,\gamma,\delta}=0;
\label{p30}
\end{equation}
$\bullet$ if $l=N^{d+}\geq \lL$ (see Table $\ref{tabNd1}$), then $J^{N,l+1}_{\alpha,\beta}\wedge \Gamma^{l}_{\gamma,\delta}p= J^{N,l+1}_{\alpha,\beta}$ and 
\begin{equation}
H^{0,N,l,l}_{\alpha,\beta,\gamma,\delta}=\Gamma^{l}_{\gamma,\delta}p-J^{N,l+1}_{\alpha,\beta}=\sum_{l'\geq N^{d+}}\Gamma^{l'}_{\gamma,\delta}p-\Gamma^{N}_{\alpha,\beta}p;
\label{tempast}
\end{equation}
$\bullet$ if $l<N^{d+}$ then $J^{N,l+1}_{\alpha,\beta}\wedge \Gamma^{l}_{\gamma,\delta}p=J^{N,l+1}_{\alpha,\beta}=0$ by Table \ref{tabNd1}; hence 
$$
H^{0,N,l,l}_{\alpha,\beta,\gamma,\delta}=\Gamma^{l}_{\gamma,\delta}p \qquad \text{ if }0<l<N^{d+}.
$$
In other words if there is no remainder of $\Gamma^{l}_{\gamma,\delta}p$ after Step $2$, the coupling term $H^{0,N,l,l}_{\alpha,\beta,\gamma,\delta}$ is null even if the attractiveness problem is not solved yet. It means that positive coupling terms begin below $N^{d+}$.
\label{remsecond}
\end{remark}
If $B^{N,\lL}_{\alpha,\beta}=0$, that is if $N^{B}\geq \lL$, we put $N^{Bd}=N+1$. We give it the same value if $\beta$ is $N-$good: in these cases $H^{0,N,l,l}_{\alpha,\beta,\gamma,\delta}=0$ for each $l>0$. We refer to Table $\ref{tabN3b}$.

\begin{table}[htb]\caption{Third step, $N^B=\lL-1$ ($B^{N,\lL}_{\alpha,\beta}>0$)}\label{tabN3b}
\begin{center}
\begin{tabular}{|l|ccc|cc|}
\cline{1-6}
&&&&&\\
$l$& $\mathcal{B}^{N,l+1}_{\alpha,\beta}$ & $\wedge$ & $[\Gamma^{l}_{\gamma,\delta}p-J^{N,l+1}_{\alpha,\beta}\wedge \Gamma^{l}_{\gamma,\delta}p]$& = & $H^{0,N,l,l}_{\alpha,\beta,\gamma,\delta}$\\
\cline{1-6} 
$\gamma-\alpha$& &&$\times$& &\\
\vdots&&&\vdots& &$\Gamma^{l}_{\gamma,\delta}p-J^{N,l+1}_{\alpha,\beta}\wedge \Gamma^{l}_{\gamma,\delta}p$\\
$N^{Bd}+1$& & &$\times$ &&\\
\cline{2-6}
&&&&&\\
$N^{Bd}$&$\times$&&& &$\mathcal{B}^{N,l+1}_{\alpha,\beta}\geq 0$ \\
\cline{2-6}
$N^{Bd}-1$&$\times$&&&&\\
\vdots &\vdots&&& &$\mathcal{B}^{N,l+1}_{\alpha,\beta}=0$\\
0&$\times$&&&&\\
\cline{1-6}
\end{tabular}
\end{center}
\end{table}
\begin{remark}
In Step 3 we do not impose any restriction as in previous steps, the pairs of values $(\alpha-N,\beta+N),(\gamma,\delta+l)$ or $(\alpha,\beta+N),(\gamma-l,\delta+l)$ breaking the partial order on $\beta$ (that is such that $\beta+N>\delta+l$) that we could reach a priori will be avoided by attractiveness (see Proposition \ref{imp}).
\label{restr_attr}
\end{remark}
If $\gamma$ is $N-$bad and $N^{D}=\lU-1$ then the symmetric construction allows to define the index $N^{Dd}$ of the last positive coupling term and the remainder $\mathcal{D}^{k,N}_{\gamma,\delta}$ of the death rate  $\Pi^{-N,0}_{\gamma,\delta}p$ after Step $1$ and $k$ recursions in Step $3b$. Remark \ref{restr_attr} works in a symmetric way. If $\gamma$ is $N-$good we put $N^{Dd}=N+1$ and $H^{k,k,-N,0}_{\alpha,\beta,\gamma,\delta}=0$ for each $k>0$.
\begin{remark}
In the indexes we have defined, the superscript $B$ means birth, $D$ death, $d^{+}$ diffusion (jump) on the lower configuration, $d^{-}$ diffusion (jump) on the upper configuration, $Bd$ birth and diffusion, $dB$ diffusion and birth, $Dd$ death and diffusion, $dD$ diffusion and death.
\end{remark}

Once we constructed the coupling rates involving $\Gamma^{N}_{\alpha,\beta}p$, $\Pi^{0,N}_{\alpha,\beta}p$, $\Gamma^{N}_{\gamma,\delta}p$ and $\Pi^{-N,0}_{\gamma,\delta}p$ we move to rates involving less than $N$ particles. If $\beta$ is $(N-1)$-bad (and/or $\gamma$ is $(N-1)$-bad) in order to solve the lower (higher) attractiveness problems given by $\Gamma_{\alpha,\beta}^{N-1}p$ and $\Pi^{0,N-1}_{\alpha,\beta}p$ ($\Gamma_{\gamma,\delta}^{N-1}p$ and $\Pi_{\gamma,\delta}^{-(N-1),0}p$) (or their remainders from previous steps) we repeat Step $1$, Step $2$ and Step $3$ for such terms. We proceed this way with the remainders of $\Gamma_{\alpha,\beta}^{k}p$ and $\Pi^{0,k}_{\alpha,\beta}p$ ($\Gamma_{\gamma,\delta}^{l}p$ and $\Pi_{\gamma,\delta}^{-l,0}p$) going downwards with respect to $k$ ($l$) until $\beta$ is $k$-good ($\gamma$ is $l$-good).\\

\textit{We couple an upper configuration jump rate $\Gamma^{l}_{\gamma,\delta}p$ with a lower configuration birth rate $\Pi^{0,k}_{\alpha,\beta}p$ if $\beta+k>l$ and $\alpha\leq \gamma-l$ in order to solve a lower attractiveness problem, and with a lower configuration death rate $\Pi^{-k,0}_{\alpha,\beta}p$ if $\alpha> \gamma-l$ in order to solve a higher attractiveness problem: we cannot couple the same higher jump rate both with lower birth and death rates. A symmetric remark holds for lower jump rates}.\\

Coupling rates are given below by downwards recursive relations
\begin{definition}
Given $k\leq N$ and $l\leq N$, the possible non null coupling terms are given by
$$
\begin{array}{ll}
\bullet \text{Step 1 }& \nonumber \\
& J^{k,l+1}_{\alpha,\beta}:=\displaystyle \Gamma^{k}_{\alpha,\beta}p-\sum_{l'>l}H^{k,k,l',l'}_{\alpha,\beta,\gamma,\delta}-\sum_{l'>l}H^{k,k,-l',0}_{\alpha,\beta,\gamma,\delta}\\
%\label{Jkl1}\\ 
&J^{k+1,l}_{\gamma,\delta}:=\displaystyle \Gamma^{l}_{\gamma,\delta}p-\sum_{k'>k}H^{k',k',l,l}_{\alpha,\beta,\gamma,\delta}-\sum_{k'>k}H^{0,k',l,l}_{\alpha,\beta,\gamma,\delta}\\
%\label{Jkl2}\\
&H^{k,k,l,l}_{\alpha,\beta,\gamma,\delta}=J^{k,l+1}_{\alpha,\beta}\wedge J^{k+1,l}_{\gamma,\delta}
\qquad \text{ if }(\alpha>\gamma-l \text{ or } \beta+k>\delta) \text{ and }\nonumber \\
&\hspace{5.5 cm}(\beta+k\leq \delta+l)\text{ and }(\alpha-k \leq \gamma-l)
\label{HJ}\\
\bullet\text{Step 2 }&\text{(births)}\nonumber\\
& B^{k,l+1}_{\alpha,\beta}:=\displaystyle \Pi^{0,k}_{\alpha,\beta}p-\sum_{l'>l}H^{0,k,0,l'}_{\alpha,\beta,\gamma,\delta}\\
%\label{Bkl1}\\
&B^{k+1,l}_{\gamma,\delta}:=\displaystyle \Pi^{0,l}_{\gamma,\delta}p-\sum_{k'>k}H^{0,k',0,l}_{\alpha,\beta,\gamma,\delta}-\sum_{k'>k}H^{k',k',0,l}_{\alpha,\beta,\gamma,\delta}\\
%\label{Bkl2}\\
&H^{0,k,0,l}_{\alpha,\beta,\gamma,\delta}=B^{k,l+1}_{\alpha,\beta}\wedge B^{k+1,l}_{\gamma,\delta} \qquad \text{ if }\delta+l\geq \beta+k>\delta
\label{HB1}\\
\bullet\text{ Step 2 }&\text{(deaths)} \nonumber\\
&D^{k,l+1}_{\alpha,\beta}:=\displaystyle \Pi^{-k,0}_{\alpha,\beta}p-\sum_{l'> l}H^{-k,0,-l',0}_{\alpha,\beta,\gamma,\delta}-\sum_{l'>l}H^{-k,0,l',l'}_{\alpha,\beta,\gamma,\delta}\\ 
%\label{Dkl1}\\
&D^{k+1,l}_{\gamma,\delta}:= \displaystyle \Pi^{-l,0}_{\gamma,\delta}p-\sum_{k'>k}H^{-k',0,-l,0}_{\alpha,\beta,\gamma,\delta} \\
%\label{Dkl2}\\
&H^{-k,0,-l,0}_{\alpha,\beta,\gamma,\delta}=D^{k,l+1}_{\alpha,\beta}\wedge D^{k+1,l}_{\gamma,\delta}\qquad \text{if }\alpha>\gamma-l\geq \alpha-k
\label{HD1}\\
\nonumber\\
\bullet\text{Step 3a)}&\text{(births)}\nonumber \\
&\mathcal{J}^{k,l+1}_{\alpha,\beta}=\displaystyle\Gamma^{k}_{\alpha,\beta}p-\sum_{l'\geq k-\delta+\beta}H^{k,k,l',l'}_{\alpha,\beta,\gamma,\delta}-\sum_{l'>l}H^{k,k,0,l'}_{\alpha,\beta,\gamma,\delta}\\
&H^{k,k,0,l}_{\alpha,\beta,\gamma,\delta}=\displaystyle
\mathcal{J}^{k,l+1}_{\alpha,\beta}\wedge [B^{k+1,l}_{\gamma,\delta}-H^{0,k,0,l}_{\alpha,\beta,\gamma,\delta}]\qquad \text{if }\beta+k>l
\label{H31a}\\
\bullet\text{Step 3a)}&\text{(deaths)}\nonumber \\
&\mathcal{J}^{k+1,l}_{\gamma,\delta}=\displaystyle\Gamma^{l}_{\gamma,\delta}p-\sum_{k'\geq l-\gamma+\alpha}H^{k',k',l,l}_{\alpha,\beta,\gamma,\delta}-\sum_{k'>k}H^{-k',0,l,l}_{\alpha,\beta,\gamma,\delta} 
\\
&H^{-k,0,l,l}_{\alpha,\beta,\gamma,\delta}=[D^{k,l+1}_{\alpha,\beta}-H^{-k,0,-l,0}_{\alpha,\beta,\gamma,\delta}]\wedge \mathcal{J}^{k+1,l}_{\gamma,\delta}\qquad \text{if }\alpha>\gamma-l
\label{H32a}\\
\nonumber\\
\bullet\text{Step 3b)}&\text{(births)}\nonumber \\
&\mathcal{B}^{k,l+1}_{\alpha,\beta}=\displaystyle \Pi^{0,k}_{\alpha,\beta}p-\sum_{l'\geq k-\delta+\beta}H^{0,k,0,l'}_{\alpha,\beta,\gamma,\delta}-\sum_{\gamma-\alpha\geq l'>l}H^{0,k,l',l'}_{\alpha,\beta,\gamma,\delta}\\
&H^{0,k,l,l}_{\alpha,\beta,\gamma,\delta}=\mathcal{B}^{k,l+1}_{\alpha,\beta}\wedge [J^{k+1,l}_{\gamma,\delta}-H^{k,k,l,l}_{\alpha,\beta,\gamma,\delta}]\qquad \text{if }\beta+k>\delta  \text{ and }\alpha \leq \gamma-l;
\label{H31b}\\
\bullet\text{Step 3b)}&\text{(deaths)}\nonumber \\
&\mathcal{D}^{k+1,l}_{\gamma,\delta}=\displaystyle \Pi^{-l,0}_{\gamma,\delta}p-\sum_{k'\geq l-\gamma+\alpha}H^{-k',0,-l,0}_{\alpha,\beta,\gamma,\delta}-\sum_{\delta-\beta\geq k'>k}H^{k',k',-l,0}_{\alpha,\beta,\gamma,\delta}\\
&H^{k,k,-l,0}_{\alpha,\beta,\gamma,\delta}=\displaystyle[J^{k,l+1}_{\alpha,\beta}-H^{k,k,l,l}_{\alpha,\beta,\gamma,\delta}]\wedge \mathcal{B}^{k+1,l}_{\gamma,\delta}\qquad \text{if }\alpha>\gamma-l  \text{ and }\beta+k\leq \delta
\label{H32b}
\end{array}
$$
$\bullet$ The uncoupled terms are
\begin{align}
i)\quad &H^{k,k,0,0}_{\alpha,\beta,\gamma,\delta}=\left\{ \begin{array}{ll}
\mathcal{J}^{k,1}_{\alpha,\beta}&\text{ if }\beta+k>\delta,\\
J^{k,1}_{\alpha,\beta}& \text{ otherwise; }
\end{array}
\right.&
ii)\quad &H^{0,0,l,l}_{\alpha,\beta,\gamma,\delta}= \left\{ \begin{array}{ll} 
\mathcal{J}^{1,l}_{\gamma,\delta}&\text{ if } \gamma-l<\alpha,\\
J^{1,l}_{\gamma,\delta}&\text{ otherwise; }
\end{array}
\right.
\label{lastjump}\\
i)\quad &H^{0,k,0,0}_{\alpha,\beta,\gamma,\delta}=\left\{ \begin{array}{ll}
\mathcal{B}^{k,1}_{\alpha,\beta}&\text{ if }\beta+k>\delta,\\
B^{k,1}_{\alpha,\beta}& \text{ otherwise; }
\end{array}
\right.&
ii)\quad &H^{0,0,0,l}_{\alpha,\beta,\gamma,\delta}=B^{1,l}_{\gamma,\delta}
\label{lastbirth}\\
i)\quad &H^{0,0,-l,0}_{\alpha,\beta,\gamma,\delta}=\left\{ \begin{array}{ll}
\mathcal{D}^{1,l}_{\gamma,\delta}&\text{ if }\gamma-l<\alpha,\\
D^{1,l}_{\gamma,\delta}& \text{ otherwise; }
\end{array}
\right.&
ii)\quad &H^{-k,0,0,0}_{\alpha,\beta,\gamma,\delta}=D^{k,1}_{\alpha,\beta}
\label{lastdeath}
\end{align}
\label{couplingHgen}
\end{definition}
\begin{remark}
The uncoupled terms are given by the remainders of the original rates: this ensures that we actually get a coupling. For instance, if $\beta+k>\delta$, then by (\ref{lastjump}) $(i)$, (\ref{H32b}) and (\ref{H31a}) we get
$$
\sum_{l\geq 0}(H^{k,k,l,l}_{\alpha,\beta,\gamma,\delta}+H^{k,k,0,l}_{\alpha,\beta,\gamma,\delta}+H^{k,k,-l,0}_{\alpha,\beta,\gamma,\delta})=\mathcal{J}^{k,1}_{\alpha,\beta}+\sum_{l>0}(H^{k,k,l,l}_{\alpha,\beta,\gamma,\delta}+H^{k,k,0,l}_{\alpha,\beta,\gamma,\delta})=\Gamma^k_{\alpha,\beta}p
$$ 
since $H^{k,k,l,l}_{\alpha,\beta,\gamma,\delta}=0$ for each $l<k-\delta+\beta$. If $\beta+k\leq \delta$, by (\ref{lastjump}) $(i)$ and (\ref{HJ}) 
$$
\sum_{l\geq 0}(H^{k,k,l,l}_{\alpha,\beta,\gamma,\delta}+H^{k,k,0,l}_{\alpha,\beta,\gamma,\delta}+H^{k,k,-l,0}_{\alpha,\beta,\gamma,\delta})=J^{k,1}_{\alpha,\beta}+\sum_{l>0}(H^{k,k,l,l}_{\alpha,\beta,\gamma,\delta}+H^{k,k,-l,0}_{\alpha,\beta,\gamma,\delta})=\Gamma^k_{\alpha,\beta}p
$$ 
since by (\ref{H31a}) $H^{k,k,0,l}_{\alpha,\beta,\gamma,\delta}=0$ for each $l>0$. One gets all the marginals by summing the coupling terms in a similar way.
\label{remark_first_terms}
\end{remark}
\begin{remark}
Coupling rates involving the largest change of $N$ particles constructed above can be obtained by an explicit calculation starting from Definition \ref{couplingHgen}. For instance Formula (\ref{gammaN0}) corresponds to (\ref{HJ}) with $k=N$. Indeed if $\beta$ is $N$-bad the coupling terms $H^{N,N,-l,0}_{\alpha,\beta,\gamma,\delta}=0$ by (\ref{H32b}),  and if $k=N$ then $\sum_{k'>N}(H^{k',k',l,l}_{\alpha,\beta,\gamma,\delta}+H^{0,k',l,l}_{\alpha,\beta,\gamma,\delta})=0$. Formulas (\ref{gammaNsim}), (\ref{PiN}), (\ref{a540}) and (\ref{temp3}) can be obtained in a similar way.
\end{remark}
%\begin{remark}
%By Definition (\ref{HJ}), (\ref{HB1}) and (\ref{HD1}) coupling terms involving jumps between jumps and births between births such that the final pairs of states break the partial order are null.
%\begin{align}
%H^{k,k,l,l}_{\alpha,\beta,\gamma,\delta}=0 &\qquad \text{ if }(\beta+k>\delta+l) \text{ or }(\alpha-k>\gamma-l).
%\label{jumpnull}
%\end{align}
%By Definition $(\ref{HB2})$ (resp. $(\ref{HD2})$) coupling terms between lower births (resp. upper deaths) and upper (resp. lower) jumps such that the final pairs of states break the partial order are null.
%\begin{align}
%H^{0,k,l,l}_{\alpha,\beta,\gamma,\delta}=0 &\qquad\text{ if }\alpha>\gamma-l;
%\label{birthnull}\\
%H^{k,k,-l,0}_{\alpha,\beta,\gamma,\delta}=0 &\qquad \text{ if }\beta+k>\delta.
%\label{deathnull}
%\end{align}
%\label{nullterms}
%\end{remark}
The easier formulation of coupling $\mathcal{H}$ when $N=1$ is derived in Appendix \ref{app}.\\

We explicitly constructed the coupling rates such that the largest change of $N$ particles breaks the partial order: the following proposition proves that coupling $\mathcal{H}$ is increasing for such rates, under Conditions ($\ref{C+}$) and ($\ref{C-}$) if $\beta$ or/and $\gamma$ are bad values. Indeed it states that all coupling terms that would break the order of configurations are equal to $0$.
\begin{proposizione}
$i)$ If $\beta$ is $N$-bad and $l<\lL$ then under Condition ($\ref{C+}$)
\begin{align}
H^{N,N,l,l}_{\alpha,\beta,\gamma,\delta}=
H^{0,N,0,l}_{\alpha,\beta,\gamma,\delta}=
H^{0,N,l,l}_{\alpha,\beta,\gamma,\delta}=
H^{N,N,0,l}_{\alpha,\beta,\gamma,\delta}=0.
\label{impi}
\end{align}
$ii)$ If $\gamma$ is $N$-bad and $k<\lU$, then under Condition ($\ref{C-}$)
\begin{align} 
H^{k,k,N,N}_{\alpha,\beta,\gamma,\delta}=
H^{-k,0,-N,0}_{\alpha,\beta,\gamma,\delta}=
H^{k,k,-N,0}_{\alpha,\beta,\gamma,\delta}=
H^{-k,0,N,N}_{\alpha,\beta,\gamma,\delta}=0.
\label{impii}
\end{align}
\label{imp}
\end{proposizione}
\begin{proof} 
$i)$ Suppose $l<\lL$. Then $H^{0,N,0,l}_{\alpha,\beta,\gamma,\delta}=H^{N,N,l,l}_{\alpha,\beta,\gamma,\delta}=0$ by (\ref{a69}) and (\ref{a69bis}).\\ 
$\bullet$ Suppose $N^{d+}=\lL-1$, then $N^B\geq \lL$ by Lemma \ref{newthird} and $H^{0,N,l,l}_{\alpha,\beta,\gamma,\delta}=0$ for each $l$ by (\ref{a68bis}).  Since $H^{N,N,0,l}_{\alpha,\beta,\gamma,\delta}$ is null for each $l<N^{dB}$ by Table \ref{tabN3a}, we prove that $N^{dB}\geq \lL$.\\ 
Assume by contradiction that $N^{dB}<\lL$. By Definition of $N^{dB}$, (\ref{a540}), Table \ref{tabNB1} and Remark $\ref{remthird}$,
\begin{align}
H^{N,N,0,\lL}_{\alpha,\beta,\gamma,\delta}=&\Pi^{0,\lL}_{\gamma,\delta}p-B^{N,\lL+1}_{\alpha,\beta}\wedge (\Pi^{0,\lL}_{\gamma,\delta}p)\nonumber \\
=&\Pi^{0,\lL}_{\gamma,\delta}p-\I_{\{\lL=N^{B}\}}(\Pi^{0,N}_{\alpha,\beta}p-\sum_{l'>N^{B}}\Pi^{0,l'}_{\gamma,\delta}p)<\mathcal{J}^{N,\lL+1}_{\alpha,\beta}.
\label{a54}
\end{align}
We compute $\mathcal{J}^{N,l+1}_{\alpha,\beta}$ if $l>N^{dB}$. By (\ref{a540}), (\ref{remainder}) and Remark $\ref{remthird}$
\begin{align}
\mathcal{J}^{N,l+1}_{\alpha,\beta}=&\Gamma^{N}_{\alpha,\beta}p-\sum_{l'\geq \lL}\Gamma^{l'}_{\gamma,\delta}p-\sum_{l'>l}H^{N,N,0,l'}_{\alpha,\beta,\gamma,\delta} \nonumber\\
=&\Gamma^{N}_{\alpha,\beta}p-\sum_{l'\geq \lL}\Gamma^{l'}_{\gamma,\delta}p-\sum_{l'>l}\big\{\I_{\{l'< N^{B}\}}\Pi^{0,l'}_{\gamma,\delta}p+\I_{\{l'=N^{B}\}}H^{N,N,0,N^{B}}_{\alpha,\beta,\gamma,\delta}\big\}\nonumber \\
=&\Gamma^{N}_{\alpha,\beta}p-\sum_{l'\geq \lL}\Gamma^{l'}_{\gamma,\delta}p-\sum_{N^{B}>l'>l}\Pi^{0,l'}_{\gamma,\delta}p\nonumber \\
&-\I_{\{N^{B}>l\}}\big(\Pi^{0,N^{B}}_{\gamma,\delta}p-(\Pi^{0,N}_{\alpha,\beta}p-\sum_{l'> N^{B}}\Pi^{0,l'}_{\gamma,\delta}p)\big).\label{tmp_T}
\end{align}
Hence by (\ref{a54}), 
$$ \mathcal{J}^{N,N^{d+}+1}_{\alpha,\beta}> \Pi^{0,\lL}_{\gamma,\delta}p-\I_{\{\lL=N^{B}\}}(\Pi^{0,N}_{\alpha,\beta}p-\sum_{l'>N^{B}}\Pi^{0,l'}_{\gamma,\delta}p).
$$
This implies that if either $N^{B}=\lL$ or $N^{B}>\lL$
\begin{equation}
\Gamma^{N}_{\alpha,\beta}p+\Pi^{0,N}_{\alpha,\beta}p>\sum_{l'\geq \lL}\Gamma^{l'}_{\gamma,\delta}p-\sum_{l'\geq \lL}\Pi^{0,l'}_{\gamma,\delta}p
\label{contr1}
\end{equation}
which contradicts (\ref{condN}).\\ 

$\bullet$ Suppose $N^B=\lL-1$, that is $B^{N,\lL}_{\alpha,\beta}>0$. Then $N^{d+}\geq \lL$ by Lemma \ref{newthird} and $H^{N,N,0,l}_{\alpha,\beta,\gamma,\delta}=0$ for each $l$ by (\ref{a27}).\\
Notice that Condition (\ref{C+}) with $K=1$, $m_1=0$ and $j_1=\lL-1$ reduces to 
\begin{equation}
\Pi^{0,N}_{\alpha,\beta}p\leq \sum_{l'\geq \lL}\Pi^{0,l'}_{\gamma,\delta}p-\sum_{\gamma-\alpha\geq l'\geq \lL}\Gamma^{l'}_{\gamma,\delta}p
\label{contr2}
\end{equation}
which contradicts $B^{N,\lL}_{\alpha,\beta}>0$ if $\lL>\gamma-\alpha$. Therefore we assume $\lL\leq \gamma-\alpha$.\\
Since $H^{0,N,l,l}_{\alpha,\beta,\gamma,\delta}$ is null for each $l<N^{Bd}$ by Table \ref{tabN3b}, we prove that $N^{Bd}\geq \lL$. Suppose by contradiction that $N^{Bd}<\lL\leq \gamma-\alpha$. By Definition of $N^{Bd}$, (\ref{temp3}), Table \ref{tabNd1} and Remark $\ref{remsecond}$,
\begin{align}
H^{0,N,\lL,\lL}_{\alpha,\beta,\gamma,\delta}=&\Gamma^{\lL}_{\gamma,\delta}p-J^{N,\lL+1}_{\alpha,\beta}\wedge (\Gamma^{\lL}_{\gamma,\delta}p)\nonumber \\
=&\Gamma^{\lL}_{\gamma,\delta}p-\I_{\{\lL=N^{d+}\}}(\Gamma^{N}_{\alpha,\beta}p-\sum_{l'>N^{d+}}\Gamma^{l'}_{\gamma,\delta}p)<\mathcal{B}^{N,\lL+1}_{\alpha,\beta}.
\label{a54new}
\end{align}
We compute $\mathcal{B}^{N,l+1}_{\alpha,\beta}$ if $l>N^{Bd}$. By (\ref{temp3}), (\ref{remainder2}) and Remark $\ref{remsecond}$
\begin{align}
\mathcal{B}^{N,l+1}_{\alpha,\beta}=&\Pi^{0,N}_{\alpha,\beta}p-\sum_{l'\geq \lL}\Pi^{0,l'}_{\gamma,\delta}p-\sum_{\gamma-\alpha\geq l'>l}H^{0,N,l',l'}_{\alpha,\beta,\gamma,\delta}\nonumber \\
=&\Pi^{0,N}_{\alpha,\beta}p-\sum_{l'\geq \lL}\Pi^{0,l'}_{\gamma,\delta}p-\sum_{\gamma-\alpha\geq l'>l}\big\{\I_{\{N^{d+}>l'\}}\Gamma^{l'}_{\gamma,\delta}p+\I_{\{l'=N^{d+}\}}H^{0,N,N^{d+},N^{d+}}_{\alpha,\beta,\gamma,\delta}\big\}\nonumber \\
=&\Pi^{0,N}_{\alpha,\beta}p-\sum_{l'\geq \lL}\Pi^{0,l'}_{\gamma,\delta}p-\sum_{N^{d+}\wedge(\gamma-\alpha+1)>l'>l}\Gamma^{l'}_{\gamma,\delta}p\nonumber \\
&-\I_{\{\gamma-\alpha \geq N^{d+}>l\}}\big(\Gamma^{N^{d+}}_{\gamma,\delta}p-(\Gamma^{N}_{\alpha,\beta}p-\sum_{l'> N^{d+}}\Gamma^{l'}_{\gamma,\delta}p)\big).\label{tmp_T_bis}
\end{align}
Hence by (\ref{a54new}), 
$$
\mathcal{B}^{N,\lL+1}_{\alpha,\beta}>\Gamma^{\lL}_{\gamma,\delta}p-\I_{\{\lL=N^{d+}\}}(\Gamma^{N}_{\alpha,\beta}p-\sum_{l'>N^{d+}}\Gamma^{l'}_{\gamma,\delta}p).
$$
If $\gamma-\alpha\geq N^{d+}\geq \lL$ this is equivalent to (\ref{contr1}), hence a contradiction. If $N^{d+}>\gamma-\alpha$ we get
$$
\Pi^{0,N}_{\alpha,\beta}p>\sum_{l'\geq \lL}\Pi^{0,l'}_{\gamma,\delta}p-\sum_{\gamma-\alpha\geq l'\geq \lL}\Gamma^{l'}_{\gamma,\delta}p
$$
which contradicts (\ref{contr2}).\\
Claim $(ii)$ is proved by symmetric arguments.
\end{proof}

\begin{remark}
As a consequence, Tables $\ref{tabN3a}$ and $\ref{tabN3b}$ (and the symmetric ones) do not contain any coupling term breaking the partial order between configurations. 
\end{remark} 
%\begin{corollario}
%$i)$ If $\gamma$ is $N$-good and $\beta+N=\delta+1$, then under Condition (\ref{C+}), $\mathcal{H}$ is increasing.\\
%$ii)$ If $\beta$ is $N$-good and $\gamma-N=\alpha-1$, then under Condition (\ref{C-}), $\mathcal{H}$ is increasing.\\
%$ii)$ If $\beta+N=\delta+1$ and $\gamma-N=\alpha-1$, then under Condition (\ref{C-})--(\ref{C+}), $\mathcal{H}$ is increasing.
%\label{cor}
%\end{corollario}
%\begin{proof}.
%Since the only attractiveness problem is given by additions or subtractions of $N$ particles, the claim follows from Proposition \ref{imp} ( resp. $i)$, $ii)$ or both).  
%\end{proof}
%We will use Corollary \ref{cor} as induction basis in next section.

\subsection{Sufficient condition}
\label{sufcond}

We complete the proof of Theorem $\ref{cns}$ by
\begin{proposizione}
Under Conditions ($\ref{C+}$)--($\ref{C-}$), $\mathcal{H}$ is increasing.
\label{suf}
\end{proposizione}
In order to prove Proposition $\ref{suf}$, we define a new system $\esse$ (in fact a new pair of systems $\esse:=(\overline{\widetilde{\mathcal{S}}},\esse)$ by Remark $\ref{abuse}$), depending on $\mathcal{S}$ and $\mathcal{H}$, whose rates are those of $\mathcal{S}$ to whom we subtract the coupled rates of $\mathcal{H}$ involving changes  of $N$ particles.  
\begin{definition}
Given $\mathcal{S}$, $\esse$ has the transition rates
\begin{align*}
\oGamma^{k}_{\alpha,\beta}p=&\left\{
\begin{array}{ll}
\mathcal{J}^{N,1}_{\alpha,\beta} &\text{ if } (k=N, \beta+N>\delta)\\
J^{k,N}_{\alpha,\beta} &\text{ if } k<N \text{ or }\\
& (k=N, \beta+N\leq \delta);
\end{array}
\right.&
\oGamma^{l}_{\gamma,\delta}p=&\left\{
\begin{array}{ll}
\mathcal{J}^{1,N}_{\gamma,\delta}& \text{ if }(l=N, \gamma-N<\alpha),\\
J^{N,l}_{\gamma,\delta}& \text{if }l<N \text{ or }\\
& (l=N,\gamma-l\geq \alpha);
\end{array}
\right.\\
\\
\oPi^{0,k}_{\alpha,\beta}p=&\left\{
\begin{array}{ll}
\mathcal{B}^{N,1}_{\alpha,\beta}& \text{ if }k=N,\\
\Pi^{0,k}_{\alpha,\beta}p &  \text{ if }k<N;
\end{array}
\right.&
\oPi^{0,l}_{\gamma,\delta}p=&B^{N,l}_{\gamma,\delta}\text{ for each }l;\\\\
\oPi^{-k,0}_{\alpha,\beta}p=&D^{-k,N}_{\alpha,\beta}\text{ for each }k;&
\oPi^{-l,0}_{\gamma,\delta}p=&\left\{
\begin{array}{ll}
\mathcal{D}^{1,N}_{\gamma,\delta}& \text{ if }l=N, \\
\Pi^{-l,0}_{\gamma,\delta}p& \text{ if }l<N.\\
\end{array}
\right.
\end{align*}
\label{newsystem}
\end{definition}
Our plan consists in working by induction on the largest change of particles $n(\mathcal{S})$ which causes either a lower or a higher attractiveness problem: given the particle system $\mathcal{S}$, it is defined by
\begin{align}
W:=&\{k:(k>\delta-\beta \text{ or }k>\gamma-\alpha)\text{ and } (\Gamma^{k}_{\alpha,\beta}+\Gamma^{k}_{\gamma,\delta}+\Pi^{0,k}_{\alpha,\beta}+\Pi^{-k,0}_{\gamma,\delta}> 0)\}\nonumber\\
n(\mathcal{S})=&n(\mathcal{S},\alpha,\beta,\gamma,\delta):=\left\{ 
\begin{array}{ll}
\sup W & \text{ if }W\neq \emptyset;\\
(\delta-\beta)\wedge (\gamma-\alpha) & \text{ otherwise.}
\end{array}
\right. 
\label{enne}
\end{align}
\begin{remark}
If $\gamma$ (resp. $\beta$) is $N$-good, then $n(\mathcal{S})=\sup\{k:(k>\delta-\beta) \text{ and } (\Gamma^{k}_{\alpha,\beta}+\Pi^{0,k}_{\alpha,\beta}> 0)\}$ (resp. $n(\mathcal{S})=\sup\{k:(k>\gamma-\alpha) \text{ and } (\Gamma^{k}_{\gamma,\delta}+\Pi^{-k,0}_{\gamma,\delta}> 0)\}$).
\label{rem_enne}
\end{remark}
By Remark \ref{remN}, $n(\mathcal{S})=N$. Let $\oN=n(\overline{\mathcal{S}})$.

%\emph{We assume both $\beta$ and $\gamma$ are $N$-bad}. If this is not the case the proof works in an easier way (see Remark \ref{rem_LU}).\\

We prove that if $\mathcal{S}$ satisfies Conditions (\ref{C+})--(\ref{C-}) and $n(\mathcal{S})=N$, then $\mathcal{H}=\mathcal{H}(\mathcal{S})$ is increasing. The induction hypothesis is: if a particle system $\mathcal{S}^*$ satisfies Conditions (\ref{C+})--(\ref{C-}) and $n(\mathcal{S}^*)\leq N-1$, then $\mathcal{H}(\mathcal{S}^*)$ is increasing.\\
We give an outline of the proof: suppose that the induction hypothesis is satisfied. We defined a new system $\overline{\mathcal{S}}$. By Proposition \ref{NLU}, $\overline{n}\leq N-1$ and by Proposition \ref{attrdue}, it satisfies Conditions (\ref{C+})--(\ref{C-}). Therefore we can use the induction hypothesis and $\overline{\mathcal{H}}=\mathcal{H}(\overline{\mathcal{S}})$ is increasing. This implies, by Proposition \ref{HHbar}, that $\mathcal{H}(\mathcal{S})$ is increasing.\\
%In other words we work by induction on the higher rate which causes an attractiveness problem by defining at each step a new system and we get the result \emph{if the induction basis is satisfied}.
\begin{proposizione}
If either $\beta$ or $\gamma$ (or both) are $N$-bad, then $\oN\leq N-1$.
\label{NLU}
\end{proposizione}
\begin{proof}. If $\beta$ is $N$-bad, we prove that $\oGamma^{N}_{\alpha,\beta}=0=\oPi^{0,N}_{\alpha,\beta}$. By Definition \ref{newsystem} and (\ref{lastjump}) $i)$
\begin{align*}
\oGamma^{N}_{\alpha,\beta}p=&\mathcal{J}^{N,1}_{\alpha,\beta}=H^{N,N,0,0}_{\alpha,\beta,\gamma,\delta};
\end{align*}
and by Definition \ref{newsystem} and (\ref{lastbirth}) $i)$ 
\begin{align*}
\oPi^{0,N}_{\alpha,\beta}p=&\mathcal{B}^{N,1}_{\alpha,\beta}=H^{0,N,0,0}_{\alpha,\beta,\gamma,\delta}.
\end{align*}
Since $\beta+N>\delta+0$, by Proposition $\ref{imp}$ $i)$, $H^{N,
N,0,0}_{\alpha,\beta,\gamma,\delta}=0$, $H^{0,N,0,0}_{\alpha,\beta,\gamma,\delta}=0$ and we are done. In a symmetric way $\oGamma^{N}_{\gamma,\delta}=\oPi^{-N,0}_{\gamma,\delta}=0$ if $\gamma$ is $N$-bad.\\
If both $\beta$ and $\gamma$ are $N$-bad, by (\ref{enne}), $\overline{n}\leq N-1$. If $\beta$ is $N$-bad but $\gamma$ is $N$-good, then $\overline{n}=\sup\{k:(k>\delta-\beta) \text{ and } (\oGamma^{k}_{\alpha,\beta}+\oPi^{0,k}_{\alpha,\beta})> 0\}$. Again $\overline{n}\leq N-1$ since $\oGamma^{N}_{\alpha,\beta}$ and $\oPi^{0,N}_{\alpha,\beta}$ are null. The same conclusion holds if $\gamma$ is $N$-bad and $\beta$ is good. 
\end{proof}

The harder part is:
\begin{proposizione}
If $\mathcal{S}$ satisfies Conditions $(\ref{C+})$--$(\ref{C-})$, then so does $\esse$.
\label{attrdue}
\end{proposizione} 
\begin{proof}.
We prove that for all $K$, $\mathbf{h}$, $\mathbf{j}, \mathbf{m}$, $I_a$, $I_b$, $I_c$, $I_d$ in Theorem \ref{cns},
\begin{align}
\sum_{k > \delta-\beta+j_{1}}\oPi_{\alpha, \beta}^{0,k}+\sum_{k \in I_{a}}\oGamma_{\alpha, \beta}^{k}\leq &
\sum_{l>j_{1}}\oPi_{\gamma, \delta}^{0,l}+\sum_{l \in I_{b}}\oGamma_{\gamma, \delta}^{l}
\label{C+bar}\\
\sum_{k>h_1}\oPi_{\alpha, \beta}^{-k,0}+\sum_{k \in I_{d}}\oGamma_{\alpha, \beta}^{k}\geq &
\sum_{l> \gamma-\alpha+h_1}\oPi_{\gamma, \delta}^{-l,0}+\sum_{l\in I_{c}}\oGamma_{\gamma, \delta}^{l}.
\label{C-bar}
\end{align}
We prove $(\ref{C+bar})$. Since, by symmetry, the proof of $(\ref{C-bar})$ is similar, we skip it. 
\begin{remark}
Let $A=\{a\in X: a\leq K, j_{a}\geq \lL\}$, then for each $k>\delta-\beta+j_i>N$ such that $j_i\geq j_a$ we have $\overline{\Gamma}^{k}_{\alpha,\beta}=0$ by Definitions of $N$ and \ref{newsystem}. Therefore $\I_{\{m_{i}\geq k>\delta-\beta+j_{i}\}}\overline{\Gamma}^{k}_{\alpha,\beta}=0$. Let
$$
K_A=\left\{\begin{array}{ll}
\min A & \text{ if }A\neq \emptyset\\
K+1 & \text{ otherwise.}
\end{array}
\right.
$$
\begin{align}
I^{K_A}_{a}=&\bigcup_{i=1}^{K_A-1}\{m_{i}\geq k>\delta-\beta+j_{i}\},& I^{K_A}_{b}=&\bigcup_{i=1}^{K_A-1}\{\gamma-\alpha+m_{i}\geq l>j_{i}\}
\label{i}
\end{align}
then condition
\begin{align*}
\sum_{k > \delta-\beta+j_{1}}\oPi_{\alpha, \beta}^{0,k}+\sum_{k \in I^{K_A}_{a}}\oGamma_{\alpha, \beta}^{k}\leq \sum_{l>j_{1}}\oPi_{\gamma, \delta}^{0,l}+\sum_{l \in I^{K_A}_{b}}\oGamma_{\gamma, \delta}^{l}
\end{align*}
implies (\ref{C+bar}) and we can suppose without loss of generality 
\begin{equation}
\lL>j_{K}.
\label{ellekappa}
\end{equation}
%As a consequence, by Proposition $\ref{attruno}$, if $\beta$ is $N$-bad
%\begin{equation}
%\big(N^{d+}\I_{\{\mathcal{R}^N_{\alpha,\beta}=0\}}+N^{dB}\I_{\{\mathcal{R}^N_{\alpha,\beta}>0\}}\big)\wedge \big(N^B\I_{\{N^B>\gamma-\alpha\}}+(N^B\vee N^{Bd})\I_{\{N^B<\gamma-\alpha\}}\big)\geq \lL>j_K.
%\label{ellekappa_bis}
%\end{equation} 
If $\gamma$ is $N$-bad a similar remark involving $\lU$ and variables $N^{d-}$, $N^{D}$, $N^{Dd}$ and $N^{dD}$ holds by symmetry.
\label{Iset} 
\end{remark}
If $\beta$ is $N$-good then Condition (\ref{C+bar}) is trivially satisfied. We suppose that both $\beta$ and $\gamma$ are $N$-bad. If $\gamma$ is $N$-good the proof is similar but easier, then we skip it. Hence we suppose $\lL>0$ and $\lU>0$. By Proposition $\ref{NLU}$, $\oPi^{0,N}_{\alpha,\beta}=0$ and by Definition $\ref{newsystem}$, $\oPi^{0,l}_{\alpha,\beta}=\Pi^{0,l}_{\alpha,\beta}$ for each $l<N$. Therefore
\begin{equation}
\sum_{k > \delta-\beta+j_{1}}\oPi_{\alpha, \beta}^{0,k}p=\sum_{k > \delta-\beta+j_{1}}\Pi_{\alpha, \beta}^{0,k}p-\Pi_{\alpha, \beta}^{0,N}p
\label{lhs1}
\end{equation}
By Proposition \ref{NLU}, $\oGamma^{N}_{\alpha,\beta}=0$. Moreover $H^{k,k,-N,0}_{\alpha,\beta,\gamma,\delta}=0$ for each $k> \delta-\beta+j_{i}\geq \delta-\beta$ by (\ref{H32b}). Then by Definition \ref{newsystem}, (\ref{HJ}) and $(\ref{ellekappa})$ we have
\begin{equation}
\sum_{k\in I_{a}}\oGamma^{k}_{\alpha,\beta}p=\sum_{k \in I_{a}\setminus \{N\}}\big(\Gamma^{k}_{\alpha,\beta}p-H^{k,k,N,N}_{\alpha,\beta,\gamma,\delta}\big)
\label{tempgammaleft}
\end{equation}
where $I_{a}\setminus \{N\}$ is the shorthand for $I_{a}\I_{\{N\notin I_{a}\}}+(I_{a}\setminus \{N\})\I_{\{N\in I_{a}\}}$.\\

The right hand side of $(\ref{C+bar})$ is given by Definition $\ref{newsystem}$, (\ref{HJ}) and (\ref{HB1}). Notice that $\oGamma^{N}_{\gamma,\delta}=0$ by Proposition $\ref{NLU}$ since $\gamma$ is $N$-bad; moreover $H^{0,N,l,l}_{\alpha,\beta,\gamma,\delta}=0$ if $l>\gamma-\alpha$ by (\ref{0Nll>}):
\begin{align}
\sum_{l>j_{1}}\oPi_{\gamma, \delta}^{0,l}p+\sum_{l\in I_{b}}\oGamma_{\gamma, \delta}^{l}p=&\sum_{l>j_{1}}\Big(\Pi^{0,l}_{\gamma,\delta}p-H^{0,N,0,l}_{\alpha,\beta,\gamma,\delta}-H^{N,N,0,l}_{\alpha,\beta,\gamma,\delta}\Big)\nonumber \\
&+\sum_{l\in I_{b}\setminus \{N\}}\Big(\Gamma_{\gamma,\delta}^{l}p-H^{N,N,l,l}_{\alpha,\beta,\gamma,\delta}-H^{0,N,l,l}_{\alpha,\beta,\gamma,\delta}\Big).
\label{a57}
\end{align}

By Remark \ref{newthird} either $N^{B}\geq \lL$ or $N^{d+}\geq \lL$. We detail the case $N^{d+}\geq \lL$, which contains all the technical difficulties of the proof. The other proof (for $N^{d+}=\lL-1$) is similar.\\  

Since $N^{d+}\geq \lL$ then $H^{N,N,0,l}_{\alpha,\beta,\gamma,\delta}=0$ for each $l$ by (\ref{a27}). Moreover $H^{0,N,l,l}_{\alpha,\beta,\gamma,\delta}=0$ for each $l>\gamma-\alpha$ by (\ref{H31b}). Since $\lU>0$ and $\{\gamma-\alpha\geq l >j_{1}\}\subseteq I_{b}$ then $\displaystyle \sum_{\{l\in I_{b}\}\cap \{l\neq N\}}H^{0,N,l,l}_{\alpha,\beta,\gamma,\delta}=\I_{\{N^B=\lL-1\}}\sum_{\gamma-\alpha \geq l>j_{1}}H^{0,N,l,l}_{\alpha,\beta,\gamma,\delta}$. Moreover using Tables $\ref{tabNB1}$ and $\ref{tabNB2}$
\begin{align}
\sum_{l>j_{1}}\big(\Pi^{0,l}_{\gamma,\delta}p-H^{0,N,0,l}_{\alpha,\beta,\gamma,\delta}\big)=& \sum_{N^{B}\geq l>j_{1}}\big(\Pi^{0,l}_{\gamma,\delta}p-H^{0,N,0,l}_{\alpha,\beta,\gamma,\delta}\big).
\label{a40}
\end{align}
The first term on the right hand side (\ref{a57}) that we consider is
\begin{equation}
\sum_{N^B\geq l>j_{1}}(\Pi^{0,l}_{\gamma,\delta}p-H^{0,N,0,l}_{\alpha,\beta,\gamma,\delta})-\I_{\{N^{B}=\lL-1\}}\sum_{\gamma-\alpha \geq l>j_{1}}H^{0,N,l,l}_{\alpha,\beta,\gamma,\delta}.
\label{term}
\end{equation}
$\bullet$ If $N^B\geq \lL$, since $\lL>j_K\geq j_1$ by $(\ref{ellekappa})$, we have $N^B>j_1$:  by using Table $\ref{tabNB1}$ and $(\ref{PiN})$, (\ref{term}) is equal to
\begin{align}
\sum_{N^B\geq l>j_{1}}\Pi^{0,l}_{\gamma,\delta}p-B^{N,N^{B}+1}_{\alpha,\beta}=\sum_{l>j_{1}}\Pi^{0,l}_{\gamma,\delta}p-\Pi^{0,N}_{\alpha,\beta}p.
\label{rhsNB}
\end{align}
$\bullet$ If $N^B= \lL-1$, since in proof of Proposition \ref{imp} we obtained that $N^{Bd}\geq \lL$, again by $(\ref{ellekappa})$, we have $N^{Bd}>j_1$: by Table \ref{tabN3b}, (\ref{temp3}) and (\ref{remainder2}), (\ref{term}) is equal to 
\begin{align}
\sum_{\lL-1\geq l>j_{1}}\Pi^{0,l}_{\gamma,\delta}p-\sum_{\gamma-\alpha \geq l>N^{Bd}}H^{0,N,l,l}_{\alpha,\beta,\gamma,\delta}-\mathcal{B}^{N,N^{Bd}+1}_{\alpha,\beta}=&\sum_{l>\lL-1}\Pi^{0,l}_{\gamma,\delta}p-B^{N,\lL}_{\alpha,\beta}\nonumber\\
=&\sum_{l>j_{1}}\Pi^{0,l}_{\gamma,\delta}p-\Pi^{0,N}_{\alpha,\beta}p.
\label{rhsNBd}
\end{align}
that is (\ref{term}) has the same value in both cases. The second sum on the right hand side of (\ref{a57}) is
\begin{align}
\sum_{l\in I_{b}\setminus \{N\}}\big(\Gamma_{\gamma,\delta}^{l}p-H^{N,N,l,l}_{\alpha,\beta,\gamma,\delta}\big)=&\sum_{l\in I_{b}\setminus \{N\}}\Gamma_{\gamma,\delta}^{l}p-\sum_{l\in I_{b}\setminus \{N\}}\Gamma_{\gamma,\delta}^{l}p\I_{\{l>N^{d+}\}}\nonumber \\
&-\sum_{l\in I_{b}\setminus \{N\}}\big(\Gamma^{N}_{\alpha,\beta}p-\sum_{l'>N^{d+}}\Gamma_{\gamma,\delta}^{l'}p\big)\I_{\{l=N^{d+}\}}\nonumber \\
=&\sum_{l\in I_{b}\setminus \{N\}}\Gamma_{\gamma,\delta}^{l}p-\sum_{l\in I_{b}\setminus \{N\}}\Gamma_{\gamma,\delta}^{l}p\I_{\{l>N^{d+}\}}\nonumber \\
&-\big(\Gamma^{N}_{\alpha,\beta}p-\sum_{l>N^{d+}}\Gamma_{\gamma,\delta}^{l}p\big)\I_{\{N^{d+}\in I_b\setminus \{N\}\}}.
\label{a40bis}
\end{align}
Therefore using (\ref{a40}), (\ref{rhsNB}), (\ref{rhsNBd}) and (\ref{a40bis})  
\begin{align}
\sum_{l>j_{1}}\oPi_{\gamma, \delta}^{0,l}p+\sum_{l\in I_{b}}\oGamma_{\gamma, \delta}^{l}p=&\sum_{l>j_{1}}\Pi^{0,l}_{\gamma,\delta}p-\Pi^{0,N}_{\alpha,\beta}p+\sum_{l\in I_{b}\setminus \{N\}}\Gamma_{\gamma,\delta}^{l}p-\sum_{l\in I_{b}\setminus \{N\}}\Gamma_{\gamma,\delta}^{l}p\I_{\{l>N^{d+}\}}\nonumber\\
&-\big(\Gamma^{N}_{\alpha,\beta}p-\sum_{l>N^{d+}}\Gamma_{\gamma,\delta}^{l}p\big)\I_{\{N^{d+}\in I_b\setminus \{N\}\}}.
\label{temprightside}
\end{align}
We use (\ref{lhs1}), (\ref{tempgammaleft}), (\ref{temprightside}) and Condition (\ref{C+}) to check that Condition (\ref{C+bar}) is satisfied.\\

\noindent \textbf{Case A:} Suppose $N^{d-}=N$. In this case $\Gamma^{N}_{\alpha,\beta}>\Gamma^{N}_{\gamma,\delta}$ (by $(\ref{firstdiffusion})$) and $H^{k,k,N,N}_{\alpha,\beta,\gamma,\delta}=0$ for each $k<N$. Therefore by $(\ref{tempgammaleft})$
\begin{equation}
\sum_{k \in I_{a}}\oGamma_{\alpha, \beta}^{k}=\sum_{k \in I_{a}\setminus \{N\}}\Gamma_{\alpha, \beta}^{k}.
\label{lhs2}
\end{equation}
$\bullet$ Suppose $N^{d+}\notin I_{b}$. Since $N^{d+}\geq \lL>j_{K}$ by $(\ref{ellekappa})$, we must have $N^{d+}>\gamma-\alpha+m_K$, hence $I_{b}\setminus \{\{N\}\cup \{N\geq l>N^{d+}\}\}=I_{b}$. It implies that $\sum_{l \in I_{b}\setminus \{N\}}\Gamma^{l}_{\gamma,\delta}\I_{\{l>N^{d+}\}}=0$, and by (\ref{lhs1}), (\ref{tempgammaleft}) and (\ref{temprightside}) Condition $(\ref{C+bar})$ becomes
\begin{align*}
&\sum_{k > \delta-\beta+j_{1}}\Pi_{\alpha, \beta}^{0,k}+\sum_{k \in I_{a}\setminus \{N\}}\Gamma_{\alpha, \beta}^{k}\leq \sum_{l>j_{1}}\Pi_{\gamma, \delta}^{0,l}+\sum_{l \in I_{b}}\Gamma^{l}_{\gamma,\delta}
\end{align*}
which holds by Condition $(\ref{C+})$.\\
$\bullet$ If $N^{d+}\in I_{b}\setminus \{N\}$, then 
\begin{equation}
\gamma-\alpha+m_{K}\geq N^{d+}\geq \lL>j_{K},
\label{a41}
\end{equation}
hence (see Table $\ref{tabNd1}$) by (\ref{lhs1}), (\ref{tempgammaleft}) and (\ref{temprightside}) Condition $(\ref{C+bar})$ is 
\begin{align}
\sum_{k > \delta-\beta+j_{1}}\Pi_{\alpha, \beta}^{0,k}+\sum_{k \in I_{a}\setminus \{N\}}\Gamma_{\alpha, \beta}^{k}\leq & \sum_{l>j_{1}}\Pi_{\gamma, \delta}^{0,l}+\sum_{l\in I_{b}\setminus \{N\}}\Gamma_{\gamma,\delta}^{l}-\sum_{l\in I_{b}\setminus \{N\}}\Gamma_{\gamma,\delta}^{l}\I_{\{l>N^{d+}\}}\nonumber \\
&-\Gamma^{N}_{\alpha,\beta}+\sum_{l>N^{d+}}\Gamma_{\gamma,\delta}^{l}\nonumber\\
=&\sum_{l>j_{1}}\Pi_{\gamma, \delta}^{0,l}+\sum_{l \in \left\{I_{b}\setminus \{N\}\right\}\cup \{l>N^{d+}\}}\Gamma^{l}_{\gamma,\delta}-\Gamma^{N}_{\alpha,\beta}.
\label{tmptmp}
\end{align}
Notice that
$$
\left\{I_{b}\setminus \{N\}\right\}\cup \{l>N^{d+}\}=I_{b}\cup \{l>N^{d+}\}.
$$
$\bullet$ If $N\in I_{a}$, $(\ref{tmptmp})$ becomes
$$
\sum_{k > \delta-\beta+j_{1}}\Pi_{\alpha, \beta}^{0,k}+\sum_{k \in I_{a}}\Gamma_{\alpha, \beta}^{k}\leq \sum_{l>j_{1}}\Pi_{\gamma, \delta}^{0,l}+\sum_{l \in I_{b}\cup \{l>N^{d+}\}}\Gamma^{l}_{\gamma,\delta}
$$
which holds by $(\ref{C+})$.\\
$\bullet$ If $N\notin I_{a}$, (\ref{tmptmp}) writes 
\begin{align}
&\sum_{k > \delta-\beta+j_{1}}\Pi_{\alpha, \beta}^{0,k}+\sum_{k \in I_{a}\cup \{N\}}\Gamma_{\alpha, \beta}^{k}\leq \sum_{l>j_{1}}\Pi_{\gamma, \delta}^{0,l}+\sum_{l \in I_{b} \cup \{l>N^{d+}\}}\Gamma^{l}_{\gamma,\delta}
\label{tempgamma4}
\end{align}
Denote by 
\begin{align*}
\widehat{I}_{a}=&I_{a}\cup \{N \geq l\geq \delta-\beta+(N-\delta+\beta)\}=I_{a}\cup \{N \geq l\geq N\}\\
%\label{Itildea}\\
\widehat{I}_{b}=&I_{b}\cup \{N+\gamma-\alpha\geq l\geq N-\delta+\beta\}
%\label{Itildeb}
\end{align*}
then by Condition $(\ref{C+})$ applied to $\widehat{I}_{a}$, $\widehat{I}_{b}$ 
\begin{align}
\sum_{k > \delta-\beta+j_{1}}\Pi_{\alpha, \beta}^{0,k}+\sum_{k \in \widehat{I}_{a}}\Gamma_{\alpha, \beta}^{k}\leq \sum_{l>j_{1}}\Pi_{\gamma, \delta}^{0,l}+\sum_{l \in \widehat{I}_{b}}\Gamma^{l}_{\gamma,\delta}.
\label{tmptmp2}
\end{align}
By $(\ref{a41})$, $\widehat{I}_{b}=I_{b}\cup \{N+\gamma-\alpha\geq l> N^{d+}\}$ and $(\ref{tmptmp2})$ implies $(\ref{tempgamma4})$.\\
 
\begin{remark}
This is the key passage where a Definition of $I_{a},\ldots I_{d}$ as a single set instead of a union of several sets does not work.
\label{key}
\end{remark}

\noindent \textbf{Case B:} Suppose $N=N^{d+}$. In this case $H^{N,N,l,l}_{\alpha,\beta,\gamma,\delta}=0$ for each $l<N$, the left hand side of (\ref{a40bis}) is equal to $\sum_{l\in I_{b}\setminus \{N\}}\Gamma_{\gamma,\delta}^{l}p$, Formula (\ref{tempgammaleft}) still holds but terms $H^{k,k,N,N}_{\alpha,\beta,\gamma,\delta}$ are not null. One works as in Case A, when $N^{d-}=N$, by using the symmetric construction of $H^{k,k,N,N}_{\alpha,\beta,\gamma,\delta}$ and by checking the condition in different cases.
\end{proof}

\begin{proposizione}
If $\acca:=\mathcal{H}(\esse)$ is increasing, then $\mathcal{H}=\mathcal{H}(\mathcal{S})$ is increasing.
\label{HHbar}
\end{proposizione}  
\begin{proof}.
We show that all coupling rates of $\mathcal{H}$ breaking the partial order between configurations are null. By Proposition \ref{imp}, the ones involving $N$ particles are null and if both $\beta$ and $\gamma$ are $(N-1)$-good then $\mathcal{H}$ is increasing.
% (Corollary \ref{cor}). 
Therefore we suppose that either $\beta$ or $\gamma$ is $(N-1)$-bad.\\
We prove that $\mathcal{H}$ and $\overline{\mathcal{H}}$ differ only on rates where a change of $N$ particles causes an attractiveness problem. By Proposition \ref{NLU} such terms are null for $\overline{\mathcal{H}}$.\\

\noindent Step 1: by Definition \ref{couplingHgen} (\ref{HJ}) on $\esse$ and Definition $\ref{newsystem}$, if $N>k\geq k-\delta+\beta$, we get 
\begin{align}
\oH^{k,k,l,l}_{\alpha,\beta,\gamma,\delta}=&[\oGamma^{k}_{\alpha,\beta}p-\sum_{  l'>l}\big(\oH^{k,k,l',l'}_{\alpha,\beta,\gamma,\delta}+\oH^{k,k,-l',0}_{\alpha,\beta,\gamma,\delta}\big)]
 \wedge  [\oGamma^{l}_{\gamma,\delta}p-\sum_{  k'>k}\big(\oH^{k',k',l,l}_{\alpha,\beta,\gamma,\delta}+\oH^{0,k',l,l}_{\alpha,\beta,\gamma,\delta})]\nonumber\\
=&[J^{k,N}_{\alpha,\beta}-\sum_{ l'>l}(\oH^{k,k,l',l'}_{\alpha,\beta,\gamma,\delta}+\oH^{k,k,-l',0}_{\alpha,\beta,\gamma,\delta})]\wedge [J^{N,l}_{\gamma,\delta}-\sum_{  k'>k}\big(\oH^{k',k',l,l}_{\alpha,\beta,\gamma,\delta}+\oH^{0,k',l,l}_{\alpha,\beta,\gamma,\delta})]
\label{tmpbarl}
\end{align}
Note that since $J^{k,N}_{\alpha,\beta}=\Gamma^{k}_{\alpha,\beta}p-H^{k,k,N,N}_{\alpha,\beta,\gamma,\delta}+H^{k,k,-N,0}_{\alpha,\beta,\gamma,\delta}$ and $J^{N,l}_{\gamma,\delta}=\Pi^{0,l}_{\gamma,\delta}p-H^{N,N,l,l}_{\alpha,\beta,\gamma,\delta}-H^{0,N,l,l}_{\alpha,\beta,\gamma,\delta}$, by definition (\ref{HJ}) when $0<k,l\leq N-1$, 
\begin{align}
H^{k,k,l,l}_{\alpha,\beta,\gamma,\delta}=&[J^{k,N}_{\alpha,\beta}-\sum_{ l'>l}(H^{k,k,l',l'}_{\alpha,\beta,\gamma,\delta}+H^{k,k,-l',0}_{\alpha,\beta,\gamma,\delta})]\wedge [J^{N,l}_{\gamma,\delta}-\sum_{  k'>k}\big(H^{k',k',l,l}_{\alpha,\beta,\gamma,\delta}+H^{0,k',l,l}_{\alpha,\beta,\gamma,\delta})].
\label{tmpbarlH}
\end{align}
By (\ref{tmpbarl}) and (\ref{tmpbarlH}), $H^{k,k,l,l}_{\alpha,\beta,\gamma,\delta}$ and $\oH^{k,k,l,l}_{\alpha,\beta,\gamma,\delta}$ have the same recursive definition when $0<k\leq N-1$, $0<l\leq N-1$.\\ 

\noindent Step 2: If $\beta$ is $k$-bad (otherwise $\oH^{0,k,0,l}_{\alpha,\beta,\gamma,\delta}=H^{0,k,0,l}_{\alpha,\beta,\gamma,\delta}=0$ for each $l>0$), by Definition \ref{couplingHgen} (\ref{HB1}) and Definition \ref{newsystem} 
\begin{align}
\oH^{0,k,0,l}_{\alpha,\beta,\gamma,\delta}=&[\oPi^{0,k}_{\alpha,\beta}p-\sum_{l'>l}\oH^{0,k,0,l'}_{\alpha,\beta,\gamma,\delta}]\wedge [\oPi^{0,l}_{\gamma,\delta}p-\sum_{ k'>k}(\oH^{0,k',0,l}_{\alpha,\beta,\gamma,\delta}+\oH^{k',k',0,l}_{\alpha,\beta,\gamma,\delta})].
\nonumber\\
=&[\Pi^{0,k}_{\alpha,\beta}p-\sum_{l'>l}\oH^{0,k,0,l'}_{\alpha,\beta,\gamma,\delta}]\wedge [
B^{N,l}_{\gamma,\delta}-\sum_{k'>k}(\oH^{0,k',0,l}_{\alpha,\beta,\gamma,\delta}+\oH^{k',k',0,l}_{\alpha,\beta,\gamma,\delta})].
\label{tmpbirth1}
\end{align}
As in previous case, since 
$B^{N,l}_{\gamma,\delta}=\Pi^{0,l}_{\gamma,\delta}p-H^{0,N,0,l}_{\alpha,\beta,\gamma,\delta}-H^{N,N,0,l}_{\alpha,\beta,\gamma,\delta}$, by definition (\ref{HB1}), $H^{0,k,0,l}_{\alpha,\beta,\gamma,\delta}$ and $\oH^{0,k,0,l}_{\alpha,\beta,\gamma,\delta}$ have the same recursive definition when $0<k\leq N-1$, $0<l\leq N$. Notice that in this case they coincide even if $l=N$, since we may have $\oPi^{0,N}_{\gamma,\delta}p>0$. By similar arguments and (\ref{HD1}) we prove the symmetric result with respect to death rates.\\

\noindent Step 3a: Suppose that $\beta$ is $k$-bad and $\overline{J}^{k,k-\delta+\beta}_{\alpha,\beta}>0$, that is Step $1$ was not enough to solve the attractiveness problem induced by $\overline{\Gamma}^{k}_{\alpha,\beta}p$ for $\esse$. By Definitions \ref{couplingHgen} (\ref{H31a})--(\ref{H31b}) and \ref{newsystem}   
\begin{align}
\oH^{k,k,0,l}_{\alpha,\beta,\gamma,\delta}=&[\oGamma^{k}_{\alpha,\beta}p-\sum_{l'\geq k-\delta+\beta}\oH^{k,k,l',l'}_{\alpha,\beta,\gamma,\delta}-\sum_{ l'>l}\oH^{k,k,0,l'}_{\alpha,\beta,\gamma,\delta}] \wedge [\oPi^{0,l}_{\gamma,\delta}p-\sum_{k'\geq k}\oH^{0,k',0,l}_{\alpha,\beta,\gamma,\delta}-\sum_{ k'> k}\oH^{k',k',0,l}_{\alpha,\beta,\gamma,\delta})]\nonumber \\
=&[J^{k,N}_{\alpha,\beta}-\sum_{l'\geq k-\delta+\beta}\oH^{k,k,l',l'}_{\alpha,\beta,\gamma,\delta}-\sum_{ l'>l}\oH^{k,k,0,l'}_{\alpha,\beta,\gamma,\delta}]\wedge [B^{N,l}_{\gamma,\delta}-\sum_{k'\geq k}\oH^{0,k',0,l}_{\alpha,\beta,\gamma,\delta}-\sum_{ k'> k}\oH^{k',k',0,l}_{\alpha,\beta,\gamma,\delta})].
\label{tmpthird1}
\end{align}
Since $J^{k,N}_{\alpha,\beta}=\Gamma^{k}_{\alpha,\beta}p-H^{k,k,N,N}_{\alpha,\beta,\gamma,\delta}-H^{k,k,0,N}_{\alpha,\beta,\gamma,\delta}$ and $B^{N,l}_{\gamma,\delta}=\Pi^{0,l}_{\gamma,\delta}p-H^{0,N,0,l}_{\alpha,\beta,\gamma,\delta}-H^{N,N,0,l}_{\alpha,\beta,\gamma,\delta}$, again $H^{k,k,0,l}_{\alpha,\beta,\gamma,\delta}$ and $\oH^{k,k,0,l}_{\alpha,\beta,\gamma,\delta}$ have the same recursive definition when $0<k\leq N-1$, $0<l\leq N$. 
The equality of recursive formulas corresponding to Step 3b is proved in a similar way.\\

Since all terms are defined by the same downwards induction formula, we just need to check that they coincide for the initial coupling rates. The one involving $\Gamma^{N-1}_{\alpha,\beta}p$ and $\Gamma^{N-1}_{\gamma,\delta}p$ is given by (\ref{tmpbarl}) with $l=N-1$, that is  
\begin{align*}
\oH^{N-1,N-1,N-1,N-1}_{\alpha,\beta,\gamma,\delta}=&J^{N-1,N}_{\alpha,\beta}\wedge J^{N,N-1}_{\gamma,\delta}=[\Gamma^{N-1}_{\alpha,\beta}p-H^{N-1,N-1,N,N}_{\alpha,\beta,\gamma,\delta}+H^{N-1,N-1,-N,0}_{\alpha,\beta,\gamma,\delta}]\\
&\wedge [\Gamma^{N-1}_{\gamma,\delta}p-H^{N,N,N-1,N-1}_{\alpha,\beta,\gamma,\delta}-H^{0,N,N-1,N-1}_{\alpha,\beta,\gamma,\delta}]=H^{N-1,N-1,N-1,N-1}_{\alpha,\beta,\gamma,\delta}
\end{align*}
and we are done. The one involving $\Pi^{0,N-1}_{\alpha,\beta}p$ is given by (\ref{tmpbirth1}) with $l=N$, that is
\begin{align*}
\oH^{0,N-1,0,N}_{\alpha,\beta,\gamma,\delta}=&\Pi^{0,N-1}_{\alpha,\beta}p\wedge B^{N,N}_{\gamma,\delta}=
\Pi^{0,N-1}_{\alpha,\beta}p\wedge [\Pi^{0,N}_{\gamma,\delta}p-H^{0,N,0,N}_{\alpha,\beta,\gamma,\delta}-H^{N,N,0,N}_{\alpha,\beta,\gamma,\delta}]\\
=&H^{0,N-1,0,N}_{\alpha,\beta,\gamma,\delta}.
\end{align*}
We prove that $H^{-N,0,-(N-1),0}_{\alpha,\beta,\gamma,\delta}=H^{-N,0,-(N-1),0}_{\alpha,\beta,\gamma,\delta}$ for death rates by symmetric arguments.\\
Therefore $\mathcal{H}$ and $\overline{\mathcal{H}}$ are identical for all coupling rates involving a change of less than $N$ particles which cause an attractiveness problem: the conclusion includes uncoupled rates. The claim follows since $\overline{\mathcal{H}}$ is increasing and by Proposition \ref{imp}.
\end{proof}

\subsubsection{Proof of Proposition \ref{suf}} 
We do an induction in two steps. First of all we suppose that $\beta$ is $N$-bad and $\gamma$ is $N$-good, that is $\beta+N>\delta$ and $\gamma-N\geq \alpha$. We skip the similar symmetric case.\\
Let $\mathcal{S}^*_N:=\mathcal{S}$ and suppose that Proposition \ref{suf} holds for each system $\mathcal{S}^*$ such that $n(\mathcal{S}^*)\leq N-1$. Notice that since $\gamma$ is $N$-good, then $\gamma$ is $n(\mathcal{S}^*)$-good for each $\mathcal{S}^*$ and, by Remark \ref{rem_enne}, definition of $n(\mathcal{S}^*)$ involves for each $\mathcal{S}^*$ only rates that cause a lower attractiveness problem.\\
Remember that Definition \ref{newsystem} of $\esse$ depends on the original system, that is $\esse=\esse(\mathcal{S})$. We define $\mathcal{S}^*_{j}=\esse(\mathcal{S}^*_{j+1})$ and $n_{j}=n(\mathcal{S}^*_{j})$ for $j\in \mathbb{N}, j\leq N-1$.\\
If $\beta$ is $(N-1)$-bad, by Proposition $\ref{NLU}$, $n_{N-1}\leq N-1$, and $\mathcal{S}^*_{N-1}$ satisfies  Conditions $(\ref{C+}$)--($\ref{C-})$ by Proposition $\ref{imp}$. We define a coupling $\mathcal{H}_{N-1}$ for $\mathcal{S}^*_{N-1}$ as in Definition $\ref{couplingHgen}$, and by induction hypothesis it is increasing. By Proposition \ref{HHbar}, if $\mathcal{H}_{N-1}$ is increasing then so is $\mathcal{H}_N$.\\
We have to check the induction basis: we proceed downwards with definitions of the new systems $\mathcal{S}^*_j$ until $\beta+j\leq\delta$. If $\beta<\delta$, then $\mathcal{S}^*_i$ with $j=\delta-\beta$ is attractive, since $\beta+j\leq\delta$ and there are no attractiveness problems. If $\beta=\delta$, then the attractiveness of a system with $n_1=1$ under Condition (\ref{C+}) is proved in Appendix \ref{app}.\\

We proved that all systems such that there is only a lower (or only a higher) attractiveness problem are attractive under Conditions (\ref{C+})--(\ref{C-}). Now we prove that this holds also for a system where both $\beta$ and $\gamma$ are $N$-bad.\\

If $\beta+N>\delta$ and $\gamma-N<\alpha$, the definition of $n(\mathcal{S})$ is given by (\ref{enne}) but everything works in a similar way: we define $\mathcal{S}^*_{N}=\mathcal{S}$, $\mathcal{S}^*_{j}=\esse(\mathcal{S}^*_{j+1})$ and $n_{j}=n(\mathcal{S}^*_{j})$ for $j\leq N-1$.\\
If both $\beta$ and $\gamma$ are $(N-1)$-bad, by Proposition $\ref{NLU}$, $n_{N-1}\leq N-1$ and $\mathcal{S}^*_{N-1}$ satisfies  Conditions $(\ref{C+}$)--($\ref{C-})$ by Proposition $\ref{imp}$. We define a coupling $\mathcal{H}_{N-1}$ for $\mathcal{S}^*_{N-1}$ as in Definition $\ref{couplingHgen}$; by induction hypothesis it is increasing. By Proposition \ref{HHbar}, if $\mathcal{H}_{N-1}$ is increasing then so is $\mathcal{H}_N$.\\
We have to check the induction basis: we proceed downwards with the definition of a new system $\mathcal{S}^*_j$ until either $\beta$ or $\gamma$ are $j$-good, that is until $j=(\delta-\beta)\vee (\gamma-\alpha)$.\\
If $\delta-\beta\neq \gamma-\alpha$, then only one attractiveness problem is present and the claim follows from the first part of the proof.\\ 
If $j=\delta-\beta=\gamma-\alpha>0$, then $\mathcal{S}^*_j$ is attractive since there are no attractiveness problems; if $j=\delta-\beta=\gamma-\alpha=0$, then the attractiveness of a system with $n_1=1$, $\beta=\delta$ and $\gamma=\alpha$ under Conditions (\ref{C+})--(\ref{C-}) is proved in Appendix \ref{app}.\hfill $\square$
 
\subsection{Sufficient conditions on more general systems}  
In order to show the sufficient conditions of Theorem \ref{cns}, we restricted ourselves to transition rates on a given pair of sites $(x,y)$: if Conditions (\ref{C+})--(\ref{C-}) are satisfied we can construct an increasing coupling for the system $\mathcal{S}_{(x,y)}$ and the final increasing coupling is given by superposition of couplings for all pairs of sites (see Section \ref{coupling} and Proposition \ref{attpairs}). We use neither the translation invariance of $P^{\pm k}_{\eta(x)}$, nor the fact that the smaller and the larger systems share the same $p(x,y)$. Therefore we can state the result for a more general pair of systems  $\mathcal{G}=\{R^{0,\pm k}_{\eta(x),\eta(y)}(x,y),\Gamma^{k}_{\eta(x),\eta(y)}(x,y), P^{\pm k}_{\eta(x)}(x),p(x,y))\}$ and $\widetilde{\mathcal{G}}=\{\widetilde{R}^{0,\pm k}_{\eta(x),\eta(y)},$
$\widetilde{\Gamma}^{k}_{\eta(x),\eta(y)}, \widetilde{P}^{\pm k}_{\eta(x)}(x),\widetilde{p}(x,y)\}$:
\begin{corollario}
A particle system $\eta_t \sim \mathcal{G}$ is stochastically larger than $\xi_t\sim \widetilde{\mathcal{G}}$ if for each $(x,y)\in S^2$, $(\alpha,\beta)\leq (\gamma,\delta)$, $(\alpha,\beta)\in X^2$, $(\gamma,\delta)\in X^2$, with $p=p(x,y)$ and $\widetilde{p}=\widetilde{p}(x,y)$
\begin{align*}
&\big(\sum_{k > \delta-\beta+j_{1}}\widetilde{\Pi}_{\alpha,\beta}^{0,k}(x,y)+\sum_{k \in I_{a}}\widetilde{\Gamma}_{\alpha,\beta}^{k}(x,y)\big)\widetilde{p}\leq 
\big(\sum_{l>j_{1}}\Pi_{\gamma,\delta}^{0,l}(x,y)+\sum_{l \in I_{b}}\Gamma_{\gamma,\delta}^{l}(x,y)\big)p,\\
&\big(\sum_{k>h_{1}}\widetilde{\Pi}_{\alpha,\beta}^{-k,0}(x,y)+\sum_{k \in I_{d}}\widetilde{\Gamma}_{\alpha,\beta}^{k}(x,y)\big)\widetilde{p}\geq 
\big(\sum_{l > \gamma-\alpha+h_{1}}\Pi_{\gamma,\delta}^{-l,0}(x,y)+\sum_{l \in I_{c}}\Gamma_{\gamma,\delta}^{l}(x,y)\big)p,
\end{align*}
for all choices of $K\leq \widetilde{N}(\alpha,\beta)\vee N(\gamma,\delta)$, $\mathbf{k}$, $\mathbf{j}$, $\mathbf{m}$ in Theorem \ref{cns}.
\label{cor_cns}
\end{corollario} 
Such conditions are not necessary. We use Corollary \ref{cor_cns} with the comparison technique with oriented percolation in \cite{cf:Borrello} to prove survival of species in metapopulation models.    
%\begin{corollario}
%If $N=1$ then Theorem \ref{cor_cns} becomes, as in Proposition \ref{condM1}: for each $(x,y) \in S^2$,
%\begin{align}
%(\widetilde{\Pi}_{\alpha, \beta}^{0,1}+\widetilde{\Gamma}_{\alpha, \beta}^{1})\widetilde{p}\leq & \Pi_{\gamma, \delta}^{0,1}+\Gamma_{\gamma, \delta}^{1}p   & \text {if } \beta=\delta \text{ and } \gamma \geq \alpha,
%\label{C+1new}\\
%\widetilde{\Pi}_{\alpha, \beta}^{0,1}\widetilde{p}\leq & \Pi_{\gamma, \delta}^{0,1}p  & \text {if } \beta=\delta \text{ and } \gamma = \alpha,
%\label{C+10new}\\
%(\widetilde{\Pi}_{\alpha, \beta}^{-1,0}+\widetilde{\Gamma}_{\alpha, \beta}^{1})\widetilde{p}\geq &(\Pi_{\gamma, \delta}^{-1,0}+\Gamma_{\gamma, \delta}^{1})p   & \text {if } \gamma=\alpha \text{ and } \delta \geq \beta,
%\label{C-1new}\\
%\widetilde{\Pi}_{\alpha, \beta}^{-1,0}\widetilde{p}\geq &\Pi_{\gamma, \delta}^{-1,0}p & \text {if } \gamma=\alpha \text{ and } \delta = \beta.
%\label{C-10new}
%\end{align}
%\label{cor1_new}
%\end{corollario}
%We guess that with a more complicated coupling construction which mixes also transitions from more than two sites one can prove that Conditions (\ref{C+})--(\ref{C-}) are the necessary and sufficient ones also for this pair of systems. 
 
\section{Proof of Theorem \ref{general_erg}}
\label{proof_appl}
Since $F(x,y)=F(0,y-x)=F(y-x,0)$, with a slight abuse of notation we write $F(y-x)$ instead of $F(x,y)$.\\
 
\noindent \begin{proof} \emph{of Proposition \ref{ucriterio}}.\\
We treat separately the cases $\gamma=0$ and $\gamma>0$.\\
$i)$ If $\gamma=0$ then the Dirac measure $\delta_{\underline{0}}$ is invariant. We denote by $\mathbb{P}(\cdot)$ the independent coupling measure and by $\mathbb{E}(\cdot)$ its expected value. We fix $x \in S$ and we compute the generator on $F(\eta^{M}_{t}(x))$ in Definition $(\ref{F})$, where $\{u_{l}(\epsilon)\}_{l\in X}$ satisfies Hypothesis $(\ref{condu})$. Let $t \geq 0$, we denote by $\eta^M_{t}(x)=l$. By $(\ref{condu})$
\begin{align}
\mathcal{L}F(\eta^M_{t}(x))=&\mathcal{L}F(l)=\I_{\{l=0\}}\Big[\lambda \sum_{y\sim x}\eta^M_{t}(y)(F(1)-F(0))\Big]\nonumber\\
&+\I_{\{1\leq l\leq M-1\}}\Big[(\beta \sum_{y\sim x}\eta^M_{t}(y)+l\phi)(F(l+1)-F(l))+l(F(l-1)-F(l))\Big]\nonumber\\
&+\I_{\{l=M\}}\Big[M(F(M-1)-F(M))\Big]=\I_{\{l=0\}}\Big[\lambda \sum_{y\sim x}\eta^M_{t}(y)u_{0}\Big]\nonumber\\
&+\I_{\{1\leq l\leq M-1\}}\Big[(\beta \sum_{y\sim x}\eta^M_{t}(y)+l\phi)u_{l}-lu_{l-1}\Big]+\I_{\{l=M\}}\Big[-Mu_{M-1}\Big]\nonumber\\
\leq &\I_{\{l=0\}}\Big[(\lambda \sum_{y\sim x}\eta^M_{t}(y)+l\phi)u_{l}-lu_{l-1}\Big]\nonumber\\
&+\I_{\{1\leq l\leq M\}}\Big[(\beta \sum_{y\sim x}\eta^M_{t}(y)+l\phi)u_{l}-lu_{l-1}\Big]\nonumber \\
\leq &(\lambda\vee \beta) \bar{u}\sum_{y\sim x}\eta^M_{t}(y)+l(\phi u_{l}-u_{l-1})\nonumber \\
\leq& (\lambda\vee \beta)\bar{u}\sum_{y\sim x}\eta^M_{t}(y)-\epsilon \sum_{j=0}^{l-1}u_{j}-(\lambda\vee \beta)2d \bar{u}l.
\label{tmp-ucrit}
\end{align}
Let $\eta^{M}_{0}\in \Omega^{M}$. By translation invariance and Definition (\ref{F})
\begin{align*}
\frac{d}{dt}\mathbb{E}(F(\eta^M_{t}(x)))=&\mathbb{E}(\mathcal{L}F(\eta^{M}_{t}(x)))\leq \mathbb{E}\Big((\lambda\vee \beta)\bar{u}\sum_{y\sim x}\eta^M_{t}(y)-\epsilon \sum_{j=0}^{\eta^{M}_{t}(x)-1}u_{j}-(\lambda\vee \beta)2d \bar{u}\eta^{M}_{t}(x)\Big)\\
=& -\epsilon \mathbb{E}(F(\eta^M_{t}(x))).
\end{align*}
By Gronwall's Lemma we get the result.\\

$ii)$ If $\gamma>0$ then $\delta_{\underline{0}}$ is not any more an invariant measure. Let $(\xi_{t},\eta_{t})_{t\geq 0}$ be a coupled process through the basic coupling probability measure $\widetilde{\mathbb{P}}$ such that $\xi_{0}\in \Omega^{0}$ and $\eta_{0}\in \Omega^{M}$. We fix $x \in \mathbb{Z}^{d}$. We denote by $k=\xi_{t}(x)$, $l=\zeta_{t}(x):=\eta_{t}(x)-\xi_{t}(x)$ and $k+l=\eta_{t}(x)$. First of all we prove
\begin{lemma}
%Let $(\xi_{t})_{t\geq 0}$, $(\eta_{t})_{t\geq 0}$ be two epidemic models such that $\xi_{0}\leq \eta_{0}$ and 
%$$
%\beta-\lambda\leq \frac{\gamma}{2dN}.
%$$
%Let $F$ as in (\ref{F}) with $\{u_{l}\}_{l\in X}$ non increasing in $l$, and let $x \in \mathbb{Z}^{d}$. Then
\begin{align}
\mathcal{L}F(k,k+l)\leq& \lambda\sum_{y \sim x}(\eta_{t}(y)-\xi_{t}(y))u_{l}+l\big(\phi u_{l}-u_{l-1}\big).
\end{align}
\label{lem_gamma}
\end{lemma}
\begin{proof}.
By Definition (\ref{F}), 
\begin{align}
\mathcal{L}F(\xi_{t}(x),\eta_{t}(x))=&\I_{\{k=l=0\}}\Big[\lambda \sum_{y\sim x}\zeta_{t}(y)u_{l}\Big]+\I_{\{k=0, 0<l\leq M-1\}}\Big[(\lambda \sum_{y\sim x}\xi_{t}(y)\nonumber\\
&+\gamma)(-u_{l-1})+(k+l)(-u_{l-1})+(\beta \sum_{y\sim x}\eta_{t}(y)+l\phi)u_{l}\Big]\nonumber\\
&+\I_{\{k=0,l=M\}}\Big[(\lambda \sum_{y\sim x}\xi_{t}(y)+\gamma)(-u_{l-1})+M(-u_{M-1})\Big]\nonumber\\
&+\I_{\{k>0,k+l\leq M-1\}}\Big[(\beta \sum_{y\sim x}\zeta_{t}(y)+l\phi)u_{l}+l(-u_{l-1})\Big]\nonumber\\
&+\I_{\{k\leq M-1,k+l=M\}}\Big[(\beta \sum_{y\sim x}\xi_{t}(y)+k\phi)(-u_{l-1})+l(-u_{l-1})\Big]\label{gen_erg2}.
\end{align}
We prove that 
\begin{align}
\mathcal{L}F(k,k+l)\leq& \I_{\{k=0,0<l\leq M-1\}}\Big[-(\lambda \sum_{y\sim x}\xi_{t}(y)+\gamma)u_{l-1}-lu_{l-1}+(\beta \sum_{y\sim x}\eta_{t}(y)+l\phi)u_{l}\Big]\nonumber \\
&+\I_{\{\{k=0,0<l\leq M-1\}^{c}\}}\Big[(\lambda \vee \beta)\sum_{y \sim x}\zeta_{t}(y)u_{l}+l\phi u_{l}-lu_{l-1}\Big].
\label{tmpgen1}
\end{align}
Notice that if $k=0$ then $\eta_{t}(x)=k+l=l$.\\
If $k\geq 0$ and $k+l=M$, then the last term in the right hand side of (\ref{gen_erg2}) is smaller or equal to
$$
\beta\sum_{y \sim x}\zeta_{t}(y)u_{l}+l\phi u_{l}-lu_{l-1}.
$$
%If $k=0$ and $l=N$, 
%$$
%-\lambda\sum_{y \sim x}\xi_{t}(y)u_{l-1}-(\gamma+N)u_{l-1} \leq \beta\sum_{y \sim x}\zeta_{t}(y)u_{l}+l\phi. %u_{l}-lu_{l-1}
%$$
The same inequality holds when $l=0$ and $k=M$, that is $\xi_{t}(x)=\eta_{t}(x)=M$, since the last term in the right hand side of $(\ref{gen_erg2})$ is null and $l\phi u_{l}-lu_{l-1}=0$.
%$$
%0=l\leq \beta\sum_{y \sim x}\zeta_{t}(y)u_{l}=\beta\sum_{y \sim x}\zeta_{t}(y)u_{l}+l\phi u_{l}-lu_{l-1}.
%$$
If $k=0$ and $l=0$, 
$$
\lambda\sum_{y \sim x}\zeta_{t}(y)u_{l}= \lambda\sum_{y \sim x}\zeta_{t}(y)u_{l}+l\phi u_{l}-lu_{l-1}.
$$
Therefore $(\ref{tmpgen1})$ holds.\\
Since $u_{l}$ is non increasing in $l$,
$$
-(\lambda \sum_{y\sim x}\xi_{t}(y)+\gamma)u_{l-1}\leq -(\lambda \sum_{y\sim x}\xi_{t}(y)+\gamma)u_{l}
$$
If $\beta\leq \lambda$,
\begin{align*}
-&(\lambda \sum_{y\sim x}\xi_{t}(y)+\gamma)u_{l-1}-lu_{l-1}+(\beta \sum_{y\sim x}\eta_{t}(y)+l\phi)u_{l}\\
&\leq l\phi u_{l}-lu_{l-1}+\lambda \sum_{y \sim x}\zeta_{t}(y)u_{l}
\end{align*}
and by $(\ref{tmpgen1})$ the claim  follows; if $\beta>\lambda$ 
\begin{align}
l&\phi u_{l}-lu_{l-1}+\beta \sum_{y \sim x}\eta_{t}(y)u_{l}-\lambda \sum_{y \sim x}\xi_{t}(y)u_{l-1}-\gamma u_{l-1}\nonumber \\
&=l\phi u_{l}-lu_{l-1}+\lambda \sum_{y \sim x}\zeta_{t}(y)u_{l}+\Big((\beta-\lambda) \sum_{y \sim x}\eta_{t}(y)-\gamma\Big) u_{l}.
\label{tmpB1}
\end{align}
Since $\beta-\lambda\leq \gamma/(2dM)$, then $(\beta-\lambda) \sum_{y \sim x}\eta_{t}(y)-\gamma\leq 0$ and the claim follows from ($\ref{tmpgen1}$) and ($\ref{tmpB1}$).
\end{proof}
By Lemma \ref{lem_gamma}, (\ref{condu}) and translation invariance
\begin{align*}
\widetilde{\mathbb{E}}\Big(\mathcal{L}F(k,k+l)\Big)&\leq \widetilde{\mathbb{E}}\Big(l\phi u_{l}-lu_{l-1}+\lambda \sum_{y \sim x}\zeta_{t}(y)u_{l}\Big)\\
&\leq \widetilde{\mathbb{E}}\Big(-\epsilon\sum_{j=0}^{l-1}u_{j}-\lambda U 2dl+\lambda \sum_{y \sim x}\zeta_{t}(y)U\Big)=-\epsilon\widetilde{\mathbb{E}}\Big(F(k,k+l)\Big).
\end{align*}
and the claim follows by Gronwall's Lemma.
\end{proof}
Now we prove that Definition \ref{ul} satisfies the hypothesis of the $u$-criterion under (\ref{cond}). We begin with a technical proposition.
\begin{proposizione}
Let $(u_{l}(\epsilon))_{l\in X}$ be given by Definition \ref{ul}. If $\phi<1$ and
$$
\frac{1-\phi}{2d}<\lambda \vee \beta<\frac{1-\phi}{2d(1-\phi^{M})}
$$ 
then there exists $\bar{\epsilon}>0$ such that $u_{l}(\epsilon)$ is positive, decreasing in $l$ for each $l\in X$ and in $\epsilon$ for each $0<\epsilon\leq \bar{\epsilon}$.
\label{decreasing}
\end{proposizione}
We prove Proposition $\ref{decreasing}$ by induction on $l$. We need the following lemma to use the induction hypothesis:
\begin{lemma}
If for $l\in X$,
$$
\lambda \vee \beta<\frac{1-\phi}{2d(1-\phi^{l})}
$$ 
and there exists $\bar{\epsilon}$ such that $u_{k}(\epsilon)>0$ for each $k\leq l-1$, $0<\epsilon\leq \bar{\epsilon}$; then there exists $0<\epsilon^{*}\leq \bar{\epsilon}$ such that 
\begin{equation}
U< \frac{u_{l-1}(\epsilon^*)}{2d(\lambda \vee \beta)}.
\label{condU1}
\end{equation}
\label{lemappl1}
\end{lemma}
\begin{proof}.
We prove by a downwards induction on $0\leq k \leq l-1$ that:\\
if there exists $0<\epsilon^* \leq \bar{\epsilon}$ such that
\begin{align}
J(k,\epsilon^*):=-(\lambda \vee \beta)2d U(1+\phi+\ldots+\phi^k)+u_{l-k-1}(\epsilon^*)>0
\label{disu_new}
\end{align}
then $J(k-1,\epsilon^*)>0$.\\
Indeed, if $k=l-1$ then $(\ref{disu_new})$ is the assumption $(\lambda \vee \beta)2d (1-\phi^l)<1-\phi$.\\
Suppose there exists $0<\epsilon^*\leq \bar{\epsilon}$ such that $J(k,\epsilon^*)>0$. By Definition $\ref{ul}$, $(\ref{disu_new})$ is equivalent to
$$
-(\lambda \vee \beta)2d U(1+\ldots+\phi^k)+\frac{-\epsilon^*\sum_{j=0}^{l-k-2}u_{j}(\epsilon^*)+(l-k-1)\big(-U(\lambda \vee \beta)2d+u_{l-k-2}(\epsilon^*)\big)}{\phi(l-k-2)}>0
$$
that is 
\begin{align*}
-(\lambda \vee \beta)2d U(l-k-1)\big[(1+\ldots+\phi^k)\phi+1\big]-\epsilon^*\sum_{j=0}^{l-k-2}u_{j}(\epsilon^*)+(l-k-1)u_{l-k-2}(\epsilon^*)>0.
\end{align*}
Therefore 
\begin{align*}
J(k-1,\epsilon^*)>-(\lambda \vee \beta)2d U(1+\phi+\ldots+\phi^{k+1})+u_{l-k-2}(\epsilon^*)>\frac{\epsilon^*}{l-k-1}\sum_{j=0}^{l-k-2}u_{j}(\epsilon^*).
\end{align*}
Since by hypothesis $u_{j}(\epsilon)>0$ for each $0<\epsilon\leq \bar{\epsilon}$, then $J(k-1,\epsilon^*)>0$ and by induction $(\ref{disu_new})$ holds for each $0\leq k \leq l-1$. By taking $k=0$ we get 
\begin{align*}
-(\lambda \vee \beta)2d U+u_{l-1}(\epsilon^*)>0
\end{align*}
which is the claim.
%Notice that if we prove the existence of $\epsilon^{*}\leq \bar{\epsilon}$ which satisfies $(\ref{condU1})$, since by 
%\begin{align}
%If $\epsilon^{*}_{1}\leq \epsilon'_{1}$ then $\displaystyle\epsilon^{*}_{1}\sum_{j=0}^{l-2}u_{j}(\epsilon^{*}_{1})\leq \epsilon'_{1}\sum_{j=0}^{l-2}u_{j}(\epsilon^{*}_{1})$. By $(\ref{tmpast0})$ and $(\ref{tmpast2})$ Condition $
\end{proof}
\begin{proof}\textit{of Proposition $\ref{decreasing}$.}\\
We prove by induction on $l\in X$ that there exists $\bar{\epsilon}_l$ such that for each $0<\epsilon<\bar{\epsilon}_l$ and for each $0\leq j\leq l$, $u_j(\epsilon)$ is positive, decreasing in $j$, $\displaystyle U<\frac{u_{j}(\epsilon)}{2d(\lambda \vee \beta)}$ and
\begin{equation}
0>\frac{d}{d\epsilon}u_{j}(\epsilon)\geq -C_{U}(j)
\label{derivata}
\end{equation}
with $C_{U}(j)$ the solution of 
\begin{equation}
C_{U}(1)=\frac{U}{\phi}, \quad C_{U}(j)=\frac{C_{U}(j-1)+U}{\phi}.
\label{CUl}
\end{equation}
This gives
\begin{align}
C_{U}(j)=\frac{U}{\phi^{j}}(1+\phi+\phi^{2}+\ldots +\phi^{j-1})=U\frac{1-\phi^{j}}{\phi^{j}(1-\phi)}
\label{CUlformula}
\end{align}
hence $u_j(\epsilon)>0$ is also decreasing in $\epsilon$. \\
We prove the induction basis when $l=1$. By Definition $\ref{ul}$
$$
u_{1}(\epsilon)=\frac{-\epsilon u_{0}-U(\lambda \vee \beta)2d+u_{0}}{\phi}=u_{0}\frac{-\epsilon-(\lambda \vee \beta)2d+1}{\phi}.
$$
Since $(\lambda \vee \beta)2d>1-\phi$ we can take $\epsilon<\epsilon_{1}$ small enough to have $1-(\lambda \vee \beta)2d-\epsilon<\phi$, that is $u_{1}(\epsilon)<u_{0}=U$; since $(\lambda \vee \beta)2d< 1$ then $U<U/(2d(\lambda \vee \beta))$ and, by taking $\epsilon<\epsilon^{*}_{1}$ small enough, $u_{1}(\epsilon)$ is positive; moreover notice that $u_{1}(\epsilon)$ is decreasing in $\epsilon$, and $\frac{d}{d\epsilon}u_{1}(\epsilon)=-\frac{U}{\phi}=-C_{U}(1)$ for each $\epsilon$. Hence the induction basis as well as the hypothesis of Lemma $\ref{lemappl1}$ are satisfied for $l=1$: there exists $\bar{\epsilon}_{1}=\epsilon_{1}\wedge \epsilon^*_{1}$ such that if $\epsilon\leq \bar{\epsilon}_{1}$ then $U>u_{1}(\epsilon)>0$, $u_{1}(\epsilon)$ is decreasing in $\epsilon$, and $(\ref{condU1})$ holds.\\
Suppose there exists $\bar{\epsilon}_{l-1}>0$ such that for each $0<\epsilon\leq\bar{\epsilon}_{l-1}$, then $0>\frac{d}{d\epsilon}u_{j}(\epsilon)\geq -C_{U}(j)$, $u_{j}(\epsilon)$ is decreasing in $j$ for each $j\leq l-1$, $\displaystyle U< \frac{u_{j}(\epsilon)}{2d(\lambda \vee \beta)}$ and $u_{j}(\epsilon)>0$ for each $j\leq l-1$.\\
First of all we prove that there exists $\epsilon_{l}>0$ such that if $0<\epsilon<\epsilon_{l}$ then $u_{l}(\epsilon)<u_{l-1}(\epsilon)$. By Definition $\ref{ul}$
\begin{equation}
u_{l}(\epsilon)=u_{l-1}(\epsilon)\Big(\frac{-\epsilon\sum_{j=0}^{l-1}u_{j}(\epsilon)-U(\lambda \vee \beta)2dl+lu_{l-1}(\epsilon)}{\phi lu_{l-1}(\epsilon)}\Big).
\label{tmpind1}
\end{equation}
By induction hypothesis, if $\epsilon\leq \bar{\epsilon}_{l-1}$ then $U>u_{j}(\epsilon)>0$ for each $j\leq l-1$ and we get
\begin{align*}
\frac{-\epsilon\sum_{j=0}^{l-1}u_{j}(\epsilon)-U(\lambda \vee \beta)2dl+lu_{l-1}(\epsilon)}{\phi lu_{l-1}(\epsilon)}<&\frac{-U(\lambda \vee \beta)2dl+lu_{l-1}(\epsilon)}{\phi lu_{l-1}(\epsilon)}\\
<&\frac{1}{\phi}-\frac{(\lambda \vee \beta)2d}{\phi}<1
\end{align*}
that is $u_{l}(\epsilon)<u_{l-1}(\epsilon)$ by $(\ref{tmpind1})$ and we set $\epsilon_{l}=\bar{\epsilon}_{l-1}$.\\
By $(\ref{CUlformula})$, $C_{U}(j)$ is always positive and increasing in $j$. We prove that $(\ref{derivata})$ holds for $j=l$. By Definition $\ref{ul}$
\begin{equation}
\frac{d}{d\epsilon}u_{l}(\epsilon)=\frac{1}{\phi l}\Big(-\epsilon\frac{d}{d\epsilon}\sum_{j=0}^{l-1}u_{j}(\epsilon)-\sum_{j=0}^{l-1}u_{j}(\epsilon)+l\frac{d}{d\epsilon}u_{l-1}(\epsilon)\Big).
\label{derepsilon}
\end{equation}
We begin with the right inequality involving the derivative of $u_{l}(\epsilon)$ in ($\ref{derivata}$). By induction hypothesis, if $0<\epsilon\leq \bar{\epsilon}_{l-1}$ by $(\ref{CUl})$, $(\ref{CUlformula})$ and $(\ref{derepsilon})$ on the one hand 
\begin{align*}
\phi l\frac{d}{d\epsilon}u_{l}(\epsilon)>&-lU-lC_{U}(l-1)=-C_{U}(l)l\phi.\\
\end{align*}
and on the other hand
\begin{align*}
\phi l\frac{d}{d\epsilon}u_{l}(\epsilon)\leq &\epsilon\sum_{j=0}^{l-1}C_{U}(j)-\sum_{j=0}^{l-1}u_{j}(\epsilon)+l\frac{d}{d\epsilon}u_{l-1}(\epsilon)
<\epsilon\sum_{j=0}^{l-1}U\frac{1-\phi^{j}}{\phi^{j}(1-\phi)}-\sum_{j=0}^{l-1}u_{j}(\epsilon).
\end{align*}
Since $0<\epsilon<\bar{\epsilon}_{l-1}$ and $u_{j}(\epsilon)$ is positive and decreasing in $\epsilon$, as $\epsilon$ approaches $0$ the first sum on the right hand side goes to $0$, while the second sum is positive and increasing: therefore for each $j\leq l-1$, we can take $\hat{\epsilon}_{l}$ small enough so that $\frac{d}{d\epsilon}u_{l}(\epsilon)<0$ for each $0<\epsilon\leq \hat{\epsilon}_{l}$.\\
Now we prove that there exists $\epsilon^{*}_{l}>0$ such that $u_{l}(\epsilon)>0$ for each $0<\epsilon\leq \epsilon^{*}_{l}$. By Definition $\ref{ul}$, $u_{l}(\epsilon)>0$ if 
\begin{equation}
-U(\lambda \vee \beta)2dl+lu_{l-1}(\epsilon)>\epsilon\sum_{j=0}^{l-1}u_{j}(\epsilon).
\label{Uco}
\end{equation}
By the induction hypothesis, assumptions of Lemma $\ref{lemappl1}$ are satisfied, hence there exists $0<\epsilon^{*}\leq \bar{\epsilon}_{l-1}$ such that (\ref{condU1}) is satisfied. Thus $-U(\lambda \vee \beta)2dl+lu_{l-1}(\epsilon^{*})>0$ and we can choose $\epsilon^{*}_{l}\leq \epsilon^*$ such that 
\begin{align}
-U(\lambda \vee \beta)2dl+lu_{l-1}(\epsilon^{*})>\epsilon^{*}_{l}lU
\end{align}
If $\epsilon\leq \epsilon_l^{*} (\leq \bar{\epsilon}_{l-1}$)
\begin{align*}
-U(\lambda \vee \beta)2dl+lu_{l-1}(\epsilon)>&-U(\lambda \vee \beta)2dl+lu_{l-1}(\epsilon^{*})\\
>&\epsilon^{*}_{l}lU>\epsilon^{*}_{l}\sum_{j=0}^{l-1}u_{j}(\epsilon)>\epsilon\sum_{j=0}^{l-1}u_{j}(\epsilon)
\end{align*}
which is $(\ref{Uco})$. By taking $0<\epsilon<\bar{\epsilon}_{0}\wedge \ldots \bar{\epsilon}_{l-1}\wedge \hat{\epsilon}_{l}\wedge \epsilon^{*}_{l}$, $u_{l}(\epsilon)>0$ and is decreasing in $l$ and $\epsilon$, and the claim follows.
\end{proof}
\begin{proof}\textit{of Theorem \ref{general_erg}.}\\
Let $\{u_{l}(\epsilon)\}_{l\in X}$ be given by Definition $(\ref{eq:ul})$. By Proposition $\ref{decreasing}$ we can choose $0<\epsilon<\bar{\epsilon}$ such that $u_{l}(\epsilon)>0$ for each $l \in X$ and ergodicity follows from Proposition $\ref{ucriterio}$.
\end{proof}

\noindent \textbf{Acknowledgments.} I thank the referee for a careful reading of the manuscript, and many comments that improve the presentation of the paper. I am grateful to Ellen Saada, the French supervisor of my Ph.D. Thesis, which was done in joint tutorage between LMRS, Universit\'e de Rouen and Universit\`a di Milano-Bicocca. I thank Institut Henri Poincar\'e, Centre Emile Borel for hospitality during the semester ``Interacting Particle Systems, Statistical Mechanics and Probability Theory'', where part of this work was done and Fondation Sciences Math\'ematiques de Paris for financial support during the stay. I acknowledge Laboratoire MAP5, Universit\'e Paris Descartes for hospitality and the support of the French Ministry of Education through the ANR BLAN07-218426 grant.

\appendix

\section{Appendix: Explicit coupling construction if $N=1$}
\label{app}
We detail the coupling when $N=1$ to understand the simplest construction. Note that we have to mix birth with jump rates also in this case.\\
 There is a lower attractiveness problem ($\beta+1 > \delta$) only if $\beta=\delta$, and a higher one ($\gamma-1<\alpha$) only if $\gamma=\alpha$. Definition $\ref{couplingHgen}$ becomes
\begin{definition}
The non zero coupling rates of $\mathcal{H}$ are given by:
\begin{align*}
&H^{1,1,1,1}_{\alpha,\beta,\gamma,\delta}=(\widetilde{\Gamma}^{1}_{\alpha,\beta}\wedge \Gamma^{1}_{\gamma,\delta})p \qquad \text{ if } \beta=\delta \text{ or } \gamma=\alpha;\\
\\
&\left.
\begin{array}{ll}
H^{0,1,0,1}_{\alpha,\beta,\gamma,\delta}=&(\widetilde{\Pi}^{0,1}_{\alpha,\beta}\wedge \Pi^{0,1}_{\gamma,\delta})p\\ H^{1,1,0,1}_{\alpha,\beta,\gamma,\delta}=&[\widetilde{\Gamma}^{1}_{\alpha,\beta}p-(\widetilde{\Gamma}^{1}_{\alpha,\beta}\wedge \Gamma^{1}_{\gamma,\delta})p]\wedge [\Pi^{0,1}_{\gamma,\delta}p-(\widetilde{\Pi}^{0,1}_{\alpha,\beta}\wedge \Pi^{0,1}_{\gamma,\delta})p]\\
H^{0,1,1,1}_{\alpha,\beta,\gamma,\delta}=&[\widetilde{\Pi}^{0,1}_{\alpha,\beta}p-(\widetilde{\Pi}^{0,1}_{\alpha,\beta}\wedge \Pi^{0,1}_{\gamma,\delta})p]\wedge [\Gamma^{1}_{\gamma,\delta}p-(\widetilde{\Gamma}^{1}_{\alpha,\beta}\wedge \Gamma^{1}_{\gamma,\delta})p]\quad \text{ if } \gamma>\alpha
\end{array}
\right\} \text{ if }\beta=\delta\\
\\
&\left.
\begin{array}{ll}
H^{-1,0,-1,0}_{\alpha,\beta,\gamma,\delta}=&
(\widetilde{\Pi}^{-1,0}_{\alpha,\beta}\wedge \Pi^{-1,0}_{\gamma,\delta})p\\
H^{-1,0,1,1}_{\alpha,\beta,\gamma,\delta}=&[\widetilde{\Pi}^{-1,0}_{\alpha,\beta}p-(\widetilde{\Pi}^{-1,0}_{\alpha,\beta}\wedge \Pi^{-1,0}_{\gamma,\delta})p]^{+}\wedge[\Gamma^{1}_{\gamma,\delta}p-(\widetilde{\Gamma}^{1}_{\alpha,\beta}\wedge \Gamma^{1}_{\gamma,\delta})p]\\
H^{1,1,-1,0}_{\alpha,\beta,\gamma,\delta}=&[\widetilde{\Gamma}^{1}_{\alpha,\beta}p-(\widetilde{\Gamma}^{1}_{\alpha,\beta}\wedge \Gamma^{1}_{\gamma,\delta})p]\wedge [\Pi^{-1,0}_{\gamma,\delta}p-(\widetilde{\Pi}^{-1,0}_{\alpha,\beta}\wedge \Pi^{-1,0}_{\gamma,\delta})p]\text{ if }\beta< \delta
\end{array}
\right\}
\text{ if }\alpha=\gamma;\\
\\
&H^{0,1,0,0}_{\alpha,\beta,\gamma,\delta}=\widetilde{\Pi}^{0,1}_{\alpha,\beta}p-H^{0,1,0,1}_{\alpha,\beta,\gamma,\delta}-H^{0,1,1,1}_{\alpha,\beta,\gamma,\delta}; \qquad
H^{1,1,0,0}_{\alpha,\beta,\gamma,\delta}=\widetilde{\Gamma}^{1}_{\alpha,\beta}p-H^{1,1,1,1}_{\alpha,\beta,\gamma,\delta}-H^{1,1,-1,0}_{\alpha,\beta,\gamma,\delta}-H^{1,1,0,1}_{\alpha,\beta,\gamma,\delta};
\\
&H^{0,0,0,1}_{\alpha,\beta,\gamma,\delta}=\Pi^{0,1}_{\gamma,\delta}p-H^{0,1,0,1}_{\alpha,\beta,\gamma,\delta}-H^{1,1,0,1}_{\alpha,\beta,\gamma,\delta}; \qquad
H^{0,0,1,1}_{\alpha,\beta,\gamma,\delta}=\Gamma^{1}_{\gamma,\delta}p-H^{1,1,1,1}_{\alpha,\beta,\gamma,\delta}-H^{0,1,1,1}_{\alpha,\beta,\gamma,\delta}-H^{-1,0,1,1}_{\alpha,\beta,\gamma,\delta};\\
&H^{-1,0,0,0}_{\alpha,\beta,\gamma,\delta}=\widetilde{\Pi}^{-1,0}_{\alpha,\beta}p-H^{-1,0,-1,0}_{\alpha,\beta,\gamma,\delta}-H^{-1,0,1,1}_{\alpha,\beta,\gamma,\delta}; \qquad 
H^{0,0,-1,0}_{\alpha,\beta,\gamma,\delta}=\Pi^{-1,0}_{\gamma,\delta}p-H^{-1,0,-1,0}_{\alpha,\beta,\gamma,\delta}-H^{1,1,-1,0}_{\alpha,\beta,\gamma,\delta}.
\end{align*}
\label{coupl_1}
\end{definition}
\begin{proposizione}
Definition \ref{coupl_1} gives an increasing coupling.
\end{proposizione}
\begin{proof}.
The definition of uncoupled terms ensures this is a coupling. We have to prove attractiveness.\\

Suppose  $\beta=\delta$ and $\gamma\geq\alpha$. Therefore $H^{-1,0,1,1}_{\alpha,\beta,\gamma,\beta}=H^{-1,0,-1,0}_{\alpha,\beta,\gamma,\beta}=H^{1,1,-1,0}_{\alpha,\beta,\gamma,\beta}=0$.\\
$\bullet$ Suppose $\widetilde{\Gamma}^{1}_{\alpha,\beta}\geq\Gamma^{1}_{\gamma,\delta}$. By Condition $(\ref{C+1})$ we must have $\widetilde{\Pi}^{0,1}_{\alpha,\beta}p \leq \Pi^{0,1}_{\gamma,\delta}p$, then $H^{0,1,0,0}_{\alpha,\beta,\gamma,\delta}=0$. Moreover
$$
\begin{array}{l}
H^{1,1,1,1}_{\alpha,\beta,\gamma,\delta}=\Gamma^{1}_{\gamma,\delta}p, \qquad H^{0,1,0,1}_{\alpha,\beta,\gamma,\delta}=\widetilde{\Pi}^{0,1}_{\alpha,\beta}p, \qquad H^{0,1,1,1}_{\alpha,\beta,\gamma,\delta}=0 \\
\\
H^{1,1,0,1}_{\alpha,\beta,\gamma,\delta}=
[(\widetilde{\Gamma}^{1}_{\alpha,\beta}-\Gamma^{1}_{\gamma,\delta})p]\wedge [(\Pi^{0,1}_{\gamma,\delta}-\widetilde{\Pi}^{0,1}_{\alpha,\beta})p]=(\widetilde{\Gamma}^{1}_{\alpha,\beta}-\Gamma^{1}_{\gamma,\delta})p
\end{array}
$$
by Condition ($\ref{C+1}$). Therefore $H^{1,1,0,0}_{\alpha,\beta,\gamma,\delta}=\widetilde{\Gamma}^{1}_{\alpha,\beta}p-\Gamma^{1}_{\gamma,\delta}p-\widetilde{\Gamma}^{1}_{\alpha,\beta}p+\Gamma^{1}_{\gamma,\delta}p=0.$\\

\noindent $\bullet$ Suppose $\widetilde{\Gamma}^{1}_{\alpha,\beta}<\Gamma^{1}_{\gamma,\delta}$. 
Then $H^{1,1,1,1}_{\alpha,\beta,\gamma,\delta}=\widetilde{\Gamma}^{1}_{\alpha,\beta}p$ and we get $H^{1,1,0,1}_{\alpha,\beta,\gamma,\delta}=H^{1,1,0,0}_{\alpha,\beta,\gamma,\delta}=0$. We have to prove that $H^{0,1,0,0}_{\alpha,\beta,\gamma,\delta}=0$.\\
If $\gamma=\alpha$, by Condition (\ref{C+10}) then $
H^{0,1,0,0}_{\alpha,\beta,\gamma,\delta}=\Pi^{0,1}_{\alpha,\beta}p-\Pi^{0,1}_{\alpha,\beta}p\wedge \Pi^{0,1}_{\gamma,\delta}p=0$ and we are done.\\
Suppose $\gamma>\alpha$:
\begin{equation*}
H^{0,1,0,0}_{\alpha,\beta,\gamma,\delta}=\widetilde{\Pi}^{0,1}_{\alpha,\beta}p-\widetilde{\Pi}^{0,1}_{\alpha,\beta}p\wedge\Pi^{0,1}_{\gamma,\delta}p-[\widetilde{\Pi}^{0,1}_{\alpha,\beta}p-\widetilde{\Pi}^{0,1}_{\alpha,\beta}p\wedge\Pi^{0,1}_{\gamma,\delta}p]\wedge [\Gamma^{1}_{\gamma,\delta}p-\widetilde{\Gamma}^{1}_{\alpha,\beta}p]
\end{equation*}
By using the relation
$$
a\wedge(c-c \wedge b)=(a+b)\wedge c-b\wedge c \qquad \text{ for } a,b,c \geq 0
$$
with $a=\Gamma^{1}_{\gamma,\delta}p-\widetilde{\Gamma}^{1}_{\alpha,\beta}p$, $b=\Pi^{0,1}_{\gamma,\delta}p $ and $c=\widetilde{\Pi}^{0,1}_{\alpha,\beta}p$, we get by Condition (\ref{C+1}),
\begin{align*}
H^{0,1,0,0}_{\alpha,\beta,\gamma,\delta}=&\widetilde{\Pi}^{0,1}_{\alpha,\beta}p-\widetilde{\Pi}^{0,1}_{\alpha,\beta}p\wedge\Pi^{0,1}_{\gamma,\delta}p-[\Gamma^{1}_{\gamma,\delta}p-\widetilde{\Gamma}^{1}_{\alpha,\beta}p+\Pi^{0,1}_{\gamma,\delta}p]\wedge \widetilde{\Pi}^{0,1}_{\alpha,\beta}p+\widetilde{\Pi}^{0,1}_{\alpha,\beta}p \wedge \Pi^{0,1}_{\gamma,\delta}p\\
=&\widetilde{\Pi}^{0,1}_{\alpha,\beta}p-\widetilde{\Pi}^{0,1}_{\alpha,\beta}p=0.
\end{align*}

If $\alpha=\gamma$  and $\delta\geq \beta$ we check that the coupling is increasing in the same way.
\end{proof}

\end{document}